\documentclass[11 pt]{amsart}
\usepackage{amsmath, amsthm,amssymb,bbm,enumerate,mathrsfs,mathtools}
\usepackage{a4wide}
\usepackage{tikz,xcolor,verbatim}
\usetikzlibrary{calc,shapes}
\usepackage{hyperref}
\usepackage{makecell}
\usepackage[numbers,sort&compress]{natbib}

\usetikzlibrary{shapes.geometric}

\tikzset{
    triangle/.style={
        draw,
        shape border rotate=0,
        regular polygon,
        regular polygon sides=3,
        fill=white,
        node distance=2cm,
        minimum height=4em
    }
}
\usepackage{amssymb}
\usepackage{dirtytalk}
\usepackage{algorithmicx}
\usepackage{algpseudocode}
\usepackage{algorithm}
\usepackage{a4wide}
\algblock{Input}{EndInput}
\algnotext{EndInput}
\algblock{Output}{EndOutput}
\algnotext{EndOutput}
\algblock{Initialization}{EndInitialization}
\algnotext{EndInitialization}
\algblock{Begin}{EndBegin}
\algnotext{EndBegin}
\def\aftermath{\par\vspace{-\belowdisplayskip}\vspace{-\parskip}\vspace{-\baselineskip}}

 \usepackage[yyyymmdd,hhmmss]{datetime}
 \usepackage{background}
 \SetBgContents{ver. \currenttime \qquad \today }
 \SetBgScale{1}
 \SetBgAngle{0}
 \SetBgOpacity{1}
 \SetBgPosition{current page.south west} 
 \SetBgHshift{3cm}
 \SetBgVshift{0.5cm} 

\tikzset{
  bigblue/.style={circle, draw=blue!80,fill=blue!40,thick, inner sep=1.5pt, minimum size=5mm},
  bigred/.style={circle, draw=red!80,fill=red!40,thick, inner sep=1.5pt, minimum size=5mm},
  bigblack/.style={circle, draw=black!100,fill=black!40,thick, inner sep=1.5pt, minimum size=5mm},
  bluevertex/.style={circle, draw=blue!100,fill=blue!100,thick, inner sep=0pt, minimum size=2mm},
  redvertex/.style={circle, draw=red!100,fill=red!100,thick, inner sep=0pt, minimum size=2mm},
  blackvertex/.style={circle, draw=black!100,fill=black!100,thick, inner sep=0pt, minimum size=1.5mm},  
  whitevertex/.style={circle, draw=black!100,fill=white!100,thick, inner sep=0pt, minimum size=2mm},  
  smallblack/.style={circle, draw=black!100,fill=black!100,thick, inner sep=0pt, minimum size=1mm},
  smallwhite/.style={circle, draw=white!100,fill=white!100,thick, inner sep=0pt, minimum size=1mm},
  redvertexv2/.style={circle, draw=black!100,fill=red!100,thick, inner sep=0pt, minimum size=3.35mm},
  bluevertexv2/.style={circle, draw=black!100,fill=blue!100,thick, inner sep=0pt, minimum size=3.35mm},
  greenvertexv2/.style={circle, draw=black!100,fill=green!100,thick, inner sep=0pt, minimum size=3.35mm},
  blackvertexv2/.style={circle, draw=black!100,fill=black!100,thick, inner sep=0pt, minimum size= 3.35mm}
}

\makeatletter
\pgfdeclareshape{myNode}{
  \inheritsavedanchors[from=rectangle] 
  \inheritanchorborder[from=rectangle]
  \inheritanchor[from=rectangle]{center}
  \inheritanchor[from=rectangle]{north}
  \inheritanchor[from=rectangle]{south}
  \inheritanchor[from=rectangle]{west}
  \inheritanchor[from=rectangle]{east}
  \backgroundpath{
    \southwest \pgf@xa=\pgf@x \pgf@ya=\pgf@y
    \northeast \pgf@xb=\pgf@x \pgf@yb=\pgf@y
    \pgfsetcornersarced{\pgfpoint{5pt}{5pt}}
    \pgfpathmoveto{\pgfpoint{\pgf@xa}{\pgf@ya}}
    \pgfpathlineto{\pgfpoint{\pgf@xa}{\pgf@yb}}
    \pgfpathlineto{\pgfpoint{\pgf@xb}{\pgf@yb}}
    \pgfsetcornersarced{\pgfpoint{5pt}{5pt}}
    \pgfpathlineto{\pgfpoint{\pgf@xb}{\pgf@ya}}
    \pgfpathclose
 }
}
\makeatother

\usepackage{xcolor}

\title[A density bound for triangle-free $4$-critical graphs]{A density bound for triangle-free $4$-critical graphs}

\author[Moore]{Benjamin Moore}
\thanks{This work is part of the first author's PhD thesis, done at the University of Waterloo. \\ 
\hphantom{|} We acknowledge the support of the Natural Sciences and Engineering Research Council of Canada (NSERC),  [CGSD2 Grant No. 534801-2019], [CGSD3 Grant No. 2020-547516] \\
\hphantom{|} Cette recherche a \'{e}t\'{e} financ\'{e}e par le Conseil de recherches en sciences naturelles et en g\'{e}nie du Canada (CRSNG), [CGSD2 Grant No. 534801-2019], [CGSD3 Grant No. 2020-547516]}
\address[Benjamin Moore]{Institute of Computer Science, Charles University, Prague, Czech Republic }  
\email{brmoore@iuuk.mff.cuni.cz}

\author[Smith-Roberge]{Evelyne Smith-Roberge}
\address[Evelyne Smith-Roberge]{Department of Combinatorics and Optimization, University of Waterloo, Waterloo, ON, Canada}
\email{evelyne.smith-roberge@uwaterloo.ca}

\date{}

\newtheorem{thm}[equation]{Theorem}
\newtheorem{lemma}[equation]{Lemma}
\newtheorem{prop}[equation]{Proposition}
\newtheorem{conj}[equation]{Conjecture}
\newtheorem{cor}[equation]{Corollary}
\newtheorem{claim}{Claim}
\newtheorem{question}[equation]{Question}

\theoremstyle{definition}
\newtheorem{definition}[equation]{Definition}
\newtheorem{obs}[equation]{Observation}

\newtheorem*{ack}{Acknowledgements}

\newtheoremstyle{case}{}{}{\normalfont}{}{\itshape}{\normalfont:}{ }{}

\theoremstyle{case}

\numberwithin{equation}{section}

\newcommand{\ch}{\textnormal{ch}}

\date{}

\begin{document}
\begin{abstract}
We prove that every triangle-free $4$-critical graph $G$ satisfies $e(G) \geq \frac{5v(G)+2}{3}$. This result gives a unified proof that triangle-free planar graphs are $3$-colourable, as well as that graphs of girth at least five which embed in either the projective plane, torus, or Klein Bottle are $3$-colourable, which are results of  Gr\"{o}tzsch, Thomassen, and Thomas and Walls. Our result is nearly best possible, as Davies has constructed triangle-free $4$-critical graphs $G$ such that $e(G) = \frac{5v(G) + 4}{3}$. To prove this result, we  prove a more general result characterizing sparse $4$-critical graphs with few vertex-disjoint triangles.
\end{abstract}
\maketitle
\section{Introduction}
A \textit{$k$-colouring} of a graph $G$ is a map $f:V(G) \to \{1,\ldots,k\}$ such that for every edge $e = xy$, we have $f(x) \neq f(y)$. We say $G$ is $k$-colourable if there exists a $k$-colouring of $G$. The \textit{chromatic number} of $G$, denoted $\chi(G)$, is the minimum number $k$ such that $G$ has a $k$-colouring.

In general, determining whether a graph is $k$-colourable for any fixed $k \geq 3$ is a well-known NP-complete problem \cite{threecolouringhard}. As such, a great deal of research has focused on restricting this problem to specific classes of graphs. The most famous such result is the four colour theorem, which states that planar graphs are $4$-colourable.
The four colour theorem is notoriously difficult, and to date there are no proofs that do not require computer assistance. This difficulty surprisingly vanishes if we add the condition that the planar graphs are triangle-free: that is, contain no $K_{3}$ subgraph.

\begin{thm}[Gr\"{o}tzsch's Theorem, \cite{grotzsch1959}]\label{grotzsch}
Every triangle-free planar graph is $3$-colourable. 
\end{thm}

The original proof of this theorem is by no means easy, but does not require the use of computers. At first glance, it does not look particularly easy to generalize Gr\"{o}tzsch's Theorem: for instance, there are planar graphs which require four colours (for example, $K_{4}$), and non-bipartite triangle-free planar graphs (for example, any odd cycle which is not $K_{3}$). Worse, we cannot even use a weaker surface embeddability condition: for any orientable surface aside from the plane, the Gr\"{o}tzsch graph \textemdash a triangle-free graph with chromatic number four \textemdash embeds in the surface. See Figure \ref{Grotzschgraph} for an illustration of the Gr\"{o}tzsch graph. Further, building on the work of Younger \cite{Youngs}, Gimbel and Thomassen \cite{Gimbel1997ColoringGW} characterized the triangle-free projective planar graphs which are $3$-colourable. In particular, all non-bipartite quadrangulations of the projective plane require four colours. Nevertheless, by strengthening the triangle-free condition, more progress can be made on other surfaces.
Recall that the \textit{girth} of a graph is the length of a shortest cycle in the graph (and if the graph is cycle-free, we define its girth as being infinite). 
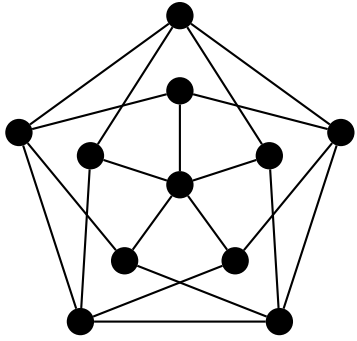
\begin{figure}
\begin{tikzpicture}
\foreach \i in {0,1,2,3,4}{
\node[blackvertexv2] at (90+\i*72:1.25cm) (u\i) {};
\node[blackvertexv2] at (90+\i*72:2.25cm) (v\i) {};
}
\node[blackvertexv2] at (0,0) (center) {};

\draw[thick,black] (v0)--(v1)--(v2)--(v3)--(v4)--(v0);
\draw[thick,black] (center)--(u0);
\draw[thick,black] (center)--(u1);
\draw[thick,black] (center)--(u2);
\draw[thick,black] (center)--(u3);
\draw[thick,black] (center)--(u4);
\draw[thick,black] (v0)--(u1);
\draw[thick,black] (v0)--(u4);
\draw[thick,black] (v1)--(u2);
\draw[thick,black] (v1)--(u0);
\draw[thick,black] (v2)--(u1);
\draw[thick,black] (v2)--(u3);
\draw[thick,black] (v3)--(u2);
\draw[thick,black] (v3)--(u4);
\draw[thick,black] (v4)--(u3);
\draw[thick,black] (v4)--(u0);
\end{tikzpicture}
\caption{The Gr\"{o}tzsch graph. It has chromatic number four, is triangle-free, and embeds in the torus.}
\label{Grotzschgraph}
\end{figure}

\begin{thm}[Thomassen, \cite{Thomassen}]
\label{Carstengirth5}
Every graph of girth at least five embeddable on the torus or projective plane is $3$-colourable.
\end{thm}

The non-orientable part of the above theorem can be strengthened:

\begin{thm}[Thomas \& Walls, \cite{thomaswalls}]
\label{ThomasWallsgirth5}
Every graph of girth at least five embeddable in the Klein Bottle is $3$-colourable.
\end{thm}

These theorems show that every graph with girth at least five embeddable in a surface of Euler genus zero is $3$-colourable. In addition, in the case of the plane, \emph{girth at least five} can be strengthened to \emph{triangle-free} (see Theorem \ref{grotzsch}). It seems reasonable to ask for a unified proof of these results. One standard way to generalize and unify results on surfaces is to observe that graphs embedded in surfaces have relatively few edges: in particular, they have bounded \textit{maximum average degree}. Recall that the maximum average degree of a graph is the maximum of the average degrees over all subgraphs. Before proceeding, we set the following convenient notation: we use $e(G)$ to denote $|E(G)|$, and $v(G)$ to denote $|V(G)|$.  From Euler's formula we get that every graph embeddable on the torus or Klein Bottle with no face of degree at most four has maximum average degree at most $\frac{10}{3}$: to see this, note that if such an embedded graph $G$ has no face of degree at most four, then $2e(G) \geq 5f(G)$, where $f(G)$ is the number of faces. Substituting this into Euler's formula and using the fact that the Euler characteristic of such surfaces is at most $0$, we get that $e(G) \leq \frac{5}{3}v(G)$. 

This suggests the following question: does every graph with maximum average degree at most $\frac{10}{3}$ admit a $3$-colouring? Unfortunately here the answer is no: for instance, $K_{4}$ has average degree $3$, but chromatic number $4$. Nevertheless, this type of question leads to a very important theorem, due to Kostochka and Yancey. To state their result, we need a definition: a graph $G$ is \textit{$k$-critical} if $G$ has chromatic number $k$, but all proper subgraphs of $G$ are $(k-1)$-colourable. 

\begin{thm}[\cite{oresconjecture},\cite{Shortproof}]
\label{KostochkaYancey4}
Every $4$-critical graph $G$ satisfies
\[e(G) \geq \frac{5v(G) -2}{3}.\]
\end{thm}

Thus $4$-critical graphs either have average degree at least $\frac{10}{3}$, or fall \emph{just} short of this bound. This theorem, combined with an easy reduction, was used to give an exceptionally short proof of Gr\"{o}tzsch's Theorem \cite{Shortproof}. Thus one may hope that it could be used to prove Theorems \ref{Carstengirth5} and \ref{ThomasWallsgirth5}. However without some additional work, the bound is not quite strong enough. Naturally one might try to improve the bound in Theorem \ref{KostochkaYancey4}, but unfortunately the bound is known to be tight infinitely often \cite{tightnessore}, and hence no general improvement can be made (See Figure \ref{tightexamples} for two graphs for which the bound is tight).

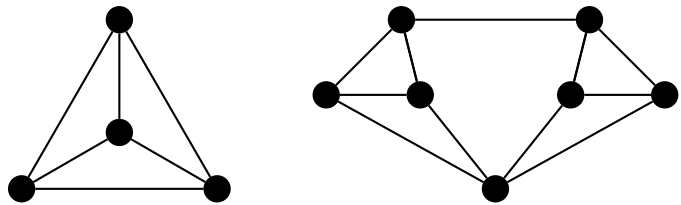
\begin{figure}
    \centering
\begin{tikzpicture}
\foreach \i in {0,1,2}{
\node[blackvertexv2] at (90+\i*120:1.5cm) (u\i) {};
}
\node[blackvertexv2] at (0,0) (u3) {};
\draw[thick,black] (u0)--(u1)--(u2)--(u3)--(u0);
\draw[thick,black] (u0)--(u2);
\draw[thick,black] (u1)--(u3);

\begin{scope}[xshift =5cm]
\node[blackvertexv2] at (0,-.75) (v1) {};
\node[blackvertexv2] at (-2.25,.5) (v2) {};
\node[blackvertexv2] at (-1,.5) (v3) {};
\node[blackvertexv2] at (-1.25,1.5) (v4) {};
\draw[thick,black] (v1)--(v2)--(v3)--(v4)--(v2);
\draw[thick,black] (v4)--(v3);
\draw[thick,black] (v1)--(v3);

\node[blackvertexv2] at (2.25,.5) (v5) {};
\node[blackvertexv2] at (1,.5) (v6) {};
\node[blackvertexv2] at (1.25,1.5) (v7) {};
\draw[thick,black] (v1)--(v5)--(v6)--(v7)--(v5);
\draw[thick,black] (v1)--(v6);
\draw[thick,black] (v6)--(v7);
\draw[thick,black] (v7)--(v4);
\end{scope}
\end{tikzpicture}
    \caption{Two tight examples for Theorem \ref{KostochkaYancey4}. The left graph is $K_{4}$, and the right graph is the Moser Spindle.}
    \label{tightexamples}
\end{figure}

Remarkably, Kostochka and Yancey characterized precisely the family of graphs for which Theorem \ref{KostochkaYancey4} is tight. We defer an explicit definition of this family until later; for now, all that is important is that all tight examples for Theorem \ref{KostochkaYancey4} contain triangles. As Theorems \ref{Carstengirth5} and \ref{ThomasWallsgirth5} concern graphs of girth at least $5$, one might wonder if the bound in Theorem \ref{KostochkaYancey4} could be strengthened under the additional assumption that the $4$-critical graphs have girth at least five; one could then deduce Theorems \ref{Carstengirth5} and \ref{ThomasWallsgirth5}. Liu and Postle showed that this is indeed the case.

\begin{thm}[\cite{4criticalgirth5}]
Every $4$-critical graph $G$ with girth at least five satisfies
\[e(G) \geq \frac{5v(G) + 2}{3}.\]
\end{thm}

This theorem implies Theorems \ref{Carstengirth5} and \ref{ThomasWallsgirth5} simultaneously, giving a unified proof of both theorems. Note however that this theorem does not imply Gr\"{o}tzsch's Theorem naively (as it says nothing about graphs with $4$-cycles), unlike Theorem \ref{KostochkaYancey4}. This begs a natural question: can we replace the girth at least five condition in the theorem with the weaker condition of being triangle-free to have a unified proof of all three surface results using Euler's formula? We answer this in the affirmative.

\begin{thm}
\label{maintheoremintro}
Every triangle-free $4$-critical graph $G$ satisfies
\[e(G) \geq \frac{5v(G)+2}{3}.\]
\end{thm}

Our theorem uses Theorem \ref{KostochkaYancey4}, and hence does not give a new proof of Gr\"{o}tszch's Theorem. It is also significantly longer. Note this theorem generalizes both Theorem \ref{Carstengirth5} and Theorem \ref{ThomasWallsgirth5}, in the sense that a graph which is triangle-free and embeds in the torus or Klein Bottle is $3$-colourable, so long as the edge density is not too large. However, one should note that both Theorem \ref{Carstengirth5} and Theorem \ref{ThomasWallsgirth5} can be strengthened to include $4$-cycles which are contractible, but our theorem generalizes even these stronger statements.

With this, we pause briefly to discuss how tight our result is. The Thomas-Walls construction (\cite{thomaswalls}) shows that the leading term $\frac{5}{3}$ cannot be improved for triangle-free $4$-critical graphs. The idea behind the construction is as follows. Begin by constructing an infinite family of $4$-critical graphs with two key properties: the first being that the graphs satisfy $e(G) = \frac{5v(G)-2}{3}$; the second, that there are two edges $e_1, e_2$ with distinct endpoints such that the deletion of these edges results in a triangle-free graph. The smallest three such graphs are shown in Figure \ref{thomaswallsdrawings}. To finish, for $i \in \{1,2\}$, perform an Ore composition (defined later in the introduction) with $e_i$ and suitable triangle-free $4$-critical graphs with few vertices (for instance, either of the graphs in Figures \ref{Grotzschgraph} and \ref{jamesgraph}). The resulting graph is $4$-critical, since Ore compositions preserve $4$-criticality under certain situations; it is triangle-free, since $e_1$ and $e_2$ were deleted; and finally its edge density is roughly $\frac{5}{3}$, since for a large Thomas-Walls graph, adding in two small graphs such as the ones in Figure \ref{Grotzschgraph} and Figure \ref{jamesgraph} adds relatively few edges.  
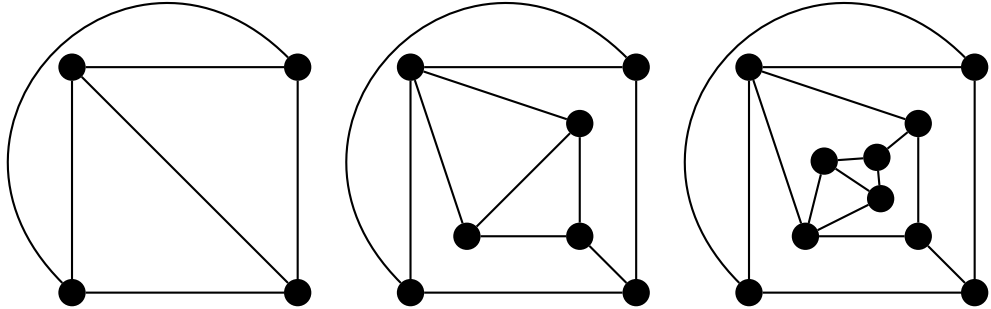
\begin{figure}
    \centering
    \begin{tikzpicture}
    \node[blackvertexv2] at (0,0) (v1) {};
    \node[blackvertexv2] at (0,3) (v2) {};
    \node[blackvertexv2] at (3,3) (v3) {};
    \node[blackvertexv2] at (3,0) (v4) {};
    \draw[thick,black] (v1)--(v2)--(v3)--(v4)--(v1);
    \draw[thick,black] (v2)--(v4);
    \draw[thick,black, bend left = 90, min distance =3cm] (v1) to (v3);
    \begin{scope}[xshift = 4.5cm]
    \node[blackvertexv2] at (0,0) (v1) {};
    \node[blackvertexv2] at (0,3) (v2) {};
    \node[blackvertexv2] at (3,3) (v3) {};
    \node[blackvertexv2] at (3,0) (v4) {};
    \node[blackvertexv2] at (.75,.75) (v6) {};
    \node[blackvertexv2] at (2.25,.75) (v7) {};
    \node[blackvertexv2] at (2.25,2.25) (v8) {};
    \draw[thick,black] (v1)--(v2)--(v3)--(v4)--(v1);
    \draw[thick,black, bend left = 90, min distance =3cm] (v1) to (v3);
    \draw[thick,black] (v2)--(v6)--(v7)--(v8)--(v2);
    \draw[thick,black] (v6)--(v8);
    \draw[thick,black] (v7)--(v4);
    \end{scope}
     \begin{scope}[xshift = 9cm]
    \node[blackvertexv2] at (0,0) (v1) {};
    \node[blackvertexv2] at (0,3) (v2) {};
    \node[blackvertexv2] at (3,3) (v3) {};
    \node[blackvertexv2] at (3,0) (v4) {};
    \node[blackvertexv2] at (.75,.75) (v6) {};
    \node[blackvertexv2] at (2.25,.75) (v7) {};
    \node[blackvertexv2] at (2.25,2.25) (v8) {};
    \node[blackvertexv2] at (1.7,1.8) (v9) {};
    \node[blackvertexv2] at (1,1.75) (v10) {};
    \node[blackvertexv2] at (1.75, 1.25) (v11) {};
    \draw[thick,black] (v1)--(v2)--(v3)--(v4)--(v1);
    \draw[thick,black, bend left = 90, min distance =3cm] (v1) to (v3);
    \draw[thick,black] (v2)--(v6)--(v7)--(v8)--(v2);
    \draw[thick,black] (v7)--(v4);
    \draw[thick,black] (v6)--(v10)--(v9)--(v11)--(v6);
    \draw[thick,black] (v10)--(v11);
    \draw[thick,black] (v9)--(v8);
    \end{scope}
    \end{tikzpicture}
    \caption{The first three Thomas-Walls graphs}
    \label{thomaswallsdrawings}
\end{figure}

Hence the only question then is whether or not the additive term in the bound from Theorem \ref{maintheoremintro} can be improved further. Liu and Postle conjectured the following.

\begin{conj}[\cite{4criticalgirth5}]
Every triangle-free $4$-critical graph $G$ satisfies
\[e(G) \geq \frac{5v(G) + 5}{3}.\]
\end{conj}

Unfortunately, this is false. An unpublished result of James Davies shows it is false infinitely often.

\begin{thm}
\label{jamestheorem}
There exists infinitely many $4$-critical triangle-free graphs $G$ such that
\[e(G) = \frac{5v(G) +4}{3}.\]
\end{thm}

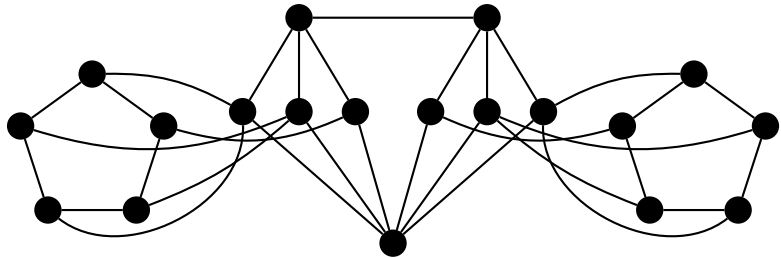
\begin{figure}
    \centering
    \begin{tikzpicture}
\begin{scope}[xshift = -2cm, yshift =-.75cm]
\foreach \i in {0,1,2,3,4}{
\node[blackvertexv2] at (90+\i*72:1cm) (u\i) {};
}
\draw[thick,black] (u0)--(u1)--(u2)--(u3)--(u4)--(u0);
\end{scope}
\node[blackvertexv2] at (2,-2) (v1) {};
\node[blackvertexv2] at (0,-.25) (v2) {};
\node[blackvertexv2] at (.75,-.25) (v3) {};
\node[blackvertexv2] at (1.5,-.25) (v4) {};
\node[blackvertexv2] at (.75, 1) (v5) {};
\draw[thick,black] (v1)--(v2);
\draw[thick,black] (v1)--(v3);
\draw[thick,black] (v1)--(v4);
\draw[thick,black] (v5)--(v2);
\draw[thick,black] (v5)--(v3);
\draw[thick,black] (v5)--(v4);
\draw[thick,black,bend left = 15] (u0) to (v2);
\draw[thick,black,bend right = 20] (u1) to (v3);
\draw[thick,black,bend right = 65] (u2) to (v2);
\draw[thick,black,bend right =10] (u3) to (v3);
\draw[thick,black,bend right = 20] (u4) to (v4);

\node[blackvertexv2] at (4,-.25) (x2) {};
\node[blackvertexv2] at (3.25,-.25) (x3) {};
\node[blackvertexv2] at (2.5,-.25) (x4) {};
\node[blackvertexv2] at (3.25, 1) (x5) {};
\draw[thick,black] (v1)--(x2);
\draw[thick,black] (v1)--(x3);
\draw[thick,black] (v1)--(x4);
\draw[thick,black] (x5)--(x2);
\draw[thick,black] (x5)--(x3);
\draw[thick,black] (x5)--(x4);
\draw[thick,black] (x5)--(v5);

\begin{scope}[xshift = 6cm, yshift =-.75cm]
\foreach \i in {0,1,2,3,4}{
\node[blackvertexv2] at (90+\i*72:1cm) (w\i) {};
}
\draw[thick,black] (w0)--(w1)--(w2)--(w3)--(w4)--(w0);
\end{scope}

\draw[thick,black,bend right = 15] (w0) to (x2);
\draw[thick,black,bend left = 20] (w1) to (x4);
\draw[thick,black,bend left = 10] (w2) to (x3);
\draw[thick,black,bend left =65] (w3) to (x2);
\draw[thick,black,bend left = 20] (w4) to (x3);
\end{tikzpicture}
    \caption{The smallest triangle-free $4$-critical graph from the construction in Theorem \ref{jamestheorem}.}
    \label{jamesgraph}
\end{figure}

To the best of the authors' knowledge, Davies' construction gives rise to the sparsest-known triangle-free $4$-critical graphs. 

This leaves open the following question.

\begin{question}\label{whatisc}
What is the largest $c \in \{2,3,4\}$ such that all $4$-critical triangle-free graphs satisfy 
\[e(G) \geq \frac{5v(G) + c}{3}?\]
\end{question}

 We now describe how we prove Theorem \ref{maintheoremintro}, from which the lower-bound $c \geq 2$ in Question \ref{whatisc} is obtained. We use the potential method, developed by Kostochka and Yancey and used in many recent papers: for example, \cite{nearbipartite, postletrianglefree5crit, tightnessore, Shortproof}. A key component of the potential method is a \emph{potential function}, which is a function of the number of vertices and edges in a graph. The method also involves a certain quotient operation which might create triangles. To deal with this, we incorporate (vertex-disjoint) triangles into our potential function: given a graph $G$, we let $T^{3}(G)$ denote the maximum number of vertex-disjoint triangles in $G$.  Our potential function is defined below.
 
 \begin{definition}
 Given a graph $G$, the \textit{potential} of $G$ is defined as $p(G) := 5v(G) -3e(G) - T^{3}(G)$.
 \end{definition} 
 
 For intuition about this function, consider the case where $G$ is triangle-free. In this case, $T^{3}(G) =0$. Now if $p(G) <0$, then $G$ has average degree smaller than $\frac{10}{3}$, and if $p(G) \geq 0$ then $G$ has average degree at most $\frac{10}{3}$. Thus the potential of $G$ is a measure of how close the average degree of $G$ is to $\frac{10}{3}$, offset by the number of vertex-disjoint triangles. 

Our theorem characterizes the $4$-critical graphs whose potential is at most $-1$. To do this, we need to know the graphs which attain the bound in Theorem \ref{KostochkaYancey4}. This requires a definition.

\begin{definition}
An \textit{Ore Composition} of two graphs $H_{1}$ and $H_{2}$ is the graph $H$ obtained by deleting an edge $xy \in E(H_{1})$, splitting a vertex $z \in V(H_2)$ into two vertices $z_1$ and $z_2$ of positive degree such that $N(z) = N(z_{1}) \cup N(z_{2})$ and $N(z_{1}) \cap N(z_{2}) = \emptyset$, and then identifying $x$ with $z_{1}$ and $y$ with $z_{2}$. Here $N(v)$ refers to the \emph{neighbourhood} of $v$: the set of vertices adjacent to $v$.  We say that $H_{1}$ is the \textit{edge side of the composition} and $H_{2}$ is the \textit{split side of the composition}, and we denote the graph obtained from $H_2$ by splitting $z$ as $H_{2}^{z}$. A graph $G$ is \textit{$4$-Ore} if $G$ is obtained from Ore compositions of $K_{4}$.
\end{definition}

The following wonderful theorem says that every tight example to Theorem \ref{KostochkaYancey4} is in fact $4$-Ore. 

\begin{thm}[\cite{tightnessore}]
\label{KostochkaYanceytight}
A $4$-critical graph $G$ has 
\[e(G) = \frac{5v(G)-2}{3}\] 
if and only if $G$ is $4$-Ore.
\end{thm}

We will also need some more graph classes. We denote by $W_{n}$ the \textit{wheel on $n+1$ vertices}, which is the graph obtained from a cycle on $n$ vertices by adding a vertex adjacent to all vertices in the cycle. A wheel is \textit{odd} if $n$ is odd. Let $T_{8}$ be the graph with vertex set $V(T_{8}) = \{u_{1},u_{2},u_{3},u_{4},u_{5},u_{6},u_{7},u_{8}\}$ and $E(T_{8}) = \{u_{1}u_{2}, u_{1}u_{3}, u_{1}u_{4}, u_{1}u_{5}, u_{2}u_{3}, u_{2}u_{4}, u_{2}u_{5}, u_{3}u_{8}, u_{4}u_{7}, \\ u_{5}u_{6}, u_{6}u_{7}, u_{6}u_{8}, u_{7}u_{8}\}$. See Figure \ref{T8pic2} for an illustration. Let $\mathcal{B}$ be defined as follows: the graph $T_{8}$ is in  $\mathcal{B}$, and given a graph $G \in \mathcal{B}$ and a $4$-Ore graph $H$, the Ore composition $G'$ of $G$ and $H$ is in $\mathcal{B}$ whenever $T^{3}(G') =2$. 

\begin{figure}
\label{T8pic}
\begin{center}
\begin{tikzpicture}
\node[blackvertexv2] at (.5,3) (u1) {};
\node[smallwhite] at (0,3) (dummy1) {$u_{1}$};
\node[blackvertexv2] at (1.5,3) (u2) {};
\node[smallwhite] at (1.95,3) (dummy2) {$u_{2}$};
\node[blackvertexv2] at (0,2) (u3) {};
\node[smallwhite] at (-.5,2) (dummy3) {$u_{3}$};
\node[blackvertexv2] at (1,2) (u4) {};
\node[smallwhite] at (.55,2) (dummy4) {$u_{4}$};
\node[blackvertexv2] at (2,2) (u5) {};
\node[smallwhite] at (1.55,2) (dummy5) {$u_{5}$};
\node[blackvertexv2] at (0,0) (u6) {};
\node[smallwhite] at (-.45,0) (dummy6) {$u_{6}$};
\node[blackvertexv2] at (1,1) (u7) {};
\node[smallwhite] at (.55,1) (dummy7) {$u_{7}$};
\node[blackvertexv2] at (2,0) (u8) {};
\node[smallwhite] at (2.55,0) (dummy8) {$u_{8}$};

\draw[thick,black] (u1)--(u2);
\draw[thick,black] (u1)--(u3);
\draw[thick,black] (u1)--(u4);
\draw[thick,black] (u1)--(u5);
\draw[thick,black] (u2)--(u3);
\draw[thick,black] (u2)--(u4);
\draw[thick,black] (u2)--(u5);
\draw[thick,black] (u3)--(u6);
\draw[thick,black] (u4)--(u7);
\draw[thick,black] (u5)--(u8);
\draw[thick,black] (u6)--(u7);
\draw[thick,black] (u6)--(u8);
\draw[thick,black] (u7)--(u8);

\begin{scope}[xshift =5cm]
\node[blackvertexv2] at (.5,3) (u1) {};
\node[blackvertexv2] at (1.5,3) (u2) {};
\node[blackvertexv2] at (0,2) (u3) {};
\node[blackvertexv2] at (1,2) (u4) {};
\node[blackvertexv2] at (2,2) (u5) {};
\node[blackvertexv2] at (0,0) (u6) {};
\node[blackvertexv2] at (1,1) (u7) {};
\node[blackvertexv2] at (2,0) (u8) {};
\node[blackvertexv2] at (1.25,.3) (u9){};
\node[blackvertexv2] at (1.25,-.3) (u10) {};
\node[blackvertexv2] at (.75,0) (u11) {};

\draw[thick,black] (u1)--(u2);
\draw[thick,black] (u1)--(u3);
\draw[thick,black] (u1)--(u4);
\draw[thick,black] (u1)--(u5);
\draw[thick,black] (u2)--(u3);
\draw[thick,black] (u2)--(u4);
\draw[thick,black] (u2)--(u5);
\draw[thick,black] (u3)--(u6);
\draw[thick,black] (u4)--(u7);
\draw[thick,black] (u5)--(u8);
\draw[thick,black] (u6)--(u7);
\draw[thick,black] (u6)--(u11);
\draw[thick,black] (u7)--(u8);
\draw[thick,black] (u8)--(u9);
\draw[thick,black] (u8)--(u10);
\draw[thick,black] (u9)--(u10);
\draw[thick,black] (u9)--(u11);
\draw[thick,black] (u10)--(u11);

\end{scope}
\end{tikzpicture}
\end{center}
\caption{
The graph on the left is $T_{8}$, and the graph on the right is an example of a graph in $\mathcal{B}$.
}
\label{T8pic2}
\end{figure}
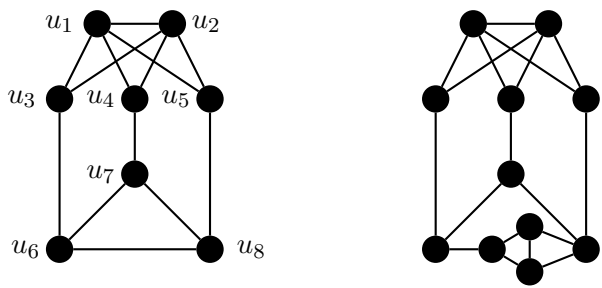

We can now state our main theorem.

\begin{thm}
\label{maintheorem}
If $G$ is a $4$-critical graph, then
\begin{itemize}
\item{$p(G) = 1$ if $G = K_4$, }
\item{$p(G) = 0$ if $T^{3}(G) = 2$ and $G$ is $4$-Ore,}
\item{$p(G) = -1$ if $G = W_{5}$, or $G \in \mathcal{B}$, or $G$ is $4$-Ore with $T^{3}(G) =3$, and}
\item{$p(G) \leq -2$ otherwise.}
\end{itemize}
\end{thm}

We note that Theorem \ref{maintheorem} implies Theorem \ref{maintheoremintro}. To see this, observe that the graphs contained in the first three bullet points all contain triangles. Hence if $G$ is triangle-free and $4$-critical, then its potential is at most $-2$. Thus we have $5v(G) - 3e(G) \leq -2$, and now simply solving for $e(G)$ gives the desired result. 

We now give a brief outline of the proof of Theorem \ref{maintheorem}. We follow the potential method as outlined by Kostochka and Yancey. This technique starts with a vertex-minimum counterexample, and then applies a ``quotient" which takes the counterexample to a smaller $4$-critical graph. Here we pause to point out an important detail about why our main theorem involves vertex-disjoint triangles rather than just triangles: though the quotient operation may create many triangles, it turns out it can only create at most three additional vertex-disjoint triangles. Our potential function allows us to use the structure of the smaller graph effectively. There are multiple situations which can occur when passing to the smaller graph. One outcome is that the potential of the smaller $4$-critical graph is at most $-2$; this case should be considered the ``easy" case. In this case, we follow a similar approach to Theorem \ref{KostochkaYancey4} and apply a counting lemma (the Potential-Extension Lemma, Lemma \ref{potentialextensionlemma}) to show our minimum counterexample has a specific structure. In particular, we  show that our minimum counterexample does not contain cycles where all vertices in the cycle have degree $3$; nor does it contain $K_{4}-e$ subgraphs. 

But this is not all that can occur: when we take the quotient we might end up with a smaller graph which has potential bigger than $-2$. In this case, we know the smaller graph belongs to one of three classes: either it is $4$-Ore, in $\mathcal{B}$, or $W_{5}$. Using the structure of these graph classes, we again will be able to deduce that our minimum counterexample does not contain $K_{4}-e$ as a subgraph, nor cycles where all of the vertices are of degree $3$. Here it is fundamental to know the structure of these graphs to be able to make these assertions. Without the stronger theorem statement characterizing the potential of these graph classes, it seems unlikely to be able to be able to make progress as the more general density bounds for $4$-critical graphs simply are not strong enough. 

Once we have proven some basic structural results, our goal will be to prove that in a minimum counterexample, if we let $X$ be the set of vertices that have degree $3$, then the graph induced by $X$ does not contain components with more than six vertices. In fact, we will ensure that if there is a component containing more than two vertices, then our minimum counterexample contains at least one triangle. With this, we end the paper with a discharging argument that shows that our minimum counterexample cannot exist. 

We note that many of the ideas developed in this paper are not specific to $4$-critical graphs and readily generalize to $k$-critical graphs. In particular, we conjecture the following strengthening of our result to $k$-critical graphs.

\begin{conj}
Every $k$-critical graph with no $K_{k-1}$ subgraph satisfies
\[e(G) \geq \frac{(k+1)(k-2)v(G) +k(k-3)}{2(k-1)}.\]
\end{conj}

Our result says that the $k=4$ case of this theorem is true; an unpublished result of Davies says that this conjecture is best possible infinitely often for each $k \geq 5$.

The paper is organized as follows.  In Section \ref{cliquesection}, we present structural lemmas regarding triangles in $4$-Ore graphs.  In Section \ref{T8structure}, we present results regarding the triangles of graphs in $\mathcal{B}$. In Section \ref{potentialsection}, we give a brief overview of the potential method and results specific to its use for $4$-critical graphs. Section \ref{mincounterexamplesection} begins by supposing the existence of a minimum counterexample $G$ to Theorem \ref{maintheorem}, and then deduces various structural properties of $G$: in particular, that $G$ has no $K_{4}-e$ subgraph, that $G$ has no cycles of degree $3$ vertices, and that each connected component of the graph induced by degree $3$ vertices of $G$ has at most six vertices. Finally, Section \ref{dischargingsection} uses discharging to show that no such minimum counterexample exists.

While essential to the overall proof, Sections \ref{cliquesection} and \ref{T8structure} are quite lengthy and technical. It may be easier for a reader to skip these sections at first, and then return to them as needed.

\section{Triangles in $4$-Ore graphs}
\label{cliquesection}

In this section we prove many structural results about triangles in $4$-Ore graphs that are required for our proof of Theorem \ref{maintheorem}. Recall that $4$-Ore graphs are the graphs for which Theorem \ref{KostochkaYanceytight} is tight; for this reason, we will expend significant effort to understand their structure. The most critical fact that we need is that aside from $K_{4}$ and the unique seven-vertex $4$-Ore graph called the \textit{Moser spindle} (see Figure \ref{tightexamples}), $4$-Ore graphs with few triangles contain two vertex-disjoint $K_{4}-e$ subgraphs with special properties. See Lemma \ref{K4e} for a more precise statement.  This leads to the second-most important fact: if a $4$-Ore graph $G$ has exactly three vertex disjoint triangles, then for any triangle $T$ in $G$, either $G-T$ has two vertex-disjoint triangles, or $G-T$ contains a special $K_{4}-e$ subgraph.  For convenience, we will denote the Moser spindle as $M$. 

The first two observations are easy and are proved in \cite{4criticalgirth5}.

\begin{obs}[Proposition 2.1, \cite{4criticalgirth5}]
\label{deletingavertex}
If $G$ is $4$-Ore and $v$ is any vertex in $V(G)$, then $G-v$ contains a triangle.
\end{obs}

\begin{obs}[Proposition 2.2, \cite{4criticalgirth5}]
\label{deletingaclique}
If $G$ is $4$-Ore and not isomorphic to $K_{4}$, then for any triangle $T$ in $G$, $G-T$ contains a triangle.
\end{obs}

The following definition is helpful to avoid repetition. 

\begin{definition}
Let $G$ be a graph. A \textit{triangle packing of $G$} is a maximum collection of vertex-disjoint triangles in $G$.
\end{definition}

Thus $T^{3}(G)$ is the size of a triangle packing of a graph $G$. It will be useful to bound the size of a triangle packing of an Ore composition, which we do now.

\begin{prop}
\label{cliqueboundinequality}
If $G$ is an Ore composition of $H_{1}$ and $H_2$ where $H_{1}$ is the edge side and uses edge $xy$ and $H_{2}$ is the split side and uses vertex $z$, then 
\[T^{3}(G) \geq T^{3}(H_{1}) + T^{3}(H_{2})-f_1(H_{1})-f_2(H_{2}),\]
where $f_1(H_1)$ and $f_2(H_2)$ are defined as follows: 
\begin{itemize}
    \item $f_1(H_1) = 0$ if there exists a triangle-packing of $H_1$ that avoids the edge $xy$, and $f_1(H_1) = 1$ otherwise; and
    \item $f_2(H_2) = 0$ if there exists a triangle-packing of $H_2$ that avoids $z$, and $f_2(H_2) = 1$ otherwise.  
\end{itemize} 

\end{prop}

\begin{proof}
Let $H_{1}$ be the edge side of the composition with edge $xy$ and $H_{2}$ the split side,where we split $z$ into vertices $z_{1}$ and $z_{2}$. 

For $i \in \{1,2\}$, let $\mathcal{T}_{i}$ be a triangle packing of $H_{i}$, and let $\mathcal{T}_{i}'$ be the triangle packing obtained from $\mathcal{T}_{i}$ by removing a triangle if it uses either $z$ or the edge $xy$. Note that $|\mathcal{T}_{i}'| \geq |\mathcal{T}_{i}| -1$, and equality holds if and only if $\mathcal{T}_{i}$ contains a triangle using either the edge $xy$ or the vertex $z$. Moreover, by construction  $\mathcal{T}_{1}' \cup \mathcal{T}_{2}'$ is a collection of vertex-disjoint triangles. It follows that
\[T^{3}(G) \geq T^{3}(H_{1}) + T^{3}(H_{2}) -f_1(H_1)-f_2(H_2).\]
\aftermath
\end{proof}

\begin{cor}
\label{Kkbound}
If $G$ is an Ore composition of a graph $H$ and $K_{4}$, then $T^{3}(G) \geq T^{3}(H)$.
\end{cor}

\begin{proof}
Note that $T^{3}(K_{4}) =1$, and that for any vertex $v \in V(K_{4})$, there is a triangle in $K_{4}-v$. Hence regardless of whether $K_{4}$ is the split side or edge side of the composition, by Proposition \ref{cliqueboundinequality} we have as desired that
\[T^{3}(G) \geq T^{3}(H) +T^{3}(K_{4}) -1 = T^{3}(H).\]
\aftermath
\end{proof}

\begin{cor}
\label{onecliquecharacterization}
The only $4$-Ore graph $G$ with $T^{3}(G) =1$ is $K_{4}$; every other $4$-Ore graph $H$ has $T^3(H) \geq 2$.
\end{cor}

\begin{proof}
Let $G$ be a vertex-minimum counterexample. Since $G \neq K_{4}$, we have that $G$ is the Ore composition of two graphs $H_{1}$ and $H_{2}$. If neither $H_{1}$ nor $H_{2}$ is $K_{4}$, then by the minimality of $G$ we have that $T^{3}(H_{i}) \geq 2$ for $i \in \{1,2\}$ and hence by Proposition \ref{cliqueboundinequality} we have $T^{3}(G)$ $\geq T^{3}(H_{1}) + T^{3}(H_{2}) -2 \geq 2$. 

Now consider the case where exactly one of $H_1$ and $H_2$ is isomorphic to $K_4$. Let $i \in \{1,2\}$ be such that $H_i = K_4$. Then $T^{3}(G)$ $\geq T^{3}(H_{3-i}) \geq 2$ by Corollary \ref{Kkbound}. Therefore $H_{1} = H_{2} = K_{4}$.
In this case $G$ is the Moser spindle, which has a triangle packing of size $2$ (see Figure \ref{tightexamples}). 
\end{proof}

The observations above all generalize easily to $k$-Ore graphs for $k \geq  5$. However, past this point the lemmas we prove require that $k =4$; and hence our paper is restricted to the study of $4$-critical graphs.

\begin{definition}
A \textit{kite} in $G$ is a $K_{4}-e$ subgraph $K$ such that the vertices of degree $3$ in $K$ have degree $3$ in $G$. The \textit{spar} of a kite $K$ is the unique edge in $E(K)$ contained in both triangles of $K$.
\end{definition}

The following two lemmas partially describe the structure of $4$-Ore graphs with triangle packings of size two.
\begin{lemma}
\label{K4e}
If $G$ is a $4$-Ore graph with $T^{3}(G) =2$, then $G$ contains two edge-disjoint kites that share at most one vertex. Furthermore, if $G$ is not the Moser spindle $M$, then $G$ contains two vertex-disjoint kites.
\end{lemma}
\begin{proof}
We proceed by induction on $v(G)$. As $T^{3}(G) =2$, by Corollary \ref{onecliquecharacterization} we have that $G \neq K_{4}$. Hence $G$ is the Ore composition of two graphs $H_{1}$ and $H_{2}$. Up to relabelling, we may assume that $H_{1}$ is the edge side of the composition where we delete the edge $xy$, that $H_{2}$ is the split side where the vertex $z$ is split into two vertices $z_{1}$ and $z_{2}$, and that $x$ is identified with $z_1$ and $y$ with $z_2$ in $G$. We break into cases depending on which (if any) of $H_1$ and $H_2$ is isomorphic to $K_4$. \\
\textbf{Case 1: $H_1 = H_2 = K_4$}. \\ In this case, $G$ is isomorphic to the Moser spindle, which contains two edge-disjoint kites that share exactly one vertex.

\noindent
\textbf{Case 2: $H_{1} = K_{4}$, and $H_2 \neq K_4$.} \\
First suppose that $H_{2} \neq M$. Then by the induction hypothesis, $H_{2}$ contains two vertex-disjoint kites. Note that $z$ belongs to at most one of these kites, and hence $H_{2}^{z}$ contains a kite not containing $z_{1}$ or $z_{2}$. Observe that $H_{1}-xy$ is a kite, and thus in this case $G$ contains two vertex-disjoint kites. Therefore we may assume that $H_{2} = M$. If $z$ is the unique vertex of degree $4$ in $M$, then $T^{3}(G) =3$, a contradiction. But if $z$ is not the unique vertex of degree $4$, then there is a kite in $M^{z}$ containing neither $z_1$ nor $z_2$, and thus this kite and $H_{1}-xy$ give two vertex-disjoint kites.

\noindent
\textbf{Case 3: $H_1 \neq K_4$, and $H_{2} = K_{4}$.} \\
First suppose that $H_{1} \neq M$. Then by the induction hypothesis there are two vertex-disjoint kites in $H_{1}$, say $D_{1}$ and $D_{2}$. Thus either there is a kite in $H_{1}-xy$ that avoids both $x$ and $y$, or up to relabelling $x \in V(D_{1})$ and $y \in V(D_{2})$. In either case, since $H_{2}^{z}$ contains a kite that contains at most one of $z_1$ and $z_2$, it follows that $G$ contains two vertex-disjoint kites.

Therefore $H_{1}=M$. If $T^{3}(H_{1}-xy) =2$, then $T^{3}(G) =3$, a contradiction. Hence in $H_{1}$, both $x$ and $y$ have degree $3$. Further, $x$ and $y$ are not the two vertices that have degree $3$ and are not incident to a spar of a kite. Thus it follows that there is a kite in $H_{1}-xy$ that avoids both $x$ and $y$. Since there is a kite in $H_{2}^{z}$, there are two vertex-disjoint kites in $G$.\\
\textbf{Case 4: $H_{1} \neq K_{4}$, and $H_{2} \neq K_{4}$.}\\
First suppose $H_{2} \neq M$. Then by induction, $H_{2}$ contains two vertex-disjoint kites, and hence $H_{2}^{z}$ contains a kite that avoids both $z_1$ and $z_2$. Similarly, $H_{1}$ contains two edge-disjoint kites by induction. Thus $H_{1}-xy$ contains a kite, and so $G$ contains two vertex-disjoint kites.

Therefore $H_{2} = M$. If $z$ is the unique vertex of degree $4$ in $M$, then $T^{3}(H_{2}^{z}-z_{1}-z_{2}) =2$, and it follows that $T^{3}(G) \geq 3$, a contradiction. Thus $z$ is  not the degree $4$ vertex in $M$, and thus there is a kite in $H_{2}^{z}$ that avoids both $z_1$ and $z_2$. By induction, there are two edge-disjoint kites in $H_{1}$, and hence there is a kite in $H_{1}-xy$. It follows that there are two vertex-disjoint kites in $G$, as desired.
\end{proof}

\begin{cor}
\label{onlytwok4e}
If $G$ is $4$-Ore and $T^{3}(G) = 2$, then $G$ contains exactly two subgraphs isomorphic to a kite.
\end{cor}

\begin{proof}
By Lemma \ref{K4e}, we have that $G$ contains at least two subgraphs isomorphic to a kite. Thus it suffices to show $G$ does not contain at least three such subgraphs. If $G = M$, then one simply checks that there are precisely two subgraphs isomorphic to a kite. Therefore we may assume that $G \neq M$, and so by Lemma \ref{K4e}, we have that $G$ contains two vertex- disjoint kites $K$ and $K'$. Let $K''$ be a third kite distinct from $K$ and $K'$. If $K''$ is disjoint from both $K$ and $K'$, then $T^{3}(G) \geq 3$, a contradiction. Therefore without loss of generality we may assume that $V(K) \cap V(K'') \neq \emptyset$, and $|V(K) \cap V(K'')| \geq |V(K') \cap V(K'')|$. Let $V(K) = \{v_{1},v_{2},v_{3},v_{4}\}$ where $v_{2}v_{3}$ is the spar of $K$, and $V(K') = \{u_{1},u_{2},u_{3},u_{4}\}$ where $u_{2}u_{3}$ is the spar of $K'$. Observe that $K''$ does not contain any of $v_{2},v_{3},u_{2},u_{3}$ as these vertices have degree $3$ in $G$. Suppose $v_{1} \in V(K'')$ but $v_{4} \not \in V(K'')$. Then note that since $|V(K) \cap V(K'')| \geq |V(K') \cap V(K'')|$, it follows that $K'-K''$ contains a triangle $T$. But then $G$ contains a triangle packing of size at least three, namely the triangle in $K''$ containing $v_{1}$, the triangle in $K$ not containing $v_{1}$, and $T$. Thus we can assume that both $v_{1}$ and $v_{4} \in V(K'')$. Let $w_1$ and $w_2$ be the endpoints of the spar of $K''$. Then since $G$ is 4-critical, $G-w_{1}-w_{2}$ has a 3-colouring $\varphi$ which extends to a 3-colouring of $G$ by setting $\varphi(w_1) = \varphi(v_2)$ and $\varphi(w_2) = \varphi(v_3)$. This contradicts the fact that $G$ is 4-critical.   Thus it follows $G$ has exactly two subgraphs isomorphic to a kite.  
\end{proof}
\begin{lemma}
\label{splittinglemma}
Let $G$ be $4$-Ore with $T^{3}(G) = 2$. Let $v \in V(G)$, and let $G^{v}$ be the graph obtained by splitting $v$ into two vertices $v_1$ and $v_2$ of positive degree with $N(v_1) \cup N(v_2) = N(v)$ and $N(v_1) \cap N(v_2) = \emptyset$. Then either
\begin{enumerate}[(i)]
\item{$T^{3}(G^{v}) \geq 2$, or}
\item{ $\deg(v) = 3$, there is an $i \in \{1,2\}$ such that $\deg(v_{i}) = 1$, and the edge $e$ incident to $v_{i}$ is the spar of a kite in $G$.} 
\end{enumerate}
\end{lemma}

\begin{proof}

We proceed by induction on the number of vertices. First suppose that $G =M$. If $v$ is the unique vertex of degree $4$, then it is easy to verify that $G-v$ contains two vertex-disjoint triangles, and hence $T^{3}(G^{v}) \geq T^3(G-v) \geq 2$.  Now suppose that $v$ is incident to one of the two spars of kites. Again, it is easy to check that if the vertex of degree one is incident to the spar of the kite, then  $T^{3}(G^{v}) = 1$, and otherwise $T^{3}(G^{v}) \geq 2$. Lastly, if $v$ is either of the other two vertices in $M$, then one simply checks that $T^{3}(G^{v}) \geq 2$ for any split. 

Therefore we can assume that $G \neq M$. Let $G$ be the Ore composition of $H_{1}$ and $H_{2}$ where $H_{1}$ is the edge side where we delete the edge $xy$, and $H_{2}$ is the split side where we split the vertex $z$ into two vertices $z_{1}$ and $z_{2}$. \\
\textbf{Case 1: $H_{1} = K_{4}$.}\\
We can assume that $H_{2} \neq K_{4}$, as otherwise $G = M$. This implies that $T^{3}(H_{2}) = 2$: if $T^{3}(H_{2}) \geq 3$, then Proposition \ref{cliqueboundinequality} implies that $T^{3}(G) \geq 3$, a contradiction.  First suppose that $v \in \{x,y\}$.  Without loss of generality, let $v =x$. By Observation \ref{deletingavertex}, there is a triangle in $H_{2} -z$. Since there is also a triangle in $H_{1}- x$, we have that $T^{3}(G^{v}) \geq 2$ as desired. 

Now suppose that $v \in V(H_{1})\setminus \{x,y\}$. Note $H_{1}-xy$ is a kite. Assume $\deg(v_{1}) =1$. If $v_{1}$ is not incident to the spar of a kite, then $(H_{1}-xy)^{v}$ contains a triangle. By Observation \ref{deletingavertex} there is also a triangle in $H_{2} -z$. This implies that $T^{3}(G^{v}) \geq 2$, as desired. 

Thus we can assume that $v \in V(H_{2} -z)$. We apply induction to $H_{2}$. If $T^{3}(H_{2}^{v}) \geq 2$, this implies $T^{3}(H_{2}^{v}-z) \geq 1$, and hence $T^{3}(G^{v}) \geq 2$.  Otherwise we split $v$ in such a way that in $H_{2}^{v}$, $\deg(v_{1}) = 1$ and $v_1$ is incident to a spar of a kite in $H_{2}$. 
Let $K$ be this kite. Note if $z \not \in V(K)$ then the same split occurs in $G^{v}$ and $(ii)$ occurs. Hence $z \in V(K)$.  If $T^{3}(H_{2}-z) \geq 2$, then $T^{3}(H_{2}^{v} - z) \geq 1$, and it follows that $T^{3}(G^{v}) \geq 2$. Hence every triangle packing of $H_{2}$ uses the vertex $z$. 

First suppose $H_2 = M$. Since every triangle packing of $H_2$ uses $z$, it follows that $z$ is the endpoint of the spar of a kite $K'$. Since $z \in V(K)$ and $z$ is the endpoint of the spar of $K'$, it follows that $K = K'$. But then $T^3(H_2 - z - v) = 1$, and so $T^3(G^v) \geq 2$, a contradiction. Therefore $H_{2} \neq M$, and thus by Lemma \ref{K4e}, $H_{2}$ contains two vertex disjoint kites $K^*$ and $K^{**}$. By Corollary \ref{onlytwok4e}, we may assume $K^{**} =K$. But then there is a triangle in $H_{2}-z-v$ (namely, either triangle in $K^*$), and thus $T^{3}(G) \geq 2$ as desired. \\
\textbf{Case 2: $H_{2} = K_{4}$.}\\
In this case, $T^{3}(H_{1}) = 2$ and every triangle packing of $H_{1}$ uses the edge $xy$ since otherwise $G = M$ or $T^{3}(G) \geq 3$ by Proposition \ref{cliqueboundinequality}.

First suppose that $v \in \{z_{1},z_{2}\}$. Without loss of generality, let $v=z_{1}$. Note that $T^3(H_2-z) \geq 1$ by Observation \ref{deletingavertex}, and so $T^{3}(H_{2}^{z}- \{z_{1},z_{2}\}) \geq 1$. Since $T^{3}(H_{1}-z_{1}) \geq 1$ again by Observation \ref{deletingavertex}, it follows that $T^{3}(G^{v}) \geq 2$, as desired. 

Now suppose that $v \in V(H_{1}) \setminus \{x,y\}$. Consider $H_{1}^{v}$, where we perform the same split as in $G^v$. First suppose that $T^3(H_{1}^{v}) \geq 2$. Then $H_{1}^{v}-xy$ contains at least one triangle. Since there is a triangle in $H_{2}^z - \{z_{1},z_{2}\}$, it follows that $T^{3}(G^{v}) \geq 2$ as desired. Therefore we may assume that  $T^3(H_1^v) < 2$, and so by induction $v$ is incident to the spar in a kite $K$ in $H_{1}$, and after splitting $v_{1}$ has degree one and is incident to the spar of the kite. If $K$ is in $G$, then we are done. Therefore we may assume that $xy \in E(K)$, and thus $\{x,y,v\}$ induces a triangle in $H_{1}$. Since $H_1 \neq K_4$ by assumption, by Observation \ref{deletingaclique} we have that $H_1 - \{x,y,v\}$ contains a triangle. As there is also a triangle in $H_{2}- z$ by Observation \ref{deletingavertex}, it follows that $T^{3}(G^{v}) \geq 2$, as desired. 

The final subcase to consider is when $v \in V(H_{2}) - \{z_{1},z_{2}\}$. If $v$ is not incident to the spar of the kite in $H_{2}^{z}$, then any split of $v$ leaves a triangle in $(H_2^{z})^v$. Recall that in Case 2, every triangle packing of $H_1$ uses the edge $xy$, and $T^3(H_1) = 2$. Thus there is a triangle in $H_{1} - \{x,y\}$, and so $T^{3}(G^{v}) \geq 2$ as desired. A similar argument works for the other splits, unless we split $v$ in such a way that $v_{1}$ has degree one and is incident to a spar of a kite in $G$. \\
\textbf{Case 3: Neither $H_{1}$ nor $H_{2}$ is $K_{4}$.}\\
First suppose that $v \in \{x,y\}$ and without loss of generality, that $v =x$. Note that since $T^{3}(H_1) \geq 2$ and $T^3(H_2) \geq 2$, it follows that  $T^{3}(H_{1}-x) \geq 1$ and  $T^{3}(H_{2}^z -\{z_{1},z_{2}\}) \geq 1$. Hence in this case, after splitting $v$ we have that $T^{3}(G^{v}) \geq 2$.

 Next suppose that $v \in V(H_{1})\setminus \{x,y\}$. If $T^3(H_1^v) \geq 2$, then it follows that $T^3(H_1^v-xy) \geq 1$. Since $T^3(H_2^z-\{z_1,z_2\}) = 1$, we have that $T^3(G^v) \geq 2$, and so (i) holds. Thus by induction we may assume that $\deg(v) = 3$, and $v_{1}$ has degree one and is incident to the spar of a kite in $H_{1}$. Let $K$ be this kite. If $K$ avoids both $x$ and $y$ then $(ii)$ holds in $G^{v}$ and we are done. Otherwise $x,y,v$ induce a triangle, and by Observation \ref{deletingaclique}, $H_{1} - \{x,y,v\}$ contains a triangle, and $H_{2}^{z}- \{z_{1},z_{2}\}$ contains a triangle by Observation \ref{deletingavertex}. Hence $T^{3}(G^{v}) \geq 2$ as desired. 
 
 Finally suppose that $v \in V(H_{2}) \setminus \{z_{1},z_{2}\}$. If $T^{3}(H_{2}^{z}) \geq 2$, then $T^{3}((H_{2}^{z})^{v}) \geq 1$, and since $T^{3}(H_{1}-xy) \geq 1$, it follows that $T^{3}(G^{v}) \geq 2$ as desired. Therefore $v$ has degree $3$, lies in a kite $K$ in $H_{2}$,  $\deg(v_{1}) = 1$ and $v_1$ is incident to the spar of the kite in $K$. If $z \not \in V(K)$, then $(ii)$ occurs in $G^{v}$. So $z \in V(K)$. Let $T$ be a triangle in $K$ which contains $z$ and $v$. Then by Observation \ref{deletingaclique}, $H_{2} - T$ contains a triangle, and since deleting any vertex in $H_{1}$ leaves a triangle, it follows that $T^{3}(G^{v}) \geq 2$. 
\end{proof}

We now describe the structure of the $4$-Ore graphs with triangle packings of size three.

\begin{lemma}
\label{deletingatriangle}
Let $G$ be a $4$-Ore graph with $T^{3}(G) =3$, and let $T$ be a triangle in $G$. Either $T^{3}(G-T) \geq 2$, or there exists a kite in $G-T$.
\end{lemma}

\begin{proof}
Suppose not, and let $G$ and $T \subseteq G$ form a vertex-minimum counterexample. Since $G \neq K_{4}$, we have that $G$ is the Ore composition of two $4$-Ore graphs $H_{1}$ and $H_{2}$. Up to relabelling, we may assume that $H_{1}$ is the edge side of the composition where we delete the edge $xy$, and that $H_{2}$ is the split side where we split the vertex $z$ into two vertices $z_{1}$ and $z_{2}$. Observe that at most one of $z_{1}$ and $z_{2}$ is in $V(T)$. Additionally, notice that if $T^3(H_{i}) \geq 3$ for each $i \in \{1,2\}$, then by Proposition \ref{cliqueboundinequality} we have that $T^{3}(G) \geq T^{3}(H_{1}) +T^{3}(H_{2}) -2 \geq 4$, contradicting a hypothesis of the lemma. We break into cases depending on which (if any) of $H_{1}$ and $H_{2}$ is isomorphic to $K_{4}$. \\
\textbf{Case 1: $H_{1} =K_{4}$.} \\
Note that $H_{2} \neq K_{4}$ as otherwise $T^{3}(G) =2$. We may assume that $T^3(H_2) \leq 3$; otherwise, Proposition \ref{cliqueboundinequality} gives $T^3(G) \geq T^3(K_4) + T^4(H_2)-1 \geq 1 + 4 - 1 = 4$, contradicting a hypothesis of the lemma.\\ 
\textbf{Subcase 1: $T^{3}(H_{2}) = 3$.}\\
We may assume that $T^{3}(H_{2}^{z}-z_1-z_2) \leq 2$ since otherwise $T^{3}(G) \geq 4$, a contradiction. Note this implies that $T^{3}(H_{2}^{z}-z_{1}-z_{2}) = 2$, since $T^{3}(H_{2}-z) \geq T^{3}(H_{2}) -1 \geq 2$. Hence every triangle packing of $H_{2}$ has a triangle which contains the vertex $z$. If $V(T) \subseteq V(H_1)$, then as $T^{3}(H_{2}^{z}-z_{1}-z_{2}) = 2$, we have $T^{3}(G-T) \geq 2$, a contradiction. Thus we may assume $V(T) \subseteq V(H_{2}^{z})$. Note that $T$ contains one of $z_1$ and $z_2$, since otherwise $H_1-xy$ is a kite in $G-T$, a contradiction. Without loss of generality, let $z_{1} \in V(T)$. Let $T'$ be the triangle in $H_{2}$ whose vertex set is $V(T) \setminus \{z_{1}\} \cup \{z\}$.  Consider $H_{2}-T'$. By minimality, either $T^3(H_2-T') \geq 2$, or $H_2-T'$ contains a kite. First suppose $T^3(H_{2}-T') \geq 2$.  Since $z \in V(T')$, it follows that there are two vertex-disjoint triangles in $H^z_2-T$, a contradiction.   Therefore $H_2-T'$ contains a kite, $D$; but then $D$ exists in $G-T$, a contradiction.\\
\textbf{Subcase 2 : $T^{3}(H_{2}) =2$.}\\
Again up to relabelling $z_{1}$ with $z_{2}$, we may assume $z_{1} \in V(T)$, as otherwise the $G-T$ contains the kite $H_{1}-e$. By Lemma \ref{K4e}, either $H_{2} = M$, or $H_{2}$ contains two vertex-disjoint kites. First consider the case where $V(T) \subseteq V(H_{1})$. If there is a kite in $H_2^z$, then $G-T$ contains a kite and we are done.  Thus we may assume that $H_2^z$ does not contain a kite, and so that $H_{2} = M$ and $z$ is the unique vertex of degree $4$ in $H_2$. But in this case there are two vertex-disjoint triangles in $H_{2}-z$, which implies that $T^{3}(G-T) \geq 2$, a contradiction.
 
Thus for the remainder of the analysis, we assume that $V(T) \subseteq V(H_{2}^z)$, and $z_{1} \in V(T)$. Let us deal with the case where $H_{2} = M$ first. Suppose that $z$ is the unique vertex of degree $4$ in $M$. Since $z_{1} \in V(T)$, we have that $H_{2}^{z}-T$ contains a triangle that avoids $z_{2}$. Since  $H_{1}-z_{1}$ also contains a triangle, it follows that $T^{3}(G-T) \geq 2$, contradicting the fact that $G$ is a counterexample. Thus $z$ is not the unique vertex of degree $4$ in $M$. If $z$ is any of the vertices in $M$ incident to a spar of a kite, then either $T^{3}(G) =2$ (a contradiction), or for any triangle intersecting $z_{1}$ in $H_{2}^{z}$, there is a triangle in $H_{2}^{z}-T-z_{2}$. Thus it follows that $T^{3}(G-T) \geq 2$ by using the triangle in $H_{1}-xy$ which contains $z_{2}$. If $z$ is either of the other two vertices of degree $3$, we have a kite in $H_{2}^{z} - T$, and hence there is a kite in $G-T$, a contradiction.

Therefore $H_{2} \neq M$, and so by Lemma \ref{K4e} we have that $H_{2}$ contains two vertex-disjoint kites $D_{1}$ and $D_{2}$. Without loss of generality, we may assume $V(D_{1}) \subseteq V(H_{2}-z)$.  We claim that no vertex in $D_1$ incident with the spar is contained in $T$. To see this, suppose not, and let $v$ be a vertex incident with the spar of $D_1$. If $v$ is in $T$, then since all neighbours of $v$ are in $D_1$ and $z$ is in $T$, it follows that $v$ is adjacent to $z$. But then $z$ is in $D_1$, contradicting the definition of $D_1$. Moreover, we claim at most one vertex of $D_1$ is contained in $T$.
If the two vertices in $D_{1}$ which are not incident to the spar of $D_{1}$ are in $T$, then $G$ contains a $K_{4}$ subgraph, which implies $G = K_{4}$, a contradiction. 
It follows that we have $T^{3}(H_{2}-T) \geq 1$. Hence using one of the triangles in $H_{1}-xy$, we see that $T^{3}(G-T) \geq 2$, contradicting the fact that $G$ is a counterexample.

\noindent
\textbf{Case 2: $H_{2} = K_{4}$.}

Note in this case $H_{1} \neq K_{4}$ as otherwise $T^{3}(G) =2$.  Similar to the previous case, we may assume that $T^3(H_1) \leq 3$; otherwise, Proposition \ref{cliqueboundinequality} gives $T^3(G) \geq T^3(K_4) + T^4(H_1)-1 \geq 1 + 4 - 1 = 4$, contradicting the hypothesis of the lemma.\\
\textbf{Subcase 1: $T^{3}(H_{1}) = 3$.}\\
In this case, $T^{3}(H_{1}-xy) = 2$ since otherwise Proposition \ref{cliqueboundinequality} implies $T^{3}(G) =4$. It follows that every triangle packing of $H_{1}$ contains a triangle using the edge $xy$, and thus there are two vertex-disjoint triangles in $H_{1}-x-y$. If $V(T) \subseteq V(H_{2}^{z})$, then since $T^{3}(H_{1}-x-y) =2$, it follows that $T^{3}(G-T) \geq 2$, a contradiction. Thus $V(T) \subseteq V(H_{1})$. Now consider $H_{1}-T$. By minimality, we have two possibilities: either $T^3(H_1-T) \geq 2$, or $H_1-T$ contains a kite. If $H_{1}-T$ contains two vertex-disjoint triangles, then $H_{1}-T-xy$ contains at least one triangle. Since $H_{2}-\{z_{1},z_{2}\}$ contains a triangle, we see that $T^{3}(G-T) \geq 2$, a contradiction. Otherwise, $H_{1}-T$ contains a kite $D$. Thus either $xy$ is the spar of $D$, or $T^{3}(H_{1}-T-xy) \geq 1$. If $T^{3}(H_{1}-T-xy) \geq 1$, then again using the triangle in $H_{2}- \{z_{1},z_{2}\}$ we see that $T^{3}(G-T) \geq 2$. Thus $xy$ is the spar of $D$. In this case, $V(T) \subseteq V(H_{1}- x-y)$ as neither $x$ nor $y$ lies in a triangle. But then $G-T$ contains the kite in $H_{2}^{z}$. \\
\textbf{Subcase 2: $T^{3}(H_{1}) = 2$.} \\
Suppose first that $H_{1} = M$. If $xy$ is not incident to the unique vertex of degree $4$, then either there is a kite in $H_{1} -xy$ that avoids both $x$ and $y$, or $H_{1}-xy$ contains two edge-disjoint kites. First suppose there is a kite in $H_{1}-xy$ that avoids both $x$ and $y$. Observe there is a kite in $H_{2}^{z}$. Since either $V(T) \subseteq V(H_{1})$ or $V(T) \subseteq V(H_{2}^{z})$, by the structure of the Moser spindle it follows that $G-T$ contains a kite for any $T$. Now consider the case where $H_{1}-xy$ contains two edge-disjoint kites. Since $T$ avoids both $x$ and $y$, $T$ contains the unique vertex of degree $4$ in $M$. Otherwise, $H_1-xy-T$ contains a kite. But then $H_{1}-xy-T$ contains a triangle using (say) $z_{1}=x$. As $H_{2}^{z}-z_{1}$ contains a triangle, we have that $T^{3}(G-T) \geq 2$, a contradiction.

So we may assume that $xy$ is incident to the unique vertex of degree $4$ in $H_1=M$. Note that in this case, either up to relabeling $\deg(z_1) = 5$ and $\deg(z_2) = 3$, or $\deg(z_1) = \deg(z_2) = 4$. If $\deg(z_1) = 5$, then $T$ contains $z_1$: otherwise, $G-T$ contains a kite. If $T \subseteq H_1-xy$, then since $H_1-xy$ contains a triangle disjoint from $T$ and $H_2-z$ contains a triangle, it follows that $T^3(G-T) = 2$, a contradiction. If on the other hand $T \subseteq H_2^z$, then since $T^3(H_1-xy-z_1)\geq 2$, again it follows that $T^3(G-T) \geq 2$. Thus we may assume $\deg(z_1) = \deg(z_2) = 4$.  But then $G$ contains two vertex-disjoint kites $K_1$ and $K_2$, and no triangle in $G$ intersects both $K_1$ and $K_2$. Thus $G-T$ contains a kite, again contradicting the fact that $G$ is a counterexample.

Therefore by Lemma \ref{K4e} we may assume that $H_{1} \neq M$, and so that $H_1$ contains two vertex-disjoint kites $D_{1}$ and $D_{2}$. Up to relabelling, let $z_{1}$ be in the kite in $H_{2}^{z}$. First suppose $V(T) \subseteq V(H_{2}^{z})$. Note there is kite in $H_{1}-xy$ not using $z_1$. Since $z_1z_2 \not \in E(G)$, we have that $z_{2} \not \in V(T)$. Thus $H_1-xy-T = H_1-xy-z_1$, and so $G-T$ contains at least one of the kites $D_{1}$ and $D_{2}$, a contradiction. \\
Therefore we may assume that $H_1 \neq M$. Up to relabeling, let $z_1$ be in the kite in $H_2^z$. First suppose $T \subseteq H_1-xy$. Note then that $z_1 \in V(T)$, since otherwise $G-T$ contains the kite in $H_z^2$.  Thus $H_1-xy-T = H_1-T$, since $xy$ is incident with a vertex in $T$. By Observation \ref{deletingaclique}, $H_1-T$ contains a triangle. Since $H_2-z$ also contains a triangle, it follows that $T^3(G-T) \geq 2$, a contradiction.  Thus we may assume $T \subseteq H_2^z$.  By Lemma \ref{K4e}, since $H_1 \neq M$, $H_1$ contains two vertex-disjoint kites. But then $H_1-z_1 \subseteq H_1-T$ contains a kite, contradicting the fact that $G$ is a counterexample.
 \\
\noindent
\textbf{Case 3: Neither $H_{1}$ nor $H_{2}$ is $K_{4}$.} \\
 Recall that by Corollary \ref{onecliquecharacterization}, both we have $T^3(H_1) \geq 2$ and $T^3(H_2) \geq 2$. Similar to the previous two cases, we may assume that $T^3(H_2) \leq 3$ and $T_3(H_1)\leq 3$; otherwise, Proposition \ref{cliqueboundinequality} gives $T^3(G) \geq  4$, contradicting the hypothesis of the lemma.\\
\textbf{Subcase 1: $T^{3}(H_{1}) =2$ and $T^{3}(H_{2}) = 2$.}\\
Note that by Lemma \ref{K4e},  $H_1$ and $H_2$ contains two edge-disjoint kites. If $T \subseteq H_2^z-z_1-z_2$, then $H_1-xy$ (and therefore $G-T$) contains a kite, a contradiction. Moreover, if $T \subseteq H_1-z_1-z_2$, then either $H_2^z$ (and therefore $G-T$) contains a kite, or $H_2 = M$ and $z$ is the unique vertex of degree $4$ in $M$, in which case $T^3(G-T)\geq T^3(H_2-z) \geq 2$. Thus we may assume that $T$ contains one of $z_1$ and $z_2$: up to relabeling, suppose $T$ contains $z_1$. If $T \subseteq H_1-xy$, then $H_1-xy-T = H-T$. By Observation \ref{deletingaclique}, $H_1-xy-T$ (and therefore $G-T$) contains a triangle. Since $T^3(H_2) = 2$, there is a triangle in $H_2-z$, and so $T^3(G-T) \geq 2$, contradicting the fact that $G$ is a counterexample. If, on the other hand, $T \subseteq H_2^z$, then since $T^3(H_1) = 2$, again we have $T^3(H_1-z_1) \geq 1$. Since $z_1 \in T$, it follows that $H_2-T \subseteq H_2^z-z_1-z_2$. By Observation \ref{deletingaclique}, $H_1-T$ (and thus $H_2^z-z_1-z_2$) also contains a triangle. Thus $T^3(G-T) \geq 2$, again a contradiction. \\
 \textbf{Subcase 2: $T^{3}(H_{1}) = 3$.}\\
Recall that $H_2 \neq K_4$ by assumption; and as noted prior to Case 1, if $T^3(H_1) \geq 3$, then $T^3(H_2) < 3$. Thus $T^3(H_2) = 2$. Suppose first that every triangle packing of $H_1$ uses the edge $xy$. Then $H_1-x-y$ has a triangle packing of size two, and so $T \subseteq H_1-xy$ as otherwise $T^3(G-T) \geq 2$ and we are done. Similarly, if there is a triangle packing of $H_1$ that does not use the edge $xy$, then $T^3(H_1-xy) = 3$, and so again $T \subseteq H_1-xy$, as otherwise  $T^3(G-T) \geq 2$ and we are done (since $T$ contains at most one of $x$ and $y$). Since $T^3(H_2) = 2$, it follows from Lemma \ref{K4e} that $H_2$ contains two edge-disjoint kites $D_1$ and $D_2$. Thus either $H_2^z$ contains a kite that avoids both $z_1$ and $z_2$ (and so $G-T$ contains this kite), or $z \in D_1 \cap D_2$. In this case, $H_2 = M$, and $z$ is the unique vertex of degree $4$ in $M$. But then $H_2^z-z_1-z_2$ contains a triangle packing of size two, and since $T \subseteq H_1-xy$, it follows that $T^3(G-T) \geq 2$, a contradiction. \\ 
\textbf{Subcase 3: $T^{3}(H_{2})=3$.} \\
Then $T^3(H_z^2-\{z_1, z_2\}) \geq 2$. It follows that $V(T) \subseteq V(H_2^z)$, as otherwise $T^3(G-T) \geq 2$. Note that since $H_1 \neq K_4$ by assumption, $T^3(H_1) \neq 1$. Furthermore, as noted prior to Case 1, $T^3(H_1) < 3$. Thus $T^3(H_1) = 2$.  By Lemma \ref{K4e}, it follows that either $H_1 = M$, or that $H_1$ contains two vertex-disjoint kites. Suppose first  $H_1 = M$. Then either $H_1-xy-T$ contains a kite, or $T^3(H_1-xy-T) = 2$, since $T \subset H_2^z$. (To see this, note that since $T \subseteq H^z_2$ and $T$ contains at most one of $x$ and $y$, removing the edge $xy$ and $T$ from $H_1$ amounts to deleting one edge and at most one of its incident vertices.) Thus we may assume $H_1 \neq M$, and so that $H_1$ contains two vertex-disjoint kites $D_1$ and $D_2$. But then $H_1-xy-T$ contains a kite (since removing $xy$ and $T$ from $H_1$ again amounts to deleting an edge and at most one of its incident vertices, and $D_1$ and $D_2$ are vertex-disjoint).
\end{proof}

\begin{definition}\label{founddef}
Let $G$ be a $4$-Ore graph with $T^3(G) =3$. An edge $f$ is \textit{foundational} if both $T^3(G-f) =2$ and there is no kite in $G-f$.
\end{definition}

\begin{lemma}
\label{foundationaledgesin4Ore}
Let $G$ be a $4$-Ore graph with $T^3(G) =3$. There is at most one foundational edge in $G$.  Moreover, if $f$ is a foundational edge, then $f$ is the spar of a kite.
\end{lemma}

\begin{proof}
Let $G$ be a $4$-Ore graph with $T^{3}(G) =3$. Since $G \neq K_{4}$, it follows that $G$ is the Ore composition of two $4$-Ore graphs $H_{1}$ and $H_{2}$. Up to relabelling, we may assume that $H_{1}$ is the edge side of the composition where we delete the edge $xy$ and that $H_{2}$ the split side of the composition where we split $z$ into two vertices $z_{1}$ and $z_{2}$,  and identify $z_{1}$ with $x$ and $z_{2}$ with $y$. Let $f$ be an edge in $G$, and suppose $f$ is foundational. \\
\textbf{Case 1: $H_{1} = K_{4}$.} \\
 Observe that if $f \not \in E(H_{1})$, then $f$ is not foundational since there is a kite left over after deleting $f$. Therefore $f \in E(H_1)$. If $T^{3}(H_{2}^{z}) = 3$, then since $f \in E(H_{1})$ we have that  $T^{3}(G-f)=3$,  a contradiction. Thus $T^3(H_2^z) = 2$, and so $T^3(H_2) \leq 3$. Further, $T^{3}(H_{2}) \geq 2$, since if $H_{2} =K_{4}$, then $G = M$ and $T^{3}(M) =2$.

If $T^3(H_2) = 3$, then there are two vertex-disjoint triangles in $H_{2}-z$, say $T_{1}$ and $T_{2}$. If $f$ is not the spar of the kite $H_{1}-xy$, then $H_{1}-xy-f$ contains a triangle, so it follows that $T^{3}(G-f) \geq 3$. Therefore in this case there is at most one foundational edge in $G$, namely, the spar of the kite in $H_{1}-xy$, as desired. 

Thus we may assume $T^3(H_2) = 2$. By Lemma \ref{K4e}, we have that $H_2$ contains two edge-disjoint kites that share at most one vertex. If $H_{2}^{z}$ contains a kite, then $G-f$ contains a kite, and so $f$ is not foundational.  Thus we may assume that $H_2^z$ does not contain a kite. By Lemma \ref{K4e}, we have that $H_2 = M$ and moreover that $z$ is the vertex of degree $4$ in $M$. But for any split of $z$ into $z_1$ and $z_2$, we get that $T^3(H_2^z -z_{1}-z_{2}) = 2$. Thus if $f$ is not the spar in $H_1-xy$, then $G-f$ has a triangle packing of size three, contradicting that $f$ is foundational. \\
\textbf{Case 2: $H_{2} = K_{4}$.} \\
By possibly relabelling, let $z_1$ be the vertex of degree two in $H_2^z$ resulting from the split of $z$. Notice that splitting $K_{4}$ leaves a kite subgraph, and hence as $f$ is foundational, $f$ is in $E(H_{2})$. Furthermore, $f$ is not incident with $z_1$ or $z_2$, as otherwise $T^3(G-f) = T^3(G) =3$.

Note that if $T^{3}(H_{1}) =2$, then by Lemma \ref{K4e} $H_{1}$ contains two edge-disjoint copies of kites, and thus $H_{1}-xy$ contains at least one kite, contradicting that $f$ is foundational. Hence $T^{3}(H_{1}) =3$.

We claim that $T^{3}(H_{1}-xy) =2$. To see this, note that if $T^{3}(H_{1}-xy) = 3$,
then $T^3(G) \geq T^3(H_1-xy) + T^3(H_2-z) = 3+1$, contradicting the lemma hypothesis. Thus every triangle packing of $H_{1}$ uses $xy$, and hence there exists a triangle packing of $H_{1}-xy$ which does not use $x$ or $y$. Therefore if $f$ is not the spar in the kite contained in $H_z^2$, $G-f$ contains three vertex-disjoint triangles, a contradiction.\\
\textbf{Case 3: Neither $H_{1}$ nor $H_{2}$ is $K_{4}$.} \\
Note that either $T^{3}(H_{1}) = 2$ or $T^{3}(H_{2}) =2$, as otherwise $T^{3}(G)\geq 4$ by Proposition \ref{cliqueboundinequality}. 

First suppose both $T^{3}(H_{1})=2$ and $T^{3}(H_{2})=2$. Then by Lemma \ref{K4e}, in each of $H_{1}$ and $H_{2}$ there are two edge-disjoint kites which share at most one vertex. Hence there is a kite in $H_{1}-xy$. If $f \in E(H_{2})$, then $G -f$ thus contains a kite, contradicting that $f$ is foundational. Therefore $f \in E(H_{1})$. If $H_{2}^{z}$ contains a kite, then $G-f$ contains a kite, and thus in this case $G$ contains no foundational edge. It follows that the kites in $H_2$ were not vertex-disjoint, and hence by Lemma \ref{K4e} we have that $H_{2} = M$ and $z$ is the unique vertex of degree $4$ in $M$. Thus $T^3(H_2-z)= T^{3}(H_{2}^{z} - \{z_{1},z_{2}\}) =2$. Since $H_1-xy$ contains a kite, if $f$ is not the spar of the kite in $H_1-xy$ then $T^3(G-f) = 3$, contradicting that $f$ is foundational.

Now suppose that $T^{3}(H_{1}) =3$. Since $T^3(G) = 3$, Proposition \ref{cliqueboundinequality} implies that $T^{3}(H_{2}) =2$. Then by Lemma \ref{K4e}, there are two edge-disjoint kites in $H_{2}$ which share at most one vertex. If there is no kite in $H_{2}^{z}$, then $H_2 = M$ and $z$ is the unique vertex of degree $4$ in $M$. But then $T^{3}(H_{2}^{z}-z_1-z_2) =2$, which implies that $T^{3}(G) \geq 4$, a contradiction. Thus there is a kite in $H_{2}^{z}$. If both of the edge-disjoint kites in $H_2$ are in $H_{2}^{z}$, then no matter the choice of $f$, we have that $G-f$ contains a kite and hence $G$ has no foundational edge. Therefore $H_{2}^{z}$ contains exactly one kite, and this kite avoids both $z_1$ and $z_2$.  If $f$ does not lie in this kite, then $G-f$ contains a kite. If $f$ is not the spar, then $T^3(H_2^z-f-z_1-z_2) \geq 1$ and so $T^3(G-f) \geq 3$, contradicting that $f$ is foundational. Thus there is at most one foundational edge, and it is the spar of a kite.

Finally suppose that $T^3(H_{2}) =3$. Thus $T^3(H_2^z -z_1-z_2) \geq 2$. Since $T^3(G) = 3$, Proposition \ref{cliqueboundinequality} implies that $T^{3}(H_{1})=2$. Thus by Lemma \ref{K4e}, $H_{1}$ contains two edge-disjoint kites which share at most one vertex. Thus $H_{1}-xy$ contains a kite. If $f$ does not lie in this kite, then $f$ is not foundational.  Moreover, if $f$ is not the spar of this kite, then $T^{3}(G-f) \geq 3$ by using two triangles in $H_{2}-z$ and a triangle from $H_1-xy-f$. Hence there is at most one foundational edge, and it is the spar of a kite.
\end{proof}

\begin{lemma}
\label{4Oresplit3triangle}
Let $G$ be $4$-Ore with $T^3(G) =3$. Let $v \in V(G)$, and let $G^{v}$ be obtained from $G$ by splitting $v$ into two vertices $v_1$ and $v_2$ of positive degree with $N(v_1) \cup N(v_2) = N(v)$ and $N(v_1) \cap N(v_2) = \emptyset$. Then one of the following occurs:
\begin{enumerate}[(i)]
\item{$T^{3}(G^{v}) \geq 3$,}
\item{$G^{v}$ contains a kite,}
\item{there is an $i \in \{1,2\}$ such that $\deg(v_{i}) =1$, and the edge incident to $v_{i}$ is foundational in $G$.}
\end{enumerate}
\end{lemma}

\begin{proof}
Let $G$ be a $4$-Ore graph with $T^{3}(G) = 3$, and $v \in V(G)$. As $G \neq K_{4}$, we have that $G$ is the Ore composition of two $4$-Ore graphs $H_{1}$ and $H_{2}$. Up to relabelling, let $H_{1}$ be the edge side of the composition where we delete the edge $xy$ and $H_{2}$ the split side of the composition where we split $z$ into two vertices $z_{1}$ and $z_{2}$, and identify $z_{1}$ with $x$ and $z_{2}$ with $y$. Note that at least one of $H_{1}$ or $H_{2}$ is not $K_{4}$, as otherwise $G = M$ and $T^{3}(M) =2$, contradicting the hypotheses of the lemma. \\
\textbf{Case 1: $H_{1} =K_{4}$.}\\
Observe that $H_{1}-xy$ is a kite, so if $v  \not \in V(H_{1})$, then $G^{v}$ contains a kite and so (ii) holds. Hence we may assume that $v \in V(H_{1})$. Since $T^3(G) = 3$ and $T^3(G) \geq T^3(H_2-z)+1$ we have that $T^3(H_2) \in \{2,3\}$.

First suppose that $T^{3}(H_{2}) =3$. Observe that there are two vertex-disjoint triangles in $H_{2}-z$. If $(H_{1}-xy)^v$ contains a triangle, then  $T^{3}(G^{v}) \geq 3$ and (ii) holds. Moreover, if $v \in \{x,y\}$,  then $H_{1}-xy$ contains a triangle, so we may assume that $v \in V(H_{1}) - \{x,y\}$. Let $w$ and $v$ be the two vertices in $V(H_{1}) - \{x,y\}$. Observe there is exactly one way to split $v$ into $v_{1}$ and $v_{2}$ so that there is no triangle left over in $H_{1}- xy)^v$: up to relabelling $v_{1}$ to $v_{2}$, we have that $v_{1}$ is adjacent to $w$, and $v_{2}$ is adjacent to $x$ and $y$. To finish, notice that the number of vertex-disjoint triangles in $G^{v}$ after performing such a split is the same as the number of vertex-disjoint triangles in $G-vw$. Hence if $T^{3}(G-vw) \geq 3$, we have $T^{3}(G^{v}) = 3$  and thus (i) holds. So $T^{3}(G-vw) = 2$, and further $G^{v}$ has no kite subgraph, which implies that $G-vw$ does not contain a kite. Thus $vw$ is foundational, and so (iii) holds, as desired.

Therefore we may assume that $T^{3}(H_{2}) = 2$. If there is a kite in $H_{2}^{z} -z_{1}-z_{2}$, then $G$ contains two vertex-disjoint kites, and thus there is a kite in $G^{v}$. By Lemma \ref{K4e}, we have that $G = M$ and $z$ is the unique vertex of degree $4$ in $M$. Then $T^{3}(H_{2}^{z}-z_{1}-z_{2}) =2$. Therefore we may assume that $v \not \in \{x,y\}$ as otherwise $T^{3}(G^{v}) \geq 3$ and (i) holds. Let $w,v$ be the two vertices in $H_{1}-x-y$. By the same argument as in the $T^{3}(H_{2})=3$ case, there is exactly one split so that $T^{3}(G^{v}) \leq 2$, and in this case, we split $v$ into two vertices $v_{1},v_{2}$ where without loss of generality, $\deg(v_{1}) = 1$, and $v_{1}$ is incident to a foundational edge in $G$. In this case, (iii) holds, as desired. \\
\textbf{Case 2: $H_{2} =K_{4}$.} \\
Since $T^3(G) = 3$ and $T^3(H_2-z) = 1$, it follows that $T^3(H_1) \in \{2,3\}$. Throughout this case, without loss of generality let $z_{1}$ have degree two in $H_{2}^{z}$ and $z_{2}$ have degree one in $H_{2}^{z}$. Observe that $H_{2}^{z}-z_{2}$ contains a kite, so if $v \not \in V(H_{2}^{z})-z_{2}$, then $G^{v}$ contains a kite and (ii) holds. Thus we may assume that $v \in V(H_2^{z}-z_2)$, as otherwise we are done.

First suppose that $T^{3}(H_{1}) = 3$. Then $T^{3}(H_{1}-x) \geq 2$, and it follows that if $v =z_{1}$, then $T^{3}(G^{v}) =3$ and (i) holds, as desired. Therefore we may assume that $v \in V(H_{2}) - \{z_{1},z_{2}\}$. If $v$ is adjacent to $z_{2}$, then there is a triangle in $H_{2}^{z}$ after splitting $v$. Since $T^3(H_1-x) \geq 2$, we get $T^{3}(G^{v}) \geq 3$ and (i) holds. If $v$ is either of the other two possible vertices, the only split which does not leave a triangle is one where up to relabelling $v_{1}$ has degree one and is incident to the spar of the kite in $H_{2}^{z}$. Let $f$ be this spar. If $T^{3}(G-f) \geq 3$, we have that $T^{3}(G^{v}) \geq 3$ and (i) holds. So $T^{3}(G-f) =2$. Moreover, $G-f$ does not contain a kite as otherwise $G^{v}$ contains a kite (satisfying (ii)). Hence $f$ is foundational in $G$, and thus $v_{1}$ is incident to the foundational edge in $G$. Thus (iii) holds, as desired.

Therefore we may assume that $T^{3}(H_{1}) = 2$. Then by Lemma \ref{K4e}, either $H_{1} =M$ or $G$ contains two vertex-disjoint kites. If $H_{1}-x-y$ contains a kite, then there are two vertex-disjoint kites in $G$, and hence $G^{v}$ contains a kite, satisfying (ii). Thus we may assume that $H_{1}=M$, and so $T^{3}(H_{1}-xy) =2$. If $v = z_{1}$, observe we have $T^{3}(G^{v}) \geq 3$ (thus (i) holds), and if we split $z_{2}$, we have a kite in $G^{v}$ (and so (ii) holds).  If $v$ is adjacent to $z_{2}$ in $H_{2}^{z}$, then observe that any split of $v$ results in a triangle in $(H_2^z)^v$, and hence $T^{3}(G^{v}) \geq 3$ and (i) holds. Thus we may assume that $v$ is a vertex in $H_{2}^{z}$ incident to a spar of the kite. By the same arguments as the case when $T^{3}(H_{1}) =3$, the only split of such a vertex that does not leave a triangle results in, up to relabelling, $v_{1}$ having degree one and being incident to a foundational edge in $G$. But then (iii) holds, as desired. \\
\textbf{Case 3: $T^{3}(H_{1}) =2$.} \\
By the previous cases, we may assume that $H_{2} \neq K_{4}$. Thus $T^{3}(H_{2}) \geq 2$. Recall that since $T^3(G) = 3$, by Proposition \ref{cliqueboundinequality} it follows that $T^3(H_2) \in \{2,3\}$. Moreover, by Lemma \ref{K4e} either $H_{1} = M$ or $H_{1}$ contains two vertex-disjoint kites. In either case, there is a kite subgraph in $H_{1}-xy$. Let $L$ be such a subgraph. Then we may assume that $v \in V(L)$, as otherwise $G^{v}$ contains a kite subgraph, satisfying (ii). 

First suppose that $T^{3}(H_{2}) = 3$. If we split $v$ and there is still a triangle left in $H_{1}^{v}-xy$, then as $T^{3}(H_{2}-z) \geq 2$, we have $T^{3}(G^{v}) \geq 3$. Hence (i) holds. Therefore we may assume that there is no triangle in $H_{1}^{v}-xy$, as otherwise we are done. By the same arguments as in previous cases, this implies that $v$ is incident to the spar of $L$, and we split $v$ in such a way that up to relabelling $v_{1}$ is incident to the spar of the $L$ and $v_1$ has degree one in $G^{v}$. Further, the spar of $L$ is foundational, satisfying (iii); otherwise, the split of $v$ satisfies at least one of (i) and (ii). 

Thus we may assume that $T^{3}(H_{2}) = 2$. First suppose that $T^{3}(H_{2}-z) =2$. Then if we split $v$ in $L$ and are left with a triangle, $T^{3}(G^{v}) \geq 3$ and (i) holds. Thus by the same arguments as in previous cases, $v$ is incident to the spar of $L$, and we split $v$ in such a way that up to relabelling $v_{1}$ is incident to the spar of $L$ and has degree one in $G^{v}$. Further, the spar of $L$ is foundational, satisfying (iii); otherwise, the split of $v$ satisfies at least one of (i) and (ii). Thus $T^{3}(H_{2}-z) =1$. By Lemma \ref{K4e}, since $T^3(H_2) = 2$ either $H_{2} = M$ or $H_{2}$ has two vertex-disjoint kites. As $T^{3}(H_{2}-z) =1$, this implies that $z$ is incident to a spar of a kite. But then regardless of whether $H_{2} = M$ or $H_2$ has two vertex-disjoint kites, we have that there is a kite in $H_{2}^{z}$ that avoids both $z_1$ and $z_2$. But then $G^{v}$ contains a kite, satisfying (ii). \\
\textbf{Case 4: $T(H_{2}) =2$.} \\
Recall that by Proposition \ref{cliqueboundinequality} we may assume that $T^3(H_1) \leq 3$; and from the previous cases, we may assume further that $T^{3}(H_{1}) = 3$. Then $T^{3}(H_{1}-xy) \geq 2$, with equality only if every triangle packing of $H_1$ contains a triangle that uses the edge $xy$. Note by Proposition \ref{cliqueboundinequality}, if $T^{3}(H_{1}-xy) \geq 3$, then $T^{3}(G) \geq 4$, a contradiction. Hence there are two disjoint triangles in $H_{1}-xy$ which do not use $x$ or $y$. Therefore $T^{3}(H_{2}^{z}) = 1$, as otherwise by Proposition \ref{cliqueboundinequality}, $T^{3}(G) =4$. By appealing to Lemma \ref{K4e}, this implies that there is a kite $L$ in $H_{2}^z$ that avoids both $z_1$ and $z_2$. If $v \not \in V(L)$, then $G^{v}$ contains a kite, satisfying (ii). If splitting $v$ leaves a triangle, then $T^{3}(G^{v}) \geq 3$ and so (i) holds. Let $w,v$ be the two vertices of degree $3$ in $L$. It follows that up to relabelling, after splitting we have $v_{1}w \in E(G^{v})$ and $v_{2}$ is incident to the other two edges of $v$. Thus $\deg(v_{1}) =1$. Note that if $vw$ is not foundational, then this split satisfies one of (i) and (ii) and we are done. Hence $vw$ is foundational, and thus (iii) holds. \\
\textbf{Case 5: Both $T^{3}(H_{1}) =3$ and $T^{3}(H_{2}) =3$.} \\
Then by Proposition \ref{cliqueboundinequality}, $T^{3}(G) \geq 3 + 3 -2 =4$, contradicting the hypotheses of the lemma. 
\end{proof}

\section{Properties of graphs in $\mathcal{B}$}
\label{T8structure}
In this section we prove lemmas similar to those in Section \ref{cliquesection}, except now we focus on graphs in $\mathcal{B}$. We recall the definition of $\mathcal{B}$: the graph $T_{8}$ (shown in Figure \ref{T8pic2}) is in $\mathcal{B}$, and given a graph $G \in \mathcal{B}$ and a $4$-Ore graph $H$, the Ore composition $G'$ of $G$ and $H$ is in $\mathcal{B}$ if $T^{3}(G') =2$. We start by proving that the potential of graphs in $\mathcal{B}$ is in fact $-1$.

\begin{definition}
The \textit{Kostochka-Yancey potential} of a graph $G$, denoted $\text{KY}(G)$, is defined as $\text{KY}(G) := 5v(G) -3e(G)$.
\end{definition}
The following observation is immediate from the definition of Ore composition. 

\begin{obs}
\label{Orecomposition}
Let $G$ be the Ore composition of two graphs $H_{1}$ and $H_{2}$. Then $v(G) = v(H_{1}) +v(H_{2}) -1$, $e(G) = e(H_{1}) + e(H_{2}) -1$, and $\text{KY}(G) = \text{KY}(H_{1}) + \text{KY}(H_{2}) -2$. 
\end{obs}

\begin{cor}\label{notinb}
If $G \in \mathcal{B}$, then $\text{KY}(G) =1$ and $p(G) =-1$. 
\end{cor}

\begin{proof}
Let $H$ be a vertex-minimum counterexample. If $H=T_{8}$, then $v(T_{8}) = 8$ and $e(T_{8}) = 13$, and thus $5 \cdot 8- 13 \cdot 3 =1$. Now suppose $H$ is the Ore composition of two graphs $H_{1}$ and $H_{2}$, where without loss of generality, $H_{1} \in \mathcal{B}$ and $H_{2}$ is $4$-Ore. By minimality, it follows that $\text{KY}(H_1) = 1$. Note that by Theorem \ref{KostochkaYancey4}, $\text{KY}(H_2) = 2$. Then by Observation \ref{Orecomposition}, it follows that $\text{KY}(H) = 1 + 2 -2=1$.  

Since $\text{KY}(G) = 1$ and by definition $T^{3}(G) = 2$, it follows that $p(G) =-1$, as desired.
\end{proof}

We overload the terminology.
\begin{definition}
 Given a graph $G \in \mathcal{B}$, an edge $e \in E(G)$ is \textit{foundational} if $T^{3}(G-e) =1$ and there is no $K_{4}-e$ subgraph in $G-e$. 
\end{definition}

Note that this version of \emph{foundational} differs from Definition \ref{founddef} in that we enforce that $G-e$ contains no $K_{4}-e$ subgraph at all: such a subgraph may not be a kite.

\begin{lemma}
\label{foundationaledgesinB}
If $G$ is a graph in $\mathcal{B}$, then $G$ contains at most one foundational edge. Further, if $G$ is not $T_{8}$ and $G$ contains a foundational edge, then this edge is the spar of a kite.
\end{lemma}

\begin{proof}
First suppose that $G = T_{8}$. Observe that the edge $u_{1}u_{2}$ is the only foundational edge in $G$, and so in this case the lemma holds. Hence we may assume that $G$ is not $T_8$, and so that $G$ is the Ore composition of a $4$-Ore graph $H_{1}$  and a graph $H_{2}$ in $\mathcal{B}$. We may assume that $G$ has a foundational edge $f$, as otherwise there is nothing to show. Note that $T^{3}(H_{1}) \leq 2$, as otherwise $T^{3}(G) \geq 3$ by Proposition \ref{cliqueboundinequality}.\\
\textbf{Case 1: $H_{1} = K_{4}$.} \\
Suppose that $H_{1}$ is the edge side where we delete the edge $xy$, and we split a vertex $z$ in $H_{2}$ into two vertices $z_{1}$ and $z_{2}$. Note that if $f$ is not in $H_{1}-xy$, then $G-f$ contains a kite, and so $f$ is not foundational, a contradiction. Hence $f$ lies in $E(H_{1}-xy)$.
As $H_{1}-xy$ is a kite, if we delete any edge that is not the spar of this kite, we have $T^{3}(H_{1}-xy-f) \geq 1$. Further, there is at least one triangle in $H_{2}$ which does not use $z$, and hence $T^{3}(G-f) \geq 2$. Thus $G$ contains a single foundational edge, and $f$ is the spar of a kite, as desired.

Now suppose that $H_{1}$ is the split side where we split $z$ into $z_{1}$ and $z_{2}$. Then $H_{1}^{z}$ contains a kite. If $f$ does not lie in this kite, then $G-f$ contains a kite and so $f$ is not foundational, a contradiction. If $f$ is not the spar of the kite, then $T^{3}(H_{1}^{z}-f) \geq 1$, and since there is a triangle in $H_{2}-xy$, we see that $T^{3}(G-f) \geq 2$, and so again $f$ is not foundational. Hence $G$ contains a single foundational edge, and $f$ is the spar of a kite, as desired. \\
\textbf{Case 2: $T^{3}(H_{1}) =2$.}\\
By Lemma \ref{K4e}, either $H_{1} = M$ or $H_{1}$ contains two vertex-disjoint kites. First suppose that $H_{1}$ is the edge side of the composition, where we delete the edge $xy$. By Proposition \ref{cliqueboundinequality} and the fact that $T^{2}(G) = 2$, every triangle packing of $H_{1}$ contains a triangle using the edge $xy$. Thus regardless of whether $H_{1} = M$ or not, there is a kite $L$ in $H_{1}-xy$. Then $f$ is in $L$, otherwise $G-f$ contains a kite, contradicting the fact that $G$ is foundational. If $f$ is not the spar of the kite in $L$, then $T^{3}(H_{1}-f-xy) \geq 1$, and since  $T^3(H_2-z) \geq 1$, we have that $T^{3}(G-f) \geq 2$. Therefore $f$ is the spar of a kite, and it is the only foundational edge.

Therefore we may assume that $H_{1}$ is the split side of the composition where we split a vertex $z$ into two vertices $z_{1}$ and $z_{2}$. Again by Proposition \ref{cliqueboundinequality} and the fact that $T^{2}(G) =2$, every triangle packing of $H_{1}$ uses $z$. Thus regardless of whether $H_{1} = M$ or not, $z$ is incident to the spar of a kite in $H_{1}$, and so there is a kite $L$ in $H_{1}^{z}$ that avoids both $z_1$ and $z_2$. Then $f$ is in $L$, as otherwise $G-f$ contains a kite contradicting the fact that $f$ is foundational. If $f$ is not the spar of the kite in $L$, then $T^{3}(H_{1}-f-z) \geq 1$, and thus $T^{3}(G-f) \geq 2$, again contradicting that $f$ is foundational. Thus $f$ is the spar of a kite, and it is moreover the only foundational edge
\end{proof}

\begin{lemma}
\label{T8splits}
Let $G \in \mathcal{B}$, and let $G^{v}$ be obtained from $G$ by splitting a vertex $v$ into two vertices $v_{1}$ and $v_{2}$. Then at least one of the following occurs:
\begin{enumerate}[(i)]
\item{$T^{3}(G^{v}) \geq 2$,}
\item{$G^{v}$ contains a $K_{4}-e$ subgraph,}
\item{there is an $i \in \{1,2\}$ such that $\deg(v_{i}) = 1$ and the edge incident to $v_{i}$ is foundational in $G$.}
\end{enumerate}
\end{lemma}
\begin{proof}
 First consider the case where $G = T_{8}$. If we do not split one of $u_{1}$ or $u_{2}$ we have a $K_{4}-e$ subgraph remaining. If we split either $u_{1}$ or $u_{2}$ such that (iii) does not hold, then it is easy to see $T^{3}(G^{v}) =2$ and so (i) holds.

Therefore we can assume that $G$ is the Ore composition of a graph $H_{1} \in \mathcal{B}$ and a $4$-Ore graph $H_{2}$. If $T^{3}(H_{2}) \geq 3$, then by Proposition \ref{cliqueboundinequality} we have that $T^{3}(G) \geq 3 +2 -2 \geq 3$ contradicting that $T^{3}(G) =2$. Hence $T^{3}(H_{2}) \leq 2$.  \\
\noindent
\textbf{Case 1: $H_{2}=K_{4}$.} \\
Suppose first that $H_{2}$ is the split side where we split a vertex $z$ into two vertices $z_{1}$ and $z_{2}$. Then $H_{2}^{z}$ contains a kite, say $L$, so if $v \not \in V(L)$, then $G^{v}$ contains a kite and (ii) holds. If $v$ is not incident to a spar of the kite, then any split of $v$ results in a triangle, and thus $T^{3}(G^{v}) \geq 2$  and so (i) holds. Therefore $v$ is incident to the spar of the kite, and further the split of $v$ must leave up to relabelling $v_{1}$ with degree one and $v_1$ incident to the spar $f$ of the kite. We claim that $f$ is foundational. To see this, note that $T^3(G^v) = T^3(G-f)$. If $G-f$ contains a $K_4-e$ subgraph, then this subgraph also appears in $G^v$ and so (ii) holds and we are done. If $T^3(G-f) \geq 2$, then $T^3(G^v) \geq 2$, and so (i) holds and we are done. If $G-f$ possesses neither of these properties, then by definition $f$ is foundational. Thus (iii) holds, and again we are done.

Now suppose that $H_{2}$ is the edge side of the composition where we delete the edge $xy$. Then $H_{2}-xy$ is a kite. If $v \not \in V(H_{2}-xy)$, then $G^{v}$ contains a kite, as desired. If $v$ is not incident to a spar of a kite, then any split leaves a triangle in $H_{2}^{v}-xy$, and thus  $T^{3}(G^{v}) \geq 2$ and (i) holds. To see this, note that  since $T^3(H_1) = 2$, it follows that $T^3(H_1-z) \geq 1$. Thus $v$ must be incident to the spar of $H_{2}-xy$, and if $T^{3}(G^{v}) =1$, then we must have split $v$ in such a way that up to relabelling, $\deg(v_{1}) = 1$ and $v_1$ is incident to the spar of the kite. Thus (iii) holds. \\
\noindent
\textbf{Case 2: $H_{2} \neq K_{4}$.}\\
Recall that $H_2$ is $4$-Ore. Thus $T^3(H_2) \geq 2$ by Corollary \ref{onecliquecharacterization}; and since $T^3(G) = 2$, we have moreover that $T^3(H_2) = 2$ by Proposition \ref{cliqueboundinequality}. First suppose $H_{2}$ is the edge side of the composition where we delete the edge $xy$. Then $T^{3}(H_{2}-xy) =1$ as otherwise $T^{3}(G) \geq 3$ by Proposition \ref{cliqueboundinequality}. By Lemma \ref{K4e} either $H_{2} = M$ or there are two vertex-disjoint kites in $H_{2}$. This implies that there is a kite in $H_{2} -xy$. Let $L$ be this kite. If $v \not \in V(L)$, then $G^{v}$ contains a kite, and so (ii) holds. If $v$ is not incident to a spar of a kite, then any split leaves a triangle in $H_{2}^{v}-xy$, and thus $T^{3}(G^{v}) \geq 2$ and (i) holds. As in the previous case, this follows from the fact that $T^3(H_1-z) \geq 1$. Thus $v$ must be incident to the spar of $H_{2}-xy$, and if $T^{3}(G^{v}) =1$, then we must have split $v$ in such a way that up to relabelling, $\deg(v_{1}) = 1$ and $v_1$ is incident to the spar of the kite. But then (iii) holds.

Therefore we can suppose that $H_{2}$ is the split side of the composition, where we split the vertex $z$ into two vertices $z_{1}$ and $z_{2}$. By Proposition \ref{cliqueboundinequality}, every triangle packing of $H_{2}$ uses the vertex $z$. Thus regardless of whether $H_{2} = M$ or not, $z$ is incident to a spar of a kite in $H_{2}$. Thus there is a kite in $H_{2}^{z}$ that avoids both $z_1$ and $z_2$. Let $L$ be this kite. If $v \not \in V(L)$, then $G^{v}$ contains a kite, and so (ii) holds. If $v$ is not incident to the spar of $L$, then any split leaves a triangle in $(H_{2}^{z}-z_1-z_2)^v$, and thus $T^{3}(G^{v}) \geq 2$, and (i) holds. Thus $v$ is incident to the spar of $L$, and if $T^{3}(G^{v}) =1$, then $v$ was split in such a way that up to relabelling, $\deg(v_{1}) = 1$ and $v_1$ is incident to the spar of the kite. But then (iii) holds, as desired.
\end{proof}

\section{A review of the Potential Method}
\label{potentialsection}

In this section we review the basics of the potential method, which is the critical tool for the rest of the paper. As in the previous sections, we specialize this section to $4$-critical graphs; however, those familiar with graph homomorphisms should be able to easily extend the observations to the setting of $k$-critical graphs for other values of $k$. Of particular importance is the Potential-Extension Lemma (Lemma \ref{potentialextensionlemma}), which will be employed liberally. We start off with an important definition. 

\begin{definition}
Let $G$ be a $4$-critical graph, and let $F$ be any induced subgraph of $G$ with $v(F) < v(G)$. Let $\phi$ be a 3-colouring of $F$. Let $C_{1},C_{2}$ and $C_{3}$ be the (possibly empty) colour classes of $\phi$ (where a \emph{colour class} is understood here to be the set of vertices in $G$ which are mapped to the same colour under $\phi$). The \textit{quotient of $G$ by $\phi$}, denoted $G_{\phi}[F]$, is a graph with vertex set $(V(G) \setminus V(F)) \cup \{c_{i} \, | \, 1 \leq i \leq 3 \}$, and edge set $E_{1} \cup E_{2} \cup E_{3}$ where:
\begin{itemize}
   \item $E_1 =\{uv \, | \, uv \in E(G[V(G) \setminus V(F)])\}$;
    \item $E_2$ is the set of edges of the form $c_{i}c_{j}$ where there is a $u \in C_{i}$ and a $v \in C_{j}$ such that $uv \in E(G)$;
   \item $E_3$ is the set of edges of the form $uc_{i}$ such that there is a $v \in C_{i}$ where $uv \in E(G)$.
\end{itemize}
\end{definition}

The observation below is very simple, but is fundamental to the usefulness of the quotient.
\begin{obs}[\cite{Shortproof}, Claim 8]
\label{easyhom}
Let $G$ be a $4$-critical graph. If $F$ is a strict induced subgraph of $G$ with a $3$-colouring $\phi$, then $G_{\phi}[F]$ contains a $4$-critical subgraph.
\end{obs}

 This motivates the following definitions which appear in many papers (see \cite{EvelyneMasters, Oredegree7}, for example). See Figure \ref{potentialextensionexample} for an illustration.

\begin{definition}
Let $G$ be a $4$-critical graph and let $F$ be a strict induced subgraph of $G$. Let $\phi$ be a $3$-colouring of $F$. Let $W$ be a $4$-critical subgraph of $G_{\phi}[F]$. Let $X$ be the graph induced in $W$ by the vertices which are not vertices of $G$. We will call $X$ the \textit{source}. Let $F'$ be the subgraph of $G$ induced by $V(F) \cup (V(W) \setminus V(X))$. We say $F'$ is the \textit{extension} of $W$ and $W$ is the \textit{extender} of $F$. 
\end{definition}

We will always assume the source $X$ is a clique. While this does not follow from the definition, we can always choose a $3$-colouring such that this occurs. To see why, suppose $X$ is not a clique. Then we have a $3$-colouring $\phi$ with colour classes $C$ and $C'$ such that no vertex in $C$ is adjacent to a vertex in $C'$. Then we can take a new colouring $\phi'$ where we simply recolour all vertices in $C'$ with the colour used by $C$.  Repeating this procedure allows us to assume $X$ induces a clique.

The next two lemmas are very similar to Lemma 3.5 in \cite{4criticalgirth5}.

\begin{lemma}
\label{countinglemma}
Let $G$ be a $4$-critical graph, and let $F$ be a strict induced subgraph of $G$. Let $\phi$ be a $3$-colouring of $F$. Let $W$ be a $4$-critical subgraph of $G_{\phi}[F]$. Let $F'$ be the extension of $W$ and let $X$ be the source of $\phi$. The following hold:
\begin{itemize}
\item{$v(F') = v(F) +v(W) - v(X)$,}
\item{$e(F') \geq e(F) + e(W) - e(X)$, and}
\item{$T^{3}(F') \geq T^{3}(F) + T^{3}(W \setminus X)$.}
\end{itemize} 
\end{lemma}

\begin{proof}
Observe that $V(F') = V(F) \cup (V(F') \setminus V(F))$. Additionally, $V(W) = (V(F') \setminus V(F)) \cup X$. Thus
$V(F') = (V(F) \cup V(W)) \setminus X$, and so $v(F') = v(F) + v(W) - v(X)$. From the above identity and the fact that the subgraphs are induced we see that $e(F') \geq e(F) + e(W) -e(X)$. 
Finally, let $\mathcal{T}_{1}$ and $\mathcal{T}_{2}$ be triangle packings of $F$ and $W\setminus X$, respectively. Then $\mathcal{T}_{1} \cup \mathcal{T}_{2}$ is a set of disjoint triangles in $F'$ and has size $T^{3}(F) + T^{3}(W \setminus X)$, which gives the result.  
\end{proof}

The following lemma is used frequently throughout the rest of the paper. We refer to it as the Potential-Extension Lemma.
\begin{lemma}[Potential-Extension Lemma]
\label{potentialextensionlemma}
Let $F$ be a strict induced subgraph of $G$, and let $\phi$ be a fixed $3$-colouring of $F$. With respect to $\phi$, let $F',W$ and $X$ be an extension, extender and source of $F$, respectively. The following identifies hold:
\[p(F') \leq p(F) + p(W) - 5v(X) +3e(X) +T^{3}(W) -T^{3}(W \setminus X),\]
and 
\[p(F') \leq p(F) + p(W) -4v(X) +3e(X).\]
\end{lemma}

\begin{proof}

Observe that $T^{3}(W) \leq T^{3}(W \setminus X) + v(X)$ since every vertex-disjoint triangle in a triangle packing of $W$ either lies in $W - X$ or uses a vertex from $X$. Thus we have

\begin{align*}
 p(F') &= 5v(F') -3e(F') -T^{3}(F') \\
 &\leq 5(v(F) + v(W) -v(X)) - 3(e(F) + e(W) - e(X)) - T^{3}(F) - T^{3}(W \setminus X) \\
 & = p(F) + p(W) -5v(X) +3e(X) +T^3(W) - T^3(W \setminus X) \\
 & \leq p(F) + p(W) -4v(X) + 3e(X).
\end{align*}
\aftermath
\end{proof}

\begin{figure}
    \centering
\begin{tikzpicture}
\node[bluevertexv2] at (0,-.75) (v1) {};
\node[redvertexv2] at (-2.25,.5) (v2) {};
\node[greenvertexv2] at (-1,.5) (v3) {};
\node[bluevertexv2] at (-1.25,1.5) (v4) {};
\draw[thick,black] (v1)--(v2)--(v3)--(v4)--(v2);
\draw[thick,black] (v4)--(v3);
\draw[thick,black] (v1)--(v3);
\draw[dashed, rotate around= {130:(-1,.75)}] (-1.5,1) ellipse (2cm and 1.5cm);
\node[smallwhite] at (-1.6,-.75) (Dummy1) {$F$};
\draw[dashed] (-3,-1) rectangle (2.5,2);
\node[smallwhite] at (0,2.3) (Dummy2) {$F'$};

\node[blackvertexv2] at (2.25,.5) (v5) {};
\node[blackvertexv2] at (1,.5) (v6) {};
\node[blackvertexv2] at (1.25,1.5) (v7) {};
\draw[thick,black] (v1)--(v5)--(v6)--(v7)--(v5);
\draw[thick,black] (v1)--(v6);
\draw[thick,black] (v6)--(v7);
\draw[thick,black] (v7)--(v4);

\begin{scope}[xshift =6cm]
\node[redvertexv2] at (-2.25,.5) (v2) {};
\node[greenvertexv2] at (-1,.5) (v3) {};
\node[bluevertexv2] at (-1.25,1.5) (v4) {};
\draw[thick,black] (v2)--(v3)--(v4)--(v2);
\draw[thick,black] (v4)--(v3);
\draw[dashed, rotate around = {160: (.75,1)}] (.75,1) ellipse (2.4cm and .9cm);

\node[blackvertexv2] at (2.25,.5) (v5) {};
\node[blackvertexv2] at (1,.5) (v6) {};
\node[blackvertexv2] at (1.25,1.5) (v7) {};
\draw[thick,black] (v5)--(v6)--(v7)--(v5);
\draw[thick,black] (v6)--(v7);
\draw[thick,black] (v7)--(v4);
\draw[thick,black] (v6)--(v4);
\draw[thick,black] (v5)--(v4);
\node[smallwhite] at (0,2.4) (Dummy3) {$W$};
\end{scope}
\end{tikzpicture}
\caption{On the left we have a $4$-critical graph $M$, and a subgraph $F$ of $M$ with a $3$-colouring $\phi$. On the right we have the quotient graph $M_{\phi}[F]$, (here simply identifying the two blue vertices) and a $4$-critical subgraph $W$. We see that $W$ extends to $F'$, which in this case is the entire graph $M$, and that the source $X$ of $W$ is the single blue vertex in $W$. A simple calculation shows that $p(F) = 4$, $p(F') = 0$, $p(W)=1$, and $p(X) = 5$. Note that $0 = p(F') \leq p(F) + p(W) -5v(X) +3e(X) + T^{3}(W) - T^{3}(W \setminus X) = 0$. Thus this gives an example where Lemma \ref{potentialextensionlemma} is tight.}
\label{potentialextensionexample}
\end{figure}
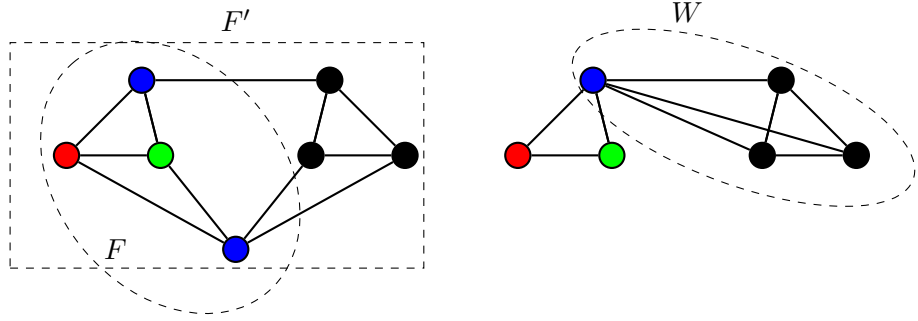

\section{Properties of a minimum counterexample}
\label{mincounterexamplesection}
In this section we prove lemmas regarding the structure of a vertex-minimum counterexample to Theorem \ref{maintheorem}. To that end, for this entire section, let $G$ be such a vertex-minimum counterexample. By Corollary \ref{notinb} and Theorem \ref{KostochkaYanceytight}, we have that $p(G) \geq -1$. We start off with a simple observation.

\begin{obs}
\label{not4ore} The graph $G$ is not $4$-Ore. 
\end{obs}

\begin{proof}
Observe that if $G$ is $4$-Ore, then by Theorem \ref{KostochkaYancey4}, $p(G) = 2 - T^{3}(G)$. If $T^{3}(G) \geq 4$, then $p(G) \leq -2$, contradicting that $p(G) \geq -1$. All other cases are covered as special cases of Theorem \ref{maintheorem}. 
\end{proof}

We now prove what is usually called a ``gap" lemma for potential method proofs.
\begin{lemma}
\label{potentiallemma}
If $F$ is a subgraph of $G$ with $v(F) < v(G)$, then $p(F) \geq 3$. Further, $p(F) \geq 4$ unless one of the following occurs: $G \setminus F$ is a triangle of degree $3$ vertices, or $G \setminus F$ is a vertex of degree $3$, or $G$ contains a kite.
\end{lemma}

\begin{proof}
Suppose not. Let $F$ counterexample that maximizes $v(F)$, and subject to that, minimizes $p(F)$.  Observe that $F$ is an induced subgraph, as adding edges reduces the potential. Observing that $p(K_{1}) = 5$, $p(K_{2}) = 7$, $p(P_2) = 9$ (where $P_2$ is the path of length two), and $p(K_{3}) = 5$, we may assume that $v(F) \geq 4$. Let $\phi$ be a $3$-colouring of $F$, and let $F', W,X$ be an extension, extender, and source of $G_{\phi}[F]$ respectively.

First we deal with the case where $F' \neq G$. Note $v(F') > v(F)$, by the definition of extension. We claim that $p(F') \leq p(F)$. By the Potential-Extension Lemma (Lemma \ref{potentialextensionlemma}), we have
\[p(F') \leq p(F) +p(W) - 4v(X) +3e(X).\]
Note that since $v(F) \geq 4$, it follows that $W$ is a smaller $4$-critical graph than $G$. Therefore $p(W) \leq 1$. Further $v(X) \geq e(X)$. Therefore it follows that $p(F') < p(F)$, and since $F'$ both has more vertices than $F$ and smaller potential, implies we should have taken $F'$ as our counterexample, a contradiction. 
Therefore we may assume that $F' = G$. Recall that as $G$ is a counterexample to Theorem \ref{maintheorem}, we have that $p(G) \geq -1$.

 We split into cases depending on what $W$ is.\\
\textbf{Case 1: $W = K_{4}$}.\\
First suppose that $v(X) =1$. Then by the Potential-Extension Lemma,  $-1  \leq p(G) \leq p(F) + \linebreak p(W)  -5v(X) +3e(X) +T^{3}(W) - T^{3}(W/X) = p(F)+1 -5$ which implies that $p(F) \geq 3$. If $p(F) \geq 4$, then we are done, so we may assume that $p(F) = 3$.  Observe that $G \setminus F$ contains  three vertices, and they must induce a triangle as $W -X$ is a triangle. Let $T$ be the triangle in $G \setminus F$. Then $-1 \leq p(G) \leq p(F) +5 - 3e(T,F)$, where $e(T,F)$ is the number of edges with one endpoint in $T$ and one endpoint in $F$. If $e(T,F) \geq 4$, then we have $-1 \leq p(F) -7$ so $p(F) \geq 6$. Hence $e(T,F) \leq 3$. But since $G$ is $4$-critical, the minimum degree is $3$. Hence $e(T,F) \geq 3$ and $T$ is a triangle of degree $3$ vertices.

If $v(X) = 2$, then by the Potential-Extension Lemma we have $ -1 \leq p(F) +1 - 7 +1$  which implies that $p(F) \geq 4$.

If $v(X) =3$, then $ -1 \leq p(F) +1 -6 +1$  which gives $p(F) \geq 3$. If further $p(F) \geq 4$, then we are done. So we may assume that $p(F) =3$. Note that $G \setminus F$ is a single vertex $v$. We will argue that this vertex has degree $3$. Note that $-1 \leq p(G) \leq p(F) + 5 - 3 \deg(v) = 8 - 3\deg(v)$. As $G$ is $4$-critical, $\deg(v) \geq 3$; and since $-1 \leq 8-3\deg(v)$, we have that $\deg(v) \leq 3$. Hence $\deg(v) =3$. \\
\textbf{Case 2: $W$ is $4$-Ore with $T^{3}(W) =2$.}\\
If $v(X) = 1$, then $-1 \leq p(F) +0 -5 + 1$ so $p(F) \geq 3$. As $W$ is $4$-Ore with $T^{3}(G) =2$, by Lemma \ref{K4e} either $W$ contains two vertex-disjoint kites, or $W = M$. If $W \neq M$, then $G$ contains a kite, so we are done. If $W = M$ and the vertex in $X$ is not the unique degree $4$ vertex in the Moser spindle, then again $G$ contains a kite and so we are done.  Otherwise, $X$ contains only the unique vertex of degree $4$ in the Moser spindle, and in this case $T^{3}(W-X) = 2$. From the Potential-Extension Lemma, we have that $-1 \leq p(F) +0 -5$ which implies that $p(F) \geq 4$. Thus it follows that either $G$ contains a kite or $p(F) \geq 4$, and we are done.

If $v(X) \in \{2,3\}$, we claim that $-1 \leq p(F) - 6+1$. To see this, note that by Lemma \ref{K4e}, $W$ contains two edge-disjoint kites that share at most one vertex. It follows from this that $T^3(W\setminus X) \in \{1,2\}$. Then $p(F) \geq 4$, so we are done.  \\
\textbf{Case 3: $W = W_{5}$.} \\ 
Recall that $X$ is assumed to induce a clique. Observe that deleting a clique of size at most three from $W_{5}$ may result in a triangle-free graph. Hence, for any $v(X) \in \{1,2,3\}$, we have 
$-1 \leq p(F) -1 - 5 +1$. Therefore, $p(F) \geq 4$ as desired. \\
\textbf{Case 4: $W \in \mathcal{B}$.} \\
Note that in this case $T^3(W) = 2$. If $v(X) = 1$, then $T^3(W\setminus X) \geq 1$, and so $-1 \leq p(F) -1 -5 +1$ which gives $p(F) \geq 4$. If $v(X) \in \{2,3\}$, then $-1 \leq p(F) -1 -6 +2$, which gives $p(F) \geq 4$, so we are done.\\
\textbf{Case 5: $W$ is $4$-Ore with $T^{3}(W) = 3$.}\\
If $v(X) =1$, then $T^3(W\setminus X) \geq 2.$ Thus $-1 \leq  p(F) -1 -5 +1$ which gives $p(F) \geq 4$, so we are done.
If $v(X) \in \{2,3\}$, then $-1 \leq p(F) - 1 -6+ 2$, which gives $p(F) \geq 4$. (Note that in the $v(X) = 3$ case, we are using the fact that $T^3(W \setminus X) \geq 1$ (see Observation \ref{deletingaclique})).\\
\textbf{Case 6: All other cases.}\\
If $v(X) =1$, then $-1 \leq p(F) -2 -5 +1$ which gives $p(F) \geq 4$. 
If $v(X) = 2$, then $-1 \leq p(F) -2 -7 +2$ which gives $p(F) \geq 6$.
If $v(X) =3$, then $-1 \leq p(F) - 2 - 6 +3$ which gives $p(F) \geq 4$. 
This is all possible cases, so the result follows. 
\end{proof}

We now strengthen the above result.
\begin{lemma}\label{k4-e}
The counterexample $G$ does not contain $K_{4}-e$ as a subgraph.  
\end{lemma}

\begin{proof}
Suppose not. Let $F$ be a $K_{4}-e$ subgraph in $G$ chosen to maximize the number of degree $3$ vertices in $F$ that are also degree $3$ vertices in $G$. Denote $V(F)$ by $\{w,x,y,z\}$, where $xy \not \in E(G)$. Note that since $G \neq K_{4}$, we have that $F$ is an induced subgraph. We claim that $x$ and $y$ have no common neighbours aside from $w$ and $z$. Suppose not, and let $u$ be a common neighbour of $x$ and $y$ with $u \not \in \{w,z\}$. By $4$-criticality, $G-ux$ has a $3$-colouring, say $\phi$. Then $\phi(u) = \phi(x)$ as otherwise $G$ has a $3$-colouring. Notice in any $3$-colouring of $F$, $\phi(x) =\phi(y)$. But then $uy \in E(G)$ and $\phi(u) = \phi(y)$, a contradiction. Hence $x$ and $y$ have no common neighbours outside $\{w, z\}$.

Fix any $3$-colouring of $F$, and let $F', W$ and $X$ be an extension, extender, and source of $F$. By the Potential-Extension Lemma, we have 
\[p(F') \leq p(F) +p(W) -5v(X) +3e(X) +T^{3}(W) - T^{3}(W \setminus X).\]
Observe that $p(F) = 4$. Throughout the proof of this lemma, we let $x^{*}y$ be the vertex obtained by identifying $x$ and $y$. Note that since $W \not \subseteq G$, it follows that $x^{*}y \in X$. Moreover, note that $W$ is a smaller $4$-critical graph than $G$, and hence by the minimality of $G$ the potential of $W$ is described by one of the outcomes of Theorem \ref{maintheorem}. We now break into cases, depending on the outcome of Theorem \ref{maintheorem} applied to $W$.

\textbf{Case 1: $W = K_{4}$.}\\
First suppose $v(X) =3$. In this case, $w, z,$ and one of $y$ and $x$ share a common neighbour, and so  $G$ contains a $K_4$. This is a contradiction, as $K_4$ is $4$-critical and $G \neq K_4$. 

 Now suppose $v(X) =2$. Then without loss of generality let $X = \{z,x^{*}y\}$. Then there is a subgraph $H$ of $G$ where $V(H) = V(F) \cup \{u,u'\}$ and there are edges $u'z$, $u'u$, $u'x$, $uy$, $uz$ and $E(F)$. But this subgraph is $W_5$, which is 4-critical; so $G =W_5$, a contradiction. 

Finally suppose that $v(X) =1$. Then similarly to the above argument, $G$ is isomorphic to the Moser spindle, and we are done. \\
\textbf{Case 2: $W$ is $4$-Ore with $T^{3}(W) =2$.}\\
First suppose that $v(X) =1$. Then it follows that $G$ contains a subgraph $H$ that is the Ore composition of $W$ and $K_{4}$. Since $H$ is $4$-critical, $G = H$. This implies that $G$ is $4$-Ore, contradicting Observation \ref{not4ore}.

Now suppose that $v(X) =2$. By Lemma \ref{K4e}, $W$ contains two edge-disjoint kites that share at most a vertex. Thus $T^3(W \setminus X) \geq 1$. By the Potential-Extension Lemma, we have $p(F') \leq 4 + 0 -7 +1$ which gives $p(F') \leq -2$. If $F' \subset G$, this contradicts Lemma \ref{potentiallemma}. If $F' = G$, this contradicts the fact that $p(G) \geq -1$. 

Now suppose $v(X) =3$. Note that by Lemma \ref{K4e}, $W$ contains two edge-disjoint kites that share at most one vertex. Thus $T^3(W \setminus X) \geq 1$. By the Potential-Extension Lemma, we have $p(F') \leq 4 + 0 -6  +1$ which implies that $p(F') \leq -1$. By Lemma \ref{potentiallemma}, since $F' \subseteq G$, it follows that $F' = G$. Thus $G$ is obtained from $W$ by unidentifying $x^{*}y$ into $x$ and $y$. Note that since $x$ and $y$ have no common neighbours aside from $w$ and $z$ and since every vertex in $G$ has degree at least three, $x^{*}y$ has degree at least four in $W$. First consider the case where $W = M$. Then $x^{*}y$ is the unique vertex of degree $4$ in $M$. As $G$ is obtained by unidentifying $x^{*}y$ to $x$ and $y$, it follows that either $G$ has a vertex of degree at most $2$ (contradicting the fact that $G$ is $4$-critical), or that $G$ is $3$-regular. But if $G$ is $3$-regular, since $G \neq K_{4}$, we have that $G$ is $3$-colourable by Brook's Theorem. Hence $W \neq M$. Therefore by Lemma \ref{K4e}, $W$ contains two vertex-disjoint kites. Since $G$ is obtained by unidentifying $x^{*}y$, this implies that $G$ contains a kite. But note that since $W$ is $4$-critical, both $w$ and $z$ have degree at least three in $W$; and thus in $G$ after unidentifying $x^{*}y$, both $w$ and $z$ have degree at least four. But since $G$ contains a kite, this kite contradicts our choice of $F$, since we picked $F$ to contain the largest number of vertices which are degree $3$ in the $K_{4}-e$ subgraph and in $G$. \\
\textbf{Case 3: $W = W_{5}$.}\\
If $v(X) = 1$, then $G$ is an Ore composition of $K_{4}$ and $W_{5}$ (as this is a $4$-critical graph). Observe every split of $W_{5}$ contains at least one triangle avoiding at least one of $x$ or $y$, and further the deletion of any edge of $W_{5}$ leaves a triangle containing neither $x$ nor $y$. It follows that an Ore composition of $K_{4}$ and $W_{5}$ has $p(G) \leq -2$, which contradicts $G$ the fact that $p(G) \geq -1$. 

If $v(X) \in \{2,3\}$, then by the Potential-Extension Lemma we have $p(F') \leq 4 -1 -6+1$ which gives $p(F') \leq -2$. If $F' \subset G$, this contradicts Lemma \ref{potentiallemma}. If $F' = G$, this contradicts the assumption that $p(G) \geq -1.$\\
\textbf{Case 4: $W \in \mathcal{B}$.}\\
If $v(X) =1$, then $G$ is the Ore composition of $W$ and $K_{4}$, as such a graph is $4$-critical. Note that if $G \in \mathcal{B}$, then $G$ is not a counterexample, a contradiction. Hence $T^{3}(G) \geq 3$, in which case $p(G) \leq -2$ (note we cannot have $T^{3}(G) \leq 1$, as there is a triangle in $K_{4}$ after deleting any vertex), contradicting that $G$ is a counterexample. If $v(X) \in \{2,3\}$,  then by the Potential-Extension Lemma we have $p(F') \leq 4 -1 -6 +1$ which gives $p(F') \leq -2$. As in Case 3, this leads to a contradiction. \\
\textbf{Case 5: $W$ is $4$-Ore with $T^{3}(G) =3$.}\\
If $v(X) =1$, then $G$ is $4$-Ore, contradicting Observation \ref{not4ore}. If $v(X) =2$, then  $p(F') \leq 4 -1 -7 + 2$ and $p(F') \leq -2$. As in Cases 3 and 4, this leads to a contradiction. 

So $v(X) =3$. In this case we claim $F'$ is all of $G$. If not, take any $3$-colouring $\psi$ of $F'$ (which exists by $4$-criticality). As $x$ and $y$ get the same colour in this $3$-colouring, this implies when we identify $x$ and $y$, we get a $3$-colouring of $W$, contradicting that $W$ is $4$-critical. Hence $F' = G$.

If $T^{3}(W \setminus X) \geq 2$, then by the Potential-Extension Lemma we have
$p(F') \leq 4 - 1 - 6 +1$, which gives $p(F') \leq -2$, a contradiction. 

Therefore by Lemma \ref{deletingatriangle} it follows that $W-X$ contains a kite. 

Let $K$ be the kite in $W-X$, with spar $st$. We claim there is at most one edge from $F$ to $K$: otherwise, $p(G[V(F) \cup V(K)]) \leq 5(8)-3(10+2)-2 = 2$, contradicting Lemma \ref{potentiallemma}. (Note trivially $G[V(F) \cup V(K)] \neq G$, since $T^3(W) = 3$ but $T^3(G[V(F) \cup V(K)]) = 2$.) Thus at least one of $s$ and  $t$ has degree $3$ in $G$. It now suffices to argue that $w$ and $z$ do not have degree $3$ in $G$, thus contradicting our choice of $F$.

To see this, note that since $W$ is $4$-critical, both $w$ and $z$ have degree at least three in $W$. But $G$ is obtained from $W$ by unidentifying $d$ into the vertices $x$ and $y$. As $x$ and $y$ share $w$ and $z$ as neighbours, $w$ and $z$ have degree at least four in $G$. Thus $K$ contradicts our choice of $F$.

\textbf{Case 6: All other cases.} \\ 
In this case, $p(W) \leq -2$. If $v(X) =1$, then $p(F') \leq 4-2  -5 +1 \leq -2$, a contradiction.

If $v(X) =2$, then $p(F') \leq 4 -2 -7 +2 \leq -3$, a contradiction. 

Lastly, assume that $v(X) =3$. In this case, by a similar argument as in Case $5$,  $F' =G$, and thus $G$ is obtained from $W$ by unidentifying $x^{*}y$. Then $T^{3}(G) \geq T^{3}(W)-1$, and thus $p(G) = p(W) +5 -6 +1 \leq -2$, a contradiction. 

\end{proof}

Let $D_{3}(G)$ be the subgraph of $G$ induced by the vertices of degree $3$. Now we will build towards showing that $D_{3}(G)$ is acyclic, and further if a vertex of degree $3$ is in a triangle in $G$, then it is the only vertex of degree $3$ in this triangle. 

\begin{definition}
For an induced subgraph $R$ of $G$ where $R \neq G$, we say $u,v \in V(R)$ are an \textit{identifiable pair} if $R + uv$ is not $3$-colourable.
\end{definition}

\begin{lemma}
\label{noidentifiablepair}
If $R$ is an induced  subgraph of $G$ with $v(R) \leq v(G) -2$ and such that $G \setminus R$ is not a triangle of degree $3$ vertices, then $R$ has no identifiable pair. 
\end{lemma}

\begin{proof}
Suppose not. Let $x$ and $y$ be an identifiable pair in $R$, and consider $R + xy$. As $R + xy$ is not $3$-colourable by definition, there exists a $4$-critical subgraph $W$ of $R + xy$. Note that $xy \in E(W)$, since $G$ does not contain a proper subgraph that is $4$-critical. Moreover, since $T^3(W-xy) \geq T^3(W)-1$, we have that $p(W-xy) \leq p(W) +4$. The hypotheses of this lemma as well as Lemma \ref{potentiallemma} and Lemma \ref{k4-e} imply that $p(W-xy) \geq 4$. If $p(W) \leq -1$, then we obtain a contradiction. If $W = K_{4}$, then $G$ has a  $K_{4}-e$ subgraph, contradicting Lemma \ref{k4-e}. If $W$ is $4$-Ore with $T^{3}(G) = 2$, then by Lemma \ref{K4e}, $G$ contains a $K_{4}-e$ subgraph, again contradicting Lemma \ref{k4-e}. For all other $W$, we have $p(W) \leq -1$, and thus we get a contradiction. 
\end{proof}

For a subgraph $H$, let $N(H)$ be the set of vertices not in $H$ which have a neighbour in $H$. We will need the following well-known consequence of the Gallai-Tree Theorem \cite{GallaiForests}.

\begin{thm}
\label{Independentsettheorem}
If $C$ is a cycle of degree $3$ vertices in a $4$-critical graph, then $v(C)$ is odd, $N(C)$ induces an independent set, and in any $3$-colouring of $G-C$, all vertices in $N(C)$ receive the same colour. 
\end{thm}

\begin{proof}
Recall that every cycle $C$ admits an $L$-colouring from a $2$-list-assignment $L$ unless $C$ is odd and the lists of each vertex are the same. Thus unless the conditions of the theorem occur, we can extend any $3$-colouring of $G-C$ to $C$, contradicting $4$-criticality.  
\end{proof}

\begin{cor}
All cycles in $D_{3}(G)$ are triangles.
\end{cor}

\begin{proof}
Let $C$ be a cycle in $D_{3}(G)$ where $v(C) \geq 5$. If $|N(C)| =1$, then since $G$ has minimum degree $3$, it follows that $G$ is isomorphic to an odd wheel. If $G = W_{5}$, then $G$ is not a counterexample to Theorem \ref{maintheorem}. So we may assume that $v(C) \geq 7$. Note that $v(G) = v(C) +1$, and $e(G) = 2v(C)$. So $p(G) = 5(v(C)+1) -6v(C) -1 = -v(C)+4 \leq -3$ since $v(C) \geq 7$.  Thus $|N(C)| \geq 2$. Then by Theorem \ref{Independentsettheorem}, any pair of vertices in $N(C)$ are an identifiable pair in $G$. This contradicts Lemma \ref{noidentifiablepair} as $v(G-C) < v(G) -3$. 

\end{proof}

\begin{cor}
\label{triangledegrees}
If $T$ is a triangle in $G$, then $V(T)$ does not contain exactly two vertices of degree $3$.
\end{cor}

\begin{proof}
Suppose not. Let $x,y$ and $z$ induce a triangle where $x$ and $y$ are vertices of degree $3$ and $z$ has degree at least four. Let $x'$ and $y'$ be the unique other neighbours of $x$ and $y$ respectively. Note $x' \neq y'$ as otherwise $G$ contains a subgraph isomorphic to $K_{4}-e$, contradicting Lemma \ref{k4-e}. 

If $x'y' \in E(G)$, then any  $3$-colouring of $G-\{x,y\}$ extends to a $3$-colouring of $G$, a contradiction. In particular, every $3$-colouring of $G-\{x,y\}$ gives $x'$ and $y'$ the same colour and hence $x',y'$ are an identifiable pair in $G-\{x,y\}$. 

Therefore $G-\{x,y\} + \{x'y'\}$ contains a 4-critical subgraph, $W$, containing $x'y'$. Moreover, $p(W-x'y') \leq p(W) +4$. By the same argument as in Lemma \ref{noidentifiablepair}, it suffices to show $ H := G \setminus (W -x'y')$ is not a triangle of degree $3$ vertices, or a single vertex of degree $3$. Notice that $H$ is not a vertex of degree $3$, as both $x$ and $y$ are in $V(H)$. We claim $H$ is not a triangle of degree $3$ vertices. If so, then $z \not \in V(H)$, since $\deg(z) \geq 4$. But since $x,y \in V(H)$, it follows that $x$ and $y$ lie in a triangle of degree $3$ vertices. But then as $x,y,z$ induce a triangle, $G$ contains a $K_{4}-e$ subgraph, again contradicting Lemma \ref{k4-e}. 
\end{proof}

\begin{definition}
An \textit{$M$-gadget} is a graph obtained from $M$ by first splitting the vertex $v$ of degree $4$ into two vertices $v_{1}$ and $v_{2}$ such that there is no $K_{4}-e$ in the resulting graph, $N(v_{1}) \cap N(v_{2}) = \emptyset$, and $N(v) = N(v_{1}) \cup N(v_{2})$;  and after this, adding a vertex $v'$ adjacent to only $v_{1}$ and $v_{2}$. We call $v'$ the \textit{end} of the $M$-gadget.
\end{definition}

Note that an $M$-gadget is not necessarily an induced subgraph; this is a property we will exploit later on.

\begin{figure}
\begin{center}
\begin{tikzpicture}
\node[blackvertexv2] at (0,0) (u1) [label = below:$x$] {};
\node[blackvertexv2] at (2,0) (u2) [label = below:$y$]{};
\node[blackvertexv2] at (4,0) (u3) [label = below:$z$]{};
\draw[thick,black] (u1)--(u2)--(u3);
\node[blackvertexv2] at (-2,1.5) (v2) {};
\node[blackvertexv2] at (-1,1.5) (v3) {};
\node[blackvertexv2] at (-1.25,2.5) (v4) {};
\draw[thick,black] (v2)--(v3)--(v4)--(v2);
\draw[thick,black] (v4)--(v3);

\node[blackvertexv2] at (1,1.5) (v5) {};
\node[blackvertexv2] at (0,1.5) (v6) {};
\node[blackvertexv2] at (.25,2.5) (v7) {};
\draw[thick,black] (v5)--(v6)--(v7)--(v5);

\draw[thick,black] (v6)--(v7);
\draw[thick,black] (v7)--(v4);

\node[blackvertexv2] at (-.85,.75) (v8)[label= left:$x'$] {};
\node[blackvertexv2] at (.45,.75) (v9) [label= right:$x''$] {}; 

\draw[thick,black] (u1)--(v8);
\draw[thick,black] (u1)--(v9);
\draw[thick,black] (v8)--(v2);
\draw[thick,black] (v8)--(v6);
\draw[thick,black] (v9)--(v3);
\draw[thick,black] (v9)--(v5);
\end{tikzpicture}
\end{center}
\caption{An example of the $M$-gadget outcome in Lemma \ref{identificationlemma}}
\label{Mgadgetfig}
\end{figure}
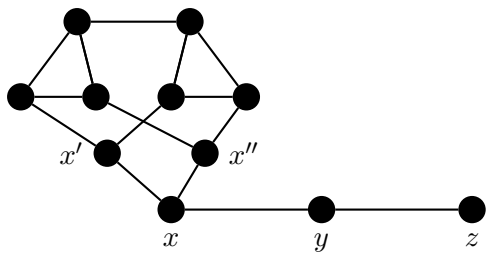

\begin{lemma}
\label{identificationlemma}

Let $C$ be a component of $D_{3}(G)$ with $v(C) \geq 3$. Let $x,y,z \in V(C)$ be such that $xy,yz \in E(G)$ and $xz \not \in E(G)$.  Let $x',x''$ be two neighbours of $x$ which are not $y$. Either $x'x'' \in E(G)$, or $x,x',x''$ lie in an $M$-gadget with end $x$, and this $M$-gadget avoids both $y$ and $z$.
\end{lemma}

\begin{proof}
Suppose not. Then $x'x'' \not \in E(G)$. Let $G'$ be the graph obtained from $G$ by identifying $x'$ and $x''$ to a new vertex $x'''$. Note neither $x'$ nor $x''$ is $z$, since $xz \not \in E(G)$. Moreover, if there exists a $3$-colouring of $G'$, then this $3$-colouring readily extends to $G$. Hence $G'$ is not $3$-colourable. Let $W'$ be a $4$-critical subgraph of $G'$.  In $G'$, let $C'$ be the induced subgraph containing the vertices of $C-x'-x''$. Observe $C'$ is $2$-degenerate (that is, all subgraphs of $C'$ have a vertex of degree at most $2$), since $x$ now has degree $1$ in $C'$,  and all other vertices in $C'$ have degree at most three. Since $W'$ is $4$-critical, it has minimum degree at least $3$ and so this implies that no vertex of $C$ is in $W'$. Since $G$ does not contain a proper 4-critical subgraph, we have that $W' \not \subseteq G$ and thus $x''' \in V(W')$.

Let $W$ be the subgraph of $G$ obtained by taking the subgraph induced by $V(W')-x'''$, and adding vertices $x, x'$ and $x''$ and edges $xx'$, $xx''$, as well as any edge incident to a vertex in $W'-x'''$ and either of $x'$ or $x''$. If $T^{3}(W) = T^{3}(W')-1$, we have $p(W) \leq p(W') +10 -6 +1 = p(W') +5$, and otherwise $p(W) \leq p(W') +4$. 

We claim $G \setminus W$ is not a cycle of degree $3$ vertices. To see this, suppose not: then since $y$ and $z$ are not contained in $W$, it follows that there exists a vertex $w \in V(C)$ such that $yzwy$ is a triangle of degree 3 vertices. Let $w'$ be the neighbour of $w$ that is not $z$ or $y$, and let $z'$ be the neighbour of $z$ that is not $w$ or $y$. Note that $z', w',$ and $x$ are distinct vertices by Lemma \ref{k4-e}. Let $R = G-\{x,y,z,w\}$, and let $\phi$ be a 3-colouring of $R$. Note that since $x$ has only two neighbours in $R$, it follows that $\phi$ extends to $R+x$. If $\phi(x), \phi(z')$, and $\phi(w')$ are not all equal, then $\phi$ extends to a 3-colouring of $G$ by Theorem \ref{triangledegrees}, a contradiction. Thus $\phi(x) = \phi(z') = \phi(w')$; and since $\phi$ is an arbitrary 3-colouring of $R$, we have that in every 3-colouring of $R$, both $w'$ and $z'$ receive the same colour. Since $v(R) = v(G)-4$, this contradicts Lemma \ref{noidentifiablepair}. Thus $G \setminus W$ is not a cycle of degree 3 vertices; and since both $x$ and $z$ are in $G \setminus W$, it follows that $G \setminus W$ is not a single vertex of degree 3.

Thus by Lemma \ref{potentiallemma} and Lemma \ref{k4-e} we have that $p(W) \geq 4$. Note that since $v(W') < v(G)$, it follows from the minimality of $G$ that $W'$ is not a counterexample to Theorem \ref{maintheorem}. Since $4 \leq p(W) \leq p(W')+5$, we have that $p(W') \geq -1$. We break into cases according to the possible outcomes of Theorem \ref{maintheorem} applied to $W'$.\\
\textbf{Case 1: $W' =K_{4}$.}\\
Since $W'$ is obtained by identifying $x'$ and $x''$, this implies that $G$ contains a $K_4-e$ subgraph, contradicting Lemma \ref{k4-e}. \\ 
\textbf{Case 2: $W'$ is $4$-Ore with $T^{3}(G) = 2$.} \\
If $T^{3}(W') =2$ and $W'$ is $4$-Ore, then by Lemma \ref{K4e} if $W'$ is not $M$, then $G$ contains a $K_{4}-e$ subgraph. Again, this contradicts Lemma \ref{k4-e}.  Moreover, since $G$ does not contain a $K_4-e$ subgraph, it follows that if $W'$ is $M$ then $x'''$ is the unique vertex of degree $4$ and the split of $x'''$ back into $x'$ and $x''$ leaves no $K_4-e$ subgraph in $G$. Thus $x$ is the end of an $M$-gadget; and since $\{y,z\} \subseteq V(C)$ and $V(W') \cap V(C) = \emptyset$, this $M$-gadget avoids both $y$ and $z$, as desired. \\
\textbf{Case 3: $W'$ is $4$-Ore with $T^{3}(G) = 3$ or $W' \in \mathcal{B}$.} \\
Suppose either $T^{3}(W') = 3$ and $W'$ is $4$-Ore, or $W' \in \mathcal{B}$. Recall that by Lemma \ref{k4-e}, $G$ does not contain a $K_4-e$ subgraph. Thus by Lemmas \ref{4Oresplit3triangle} and \ref{T8splits}, then either splitting $x'''$ back into $x'$ and $x''$ does not reduce the size of a triangle packing, or in $W \setminus x$ either $x'$ or $x''$ has degree one and is incident to a foundational edge in $G$. The first case gives a contradiction as then $p(W) \leq p(W') + 4$, and $p(W') \leq -1$, which contradicts that $p(W) \geq 4$. Therefore we may assume without loss of generality that $x'$ has degree one in $W'$ after splitting $x'''$ back into $x'$ and $x''$ and that the edge incident to $x'$ is foundational in $W'$. Let the other endpoint of this foundational edge be $y'$. 

Let $W'' =W' - x''' + x''$. Observe that $W''$ is $3$-colourable, as $W'' = W - \{x,x'\}$, implying that $W''$ is a strict subgraph of $G$.

We claim that in every $3$-colouring of $W''$, the vertices $x''$ and $y'$ get the same colour as one another. If not, then since $\deg(x') = 1$ in $(W')^{x'''}$, we have a $3$-colouring of $W'$, which contradicts that $W'$ is $4$-critical. Hence $W''$ contains an identifiable pair. Further, $y,z,x,x' \not \in V(W'')$. Thus we contradict Lemma \ref{noidentifiablepair}.\\
\textbf{Case 4: $W' = W_{5}$.} \\
In this case, regardless of whether $x'''$ is a $3$-vertex or $5$-vertex in $W_{5}$, we have that $W$ contains at least one triangle, and that $v(W) = 8$ and $e(W) = 12$. With this, it follows that $p(W) \leq 8(5) - 3(12) -1 =3$, contradicting Lemma \ref{potentiallemma}, and Lemma \ref{k4-e}.
\end{proof}

\begin{figure}
\begin{center}
\begin{tikzpicture}
\node[blackvertexv2] at (-.25,0) (v1) [label = below:$v_{1}$] {};
\node[blackvertexv2] at (2,0) (v2) [label = below:$v_{2}$] {};
\node[blackvertexv2] at (4,0) (v3) [label = below:$v_{3}$] {};
\node[blackvertexv2] at (6,0) (v4) [label = below:$v_{4}$] {};
\node[blackvertexv2] at (8,0) (v5) [label = below:$v_{5}$] {};
\node[blackvertexv2] at (4,2) (x3) [label = right:$x_{3}$] {};
\draw[thick,black] (v1)--(v2)--(v3)--(v4)--(v5);
\draw[thick,black] (v3)--(x3);

\node[blackvertexv2] at (-.75,2.5) (u2) [label = left:$u_{2}$] {};
\node[blackvertexv2] at (1.1,1.4) (u1) [label = above:$u_{1}$] {};
\draw[thick,black] (u2)--(u1)--(v1)--(u2);
\node[blackvertexv2] at (1.5,4.2) (u4) [label= above:$u_{4}$] {};
\node[blackvertexv2] at (2,2.4) (x2) [label = left:$x_{2}$] {};
\node[blackvertexv2] at (3.3,3.7) (u3) [label =right:$u_{3}$] {};
\draw[thick,black] (u4)--(u3)--(x2)--(u4);
\draw[thick,black] (x3)--(u1);
\draw[thick,black] (u4)--(u2);
\draw[thick,black] (x3)--(u3);
\draw[thick,black] (v2)--(x2);
\draw[thick,dotted, bend left = 30] (x2) to (v1);

\begin{scope}[yshift = -6 cm]
\node[blackvertexv2] at (-.25,0) (v1) [label = below:$v_{1}$] {};
\node[blackvertexv2] at (2,0) (v2) [label = below:$v_{2}$] {};
\node[blackvertexv2] at (4,0) (v3) [label = below:$v_{3}$] {};
\node[blackvertexv2] at (6,0) (v4) [label = below:$v_{4}$] {};
\node[blackvertexv2] at (8,0) (v5) [label = below:$v_{5}$] {};

\draw[thick,black] (v1)--(v2)--(v3)--(v4)--(v5);

\node[blackvertexv2] at (2,2) (x2) [label = right:$x_{2}$] {};
\draw[thick,black] (v2)--(x2);
\node[blackvertexv2] at (-1.9,1.2) (u1) [label = left:$u_{1}$] {};
\node[blackvertexv2] at (-1,1.9) (y1) {};
\node[blackvertexv2] at (-2,3) (y2) {};
\draw[thick,black] (y1)--(y2)--(u1)--(y1);

\node[blackvertexv2] at (0,4.2) (y3) {};
\node[blackvertexv2] at (.25,2.6) (u2) [label= left:$u_{2}$] {};
\node[blackvertexv2] at (1.25,3.3) (y4) {};

\draw[thick,black] (y3)--(y4)--(u2)--(y3);
\draw[thick,black] (u1)--(v1)--(u2);
\draw[thick,black] (y1)--(x2)--(y4);
\draw[thick,black] (y3)--(y2);
\draw[thick,black, dotted, bend right = 30] (u1) to (u2);

\end{scope}

\end{tikzpicture}
\end{center}
\caption{The two $M$-gadgets in Corollary \ref{ind-p4}. The top figure is an $M$-gadget which has end $v_{3}$, and the dotted edge $v_{1}x_{2}$ is a possible outcome of Lemma \ref{identificationlemma} when applied to $v_{2},v_{3},v_{4}$ where $v_{2} =x$. If this occurs, this subgraph has too small potential, contradicting Lemma \ref{potentiallemma}. Otherwise, the second figure occurs, and we have a second $M$-gadget, however in this case we find the additional edge $u_{1}u_{2}$ from the first $M$-gadget, and again contradict Lemma \ref{potentiallemma}.}
\label{Mgadgetfig}
\end{figure}
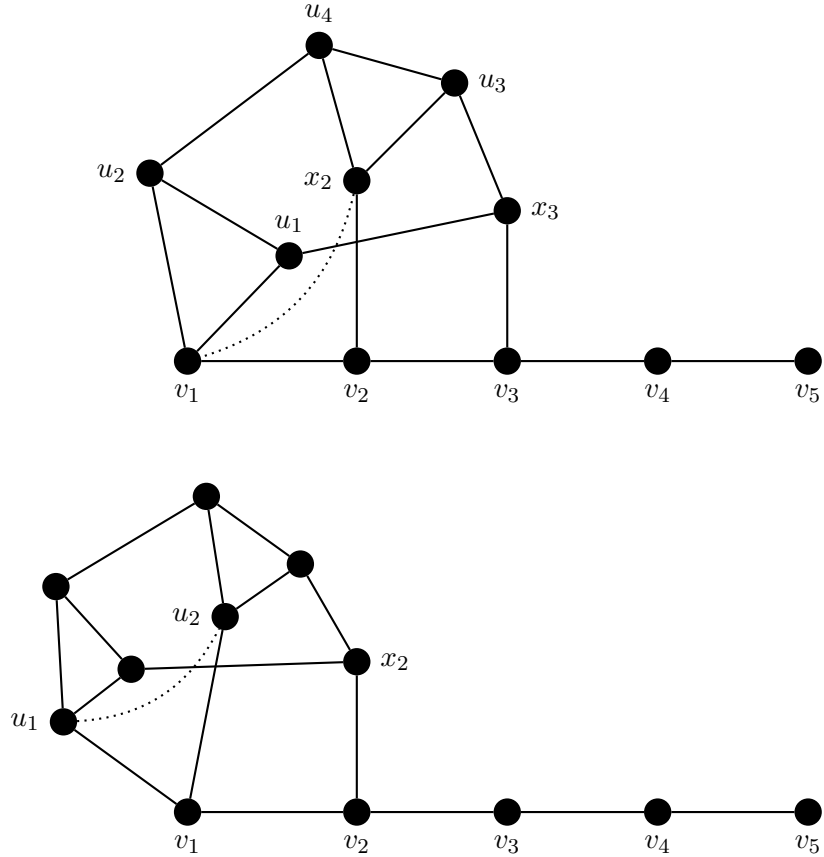

\begin{cor}\label{ind-p4}
The graph $D_{3}(G)$ does not contain an induced path of length four.
\end{cor}

\begin{proof}
Suppose not. Let $C$ be a component of $D_3(G)$ containing an induced path $v_1v_2v_3v_4v_5$.  Let $x_{3}$ be the vertex other than $v_{2}$ and $v_{4}$ that is adjacent to $v_{3}$. By Lemma \ref{k4-e}, $G$ does not contain a $K_{4}-e$ subgraph, and as such either $x_{3}v_{2} \not \in E(G)$, or $x_{3}v_{4} \not \in E(G)$. Without loss of generality, we may assume that $x_{3}v_{2} \not \in E(G)$. Thus by Lemma \ref{identificationlemma}, $v_{3}$ is the end of an $M$-gadget containing $v_{2}$ and $x_{3}$ but not containing $v_{4}$ or $v_{5}$.
Since $\deg(v_2) = 3$, this implies that $v_{1}$ is in a triangle, say $v_{1}u_{1}u_{2}v_{1}$. Let $x_{2}$ be the neighbour of $v_{2}$ which is not $v_{1}$ or $v_{3}$. Similarly, $x_{2}$ is in the $M$-gadget, and is in a triangle $x_{2}u_{3}u_{4}x_{2}$ which does not contain any of the vertices $v_{1},u_{1},u_{2}$. Now we apply Lemma \ref{identificationlemma} to $v_{2}, v_{3}, v_{4}$ with $v_{2}$ playing the role of $x$, and $v_{1}$ and $x_{2}$ playing the role of $x'$ and $x''$. Notice that if $x_{2}v_{1} \in E(G)$, then the as $x_{2}v_{1}$ is not an edge in the $M$-gadget, the graph induced by the vertices in the $M$-gadget has potential at most $1$, contradicting Lemma \ref{potentiallemma}. Therefore by Lemma \ref{identificationlemma} we get that $v_{2}$ is the end of an $M$-gadget, which we refer to as $M'$. We claim the subgraph $M'$ is not induced. First observe that $v_{1} \in V(M')$, since $v_{2}$ has degree $3$. Then it follows that $u_{1},u_{2} \in V(M')$, as $v_{1} \in V(M')$ and $v_{1}$ has degree $3$. Further,  as $v_{2}$ is the degree 2 vertex in the $M$-gadget, the edge $u_{1}u_{2}$ does not lie in $M'$. Let $H' = M' \cup \{u_{1}u_{2}\}$. Then $v(H') =9$ and $e(H') \geq 14$ so $p(H') \leq 45-42-2 =1$. This contradicts Lemma \ref{potentiallemma}.
\end{proof}

\begin{cor}\label{atmostsix}
If $C$ is an acyclic component of $D_3(G)$, then $v(C) \leq 6$.
\end{cor}

\begin{proof}
Let $P$ be a longest path in a component $C$. Note that $P$ is induced, since $C$ is acyclic; and thus by Corollary \ref{ind-p4} we have that $v(P) \leq 4$. First suppose that $P$ contains four vertices; say $P=v_{1}v_{2}v_{3}v_{4}$. Then $v_{2}$ and $v_{3}$ are each adjacent to exactly one vertex not in the path, say $v_{2}'$ and $v_{3}'$ respectively.

Suppose $v_{2}'$ has degree 3 in $G$. If $v_{2}'$ is adjacent to a vertex $u \in D_{3}(G) \setminus \{v_{2}\}$, then $uv_{2}'v_{2}v_{3}v_{4}$ is longer than $P$, contradicting our choice of path.  Applying a similar argument to $v_{3}'$ we see that $v(C) \leq 6$ in this case. Now suppose that $P$ is a path of length $2$; say $P=v_{1}v_{2}v_{3}$. If $v_{2}$ is adjacent to a vertex of degree $3$, say $v_{2}'$, then $v_{2}'$ is not adjacent to another vertex of degree $3$, as otherwise we have a path of length $3$, contradicting our choice of $P$. Hence in this case, $v(C) \leq 4$. Lastly, if the longest path has length at most one, then $v(C) \leq 2$ as desired.
\end{proof}

Now we build towards proving every component of $D_{3}(G)$ is acyclic. 

\begin{lemma}
\label{neighbourshavelargedegree}
Let $T = xyzx$ be a triangle of degree three vertices. Then at most one vertex in $N(T)$ has degree $3$. 
\end{lemma}

\begin{proof}

Suppose not. We claim that $|N(T)| =3$. If $|N(T)| =1$, then $G = K_{4}$ and $G$ is not a counterexample. If $|N(T)| =2$, then let $u,v$ be the two vertices in $N(T)$. By the pigeonhole principle, we may assume $u$ has two neighbours in $T$. Then $T \cup \{u\}$ is a $K_{4}-e$ subgraph of $G$, contradicting Lemma \ref{k4-e}. Let $x',y',z'$ be the vertices in $N(T)$, where $x'$ is adjacent to $x$, $y'$ is adjacent to $y$, and $z'$ is adjacent to $z$. Without loss of generality, suppose that $x'$ and $y'$ both have degree $3$.  Note that $xy' \not \in E(G)$ and similarly $x'z \not \in E(G)$, since $G$ contains no $K_4-e$ subgraph by Lemma \ref{k4-e}. Thus by Lemma \ref{identificationlemma} applied to $x, y, y'$, we have that $x$ is the end of an $M$-gadget not containing $y$ or $y'$.   But now it follows that there are two vertex-disjoint triangles in $G-x-y-z$, and hence $T^{3}(G) \geq 3$. As $G$ is not $4$-Ore, $\text{KY}(G) \leq 1$, and thus $p(G) \leq -2$, a contradiction. 
\end{proof}

\begin{lemma}\label{acyclic}
The graph $D_{3}(G)$ is acyclic.
\end{lemma}

\begin{proof}
Suppose not. Let $T$ be a triangle in $D_{3}$. As $G$  contains no $K_4-e$ subgraph by Lemma \ref{k4-e}, it follows that $|N(T)| =3$. By Theorem \ref{Independentsettheorem} all vertices of $N(T)$ receive the same colour in any 3-colouring of $G-T$. Hence every pair of vertices in $N(T)$ are an identifiable pair in $G-T$. Let $R := G-T$. Let $x,y$ be two vertices in $N(T)$, such that $y$ is adjacent to a vertex $z$ in $T$. Observe that in $R + xy$, we have a $4$-critical graph $W$, and since $T^3(W-xy) \geq T^3(W)-1$, we have that
\begin{equation}\label{p(W-xy)}
    p(W-xy) \leq p(W) +4.
\end{equation} If $W = K_{4}$, then $G$ has a  $K_{4}-e$ subgraph, contradicting Lemma \ref{k4-e}. If $W$ is $4$-Ore with $T^{3}(G) = 2$, then by Lemma \ref{K4e}, again $G$ contains a $K_{4}-e$ subgraph,  contradicting Lemma \ref{k4-e}. If $p(W) \leq -2$, then we obtain a contradiction to Lemma \ref{potentiallemma}. Further, we can assume that $W \neq W_{5}$ since otherwise $p(W-xy) = 5(6)-3(9)-1 = 2$, again contradicting Lemma \ref{potentiallemma}.
Additionally, if $W -xy \neq R$, then we obtain a contradiction to Lemma \ref{potentiallemma} when $p(W) \leq -1$. Thus we can assume that $W=R +xy$ and that either $W$ is $4$-Ore with $T^{3}(W)=3$, or $W \in \mathcal{B}$. 

First assume that $W$ is $4$-Ore with $T^{3}(W) = 3$. Now consider splitting $x$ into two vertices $x_{1}$ and $x_{2}$ such that $\deg(x_{1}) = 1$ and $x_{1}$ is only adjacent to $y$. Let $W^{x}$ denote this graph. Note that $W^x$ is isomorphic to a subgraph $H$ of $G$: an isomorphism is given by $g:V(W^x)\longrightarrow H$ where $g$ is the identity for all $v \in V(W^x)\setminus \{x_1,x_2\}$, where $g(x_1) = z$, and where $g(x_2) = x$. By Lemma \ref{4Oresplit3triangle}, either $W^{x}$ has $T^{3}(W^{x}) \geq T^3(W)$; or  $W^{x}$ contains a $K_{4}-e$; or $xy$ is a foundational edge in $W$. If $W^{x}$ has $T^{3}(W^{x}) \geq T^3(W)$, then
since $T^3(W^x) = T^3(W^x-x_1y) = T^3(W-xy)$ it follows that  Equation \ref{p(W-xy)} can be strengthened to $p(W-xy) \leq p(W) + 3$. Since $p(W) = -1$ and $W-xy \subset G$, this contradicts Lemma \ref{potentiallemma}.
If $W^{x}$ contains a $K_{4}-e$, then $G$ contains a $K_{4}-e$, contradicting Lemma \ref{k4-e}. Therefore we can assume that $xy$ is a foundational edge, and by Lemma \ref{foundationaledgesin4Ore} such an edge is the spar of a kite. Thus in $W-xy$, both $x$ and $y$ have degree two, which implies that in $G$, both $x$ and $y$ have degree $3$. But this contradicts Lemma \ref{neighbourshavelargedegree}.

Therefore we can assume that $W$ is in $\mathcal{B}$. Then $xy$ is a foundational edge, as otherwise by Lemma \ref{4Oresplit3triangle} either $G-xy$ contains a $K_{4}-e$ subgraph, contradicting Lemma \ref{k4-e}, or as above we can strengthen Equation \ref{p(W-xy)} and obtain a contradiction. If $W \neq T_{8}$, then by Lemma \ref{foundationaledgesinB} we have that $xy$ is the spar of a kite. Then in $W-xy$, both $x$ and $y$ have degree two, which implies that in $G$, both $x$ and $y$ have degree $3$, contradicting Lemma \ref{neighbourshavelargedegree}. Therefore $W = T_{8}$. As $W = R +xy = G-T+xy$, our entire graph is $T_{8} - u_{1}u_{2} + T$. In this case,  we label the vertices of $T$ by setting $T = v_1v_2v_3v_1$. We may assume without loss of generality, $v_1$ is adjacent to $u_1$, and $v_2$ is adjacent to $u_2$. Moreover, by Theorem \ref{Independentsettheorem}, the neighbour of $v_3$ outside of $\{v_1, v_2\}$ forms an independent set with $\{u_1, u_2\}$. It follows that the third edge incident with $v_3$ is incident with a vertex in $\{u_6, u_7, u_8\}$. It is easy to verify that the resulting graph is 3-colourable. As these are all the cases, it follows that $D_{3}(G)$ is acyclic. 
\end{proof}

Lemma \ref{acyclic} and Corollary \ref{atmostsix} imply the following.
\begin{cor}\label{nomorethansix}
Every component in $D_{3}(G)$ has at most six vertices.
\end{cor}

\section{Discharging}
\label{dischargingsection}
In this section we provide the discharging argument which shows that a vertex-minimum counterexample to Theorem \ref{maintheorem} does not exist. We start by showing that there exists a component of $D_{3}(G)$ with at least three vertices. Note that though the proof of Lemma \ref{bigcomp} uses discharging, what follows is not the main discharging argument in the paper.

\begin{lemma}\label{bigcomp}
There exists a component of $D_{3}(G)$ with at least three vertices.
\end{lemma}

\begin{proof}
Suppose not. Note that $D_3(G)$ is bipartite. Let $F$ be the subgraph of $G$ with $V(F) = V(G)$ and $E(F) = \{ xy  \in E(G) \, | \, \deg(x) \geq 4 \text{ and } \deg(y) \geq 4\}$.

\begin{claim}
\label{F2edges}
The graph $F$ has $e(F) \geq 2$. If $e(F) =2$, the two edges $e_{1}$ and $e_{2}$ in $F$ do not share an endpoint.
\end{claim}

\begin{proof}
Suppose not. If $F$ is an independent set, then as $D_{3}(G)$ is bipartite, we get that $V(G)$ can be partitioned into three independent sets. This implies that $G$ is $3$-colourable, a contradiction. Now suppose that $F$ contains precisely one edge $e = xy$.Then $F-y$ is an independent set. We claim that $D_{3}(G) \cup \{y\}$ is bipartite. To see this, suppose for a contradiction that $C$ is an odd cycle in $D_3(G) \cup \{y\}$. Note that $D_{3}(G) \cup \{y\}$ is triangle-free, since no triangle in $G$ contains exactly two vertices of degree $3$ by Corollary \ref{triangledegrees}. Thus $C$ has at least five vertices. Since each component of $D_{3}(G)$ has at most a single edge,  it follows that $C$ contains at least two vertices not in $D_3(G)$ \textemdash a contradiction, since $C \subseteq D_3(G) \cup \{y\}$. Thus $D_{3}(G) \cup \{y\}$ is bipartite, which implies that $G$ is $3$-colourable, a contradiction.  

From the above, $F$ has at least two distinct edges $e_1=xy$ and $e_2=y'z$. It remains to show that $e_1$ and $e_2$ do not share an endpoint. To see this, suppose not: suppose without loss of generality that $y = y'$. In this case, $F-y$ is an independent set, and by the same argument as the previous case, $D_{3}(G) \cup \{y\}$ is bipartite. This implies $G$ is $3$-colourable, a contradiction. 
\end{proof}
\begin{claim}
\label{isolatedvertexormanyedges}
Either $e(F) \geq 3$, or $e(F) = 2$ and there is a component in $D_{3}(G)$ that is an isolated vertex.
\end{claim}
\begin{proof}
By Claim \ref{F2edges}, the only case we need to consider is the one where $F$ has exactly two edges $e_{1} = xy$ and $e_{2} = uv$ where $\{x,y\} \cap \{u,v\} = \emptyset$. In this case, we aim to show that there is a component in $D_3(G)$ that is an isolated vertex. Observe that $F -y -v$ is an independent set. If $D_{3}(G) \cup \{y,v\}$ is a bipartite graph, then by the same argument as before we find that $G$ is $3$-colourable, a contradiction. Thus $D_{3}(G) \cup \{y,v\}$ contains an odd cycle, and since no triangle contains precisely two vertices of degree $3$ by Corollary \ref{triangledegrees}, there are no triangles in $D_{3}(G) \cup \{y,v\}$ (as $yv \not \in E(F)$). Further $D_{3}(G) \cup \{y,v\}$ does not contain an odd cycle of length at least $7$, as any such cycle contains at least three vertices from $V(G) -V(D_{3}(G))$. Therefore the only odd cycles in $D_{3}(G) \cup \{y,v\}$ are $5$-cycles.  Let $\{H_1, \ldots, H_t\}$ be the set of all $5$-cycles in $D_{3}(G) \cup \{y,v\}$. Recall that by assumption every component of $D_3(G)$ contains at most one edge. Thus we may assume that for each $i \in \{1,2, \ldots, t\}$, we have that $H_i=a_{i,1}a_{i,2}va_{i,3}ya_{i,1}$, where $\{a_{i,1}, a_{i,2}, a_{i,3}\} \subset V(D_3(G))$. If there exists an $i \in \{1, \ldots, t\}$ such that $a_{i,3}$ is an isolated vertex in $D_{3}(G)$, then we are done. Therefore we may assume that for each $i \in \{1, \dots, t\}$,  $a_{i,3}$ has a neighbour $a_{i,4} \in V(D_{3}(G))$. Since every vertex in $D_3(G)$ has degree $3$ by definition, it follows that for each $i \in \{1, \dots, t\}$, the neighbourhood of $a_{i,3}$ is precisely $\{y,v,a_{4}\}$. Note that $\{a_{1,3}, \dots, a_{t,3}\}$ is an independent set, since each $a_{i,3}$ is adjacent to $y$ and no triangle contains exactly two vertices of degree $3$ by Corollary \ref{triangledegrees}. Thus $F -y-v \cup \{a_{1,3}, \dots, a_{t,3}\}$ is an independent set, and $D_{3}(G) \cup \{y,v\} - \{a_{1,3}, \dots, a_{t,3}\}$ contains no odd cycles. It follows that $G$ is $3$-colourable, a contradiction.
\end{proof}

We now use discharging to complete the proof of Lemma \ref{bigcomp}. We define $\ch_{i}$ to be the initial charge, and set $\ch_{i}(v) = \deg(v)$ for each vertex $v \in V(G)$. Let each vertex of degree at least four send $\frac{1}{6}$ charge to each neighbour of degree $3$. For each $v \in V(G)$, let $\ch_{f}(v)$ denote the final charge of $v$. Note that all degree $3$ vertices end up with at least $\frac{10}{3}$ final charge, and any degree $3$ vertex which is isolated in $D_{3}(G)$ ends up with $\frac{10}{3} + \frac{1}{6}$ charge. If $v$ has degree at least four, then $\ch_f(v) = \frac{10}{3}$ if and only if $\deg(v) = 4$ and $v$ is adjacent to exactly four vertices of degree $3$. Further, if either of those conditions do not hold, the final charge of $v$ is at least $\frac{10}{3} + \frac{1}{6}$. Therefore for every edge $e=xy \in E(F)$, we have $\ch_{f}(x) \geq \frac{10}{3} + \frac{1}{6}$ and $\ch_{f}(y) \geq \frac{10}{3} + \frac{1}{6}$. Let $i$ denote the number of isolated vertices in $D_{3}(G)$. It follows that
\begin{equation}\label{finalcharges}
 \sum_{v \in v(G)}\ch_{f}(v) \geq \frac{10v(G)}{3} + \frac{e(F)}{3} + \frac{i}{6}.   
\end{equation}

If $e(F) \geq 3$, then we have 
\[2e(G) \geq \frac{10v(G) + 3}{3}.\]
Multiplying each side by $\frac{3}{2}$ gives $3e(G) \geq 5v(G)+\frac{3}{2}$. Thus it follows that
\begin{align*}
    p(G) &\leq \text{KY}(G) \\
    &= 5v(G)-3e(G) \\
    &\leq 5v(G) -\left(5v(G)+\frac{3}{2}\right) \\
    &= -\frac{3}{2}
\end{align*}
Since potential is integral, we get that $p(G) \leq -2$, contradicting that $G$ is a counterexample to Theorem \ref{maintheorem}. 

Therefore by Claim \ref{isolatedvertexormanyedges}, we have that $e(F) =2$ and $i \geq 1$, and so by multiplying Equation \ref{finalcharges} by $\frac{3}{2}$ we get that $3e(G) \geq 5v(G) + \frac{5}{4}$. 

As above, this implies that
\begin{align*}
    p(G) &\leq \text{KY}(G) \\
    &=5v(G)-3e(G) \\
    &\leq 5v(G)-\left(5v(G)+\frac{5}{4}\right) \\
    &=-\frac{5}{4}.
\end{align*}

Since potential is integral, this implies that $p(G) \leq -2$, again contradicting that $G$ is a counterexample.
\end{proof}

Now we proceed with the main discharging argument.
We assign to each vertex $v \in V(G)$ an initial charge $\ch_i(v) = \deg(v)$. We discharge in three steps: in each step, the discharging occurs instantaneously throughout the graph. The final charge will be denoted by $\ch_f$. For $v \in V(G)$, let $i_3(v)$ denote the number of neighbours of $v$ that are isolated vertices in $D_3(G)$, and similarly let $\deg_{3}(v)$ denote the number of neighbours of degree $3$ a vertex $v$ has.

\textbf{Discharging Steps}
\begin{enumerate}
    \item If $u$ is a vertex of degree at least four, $uv$ is an edge, and $v$ is a vertex of degree $3$, then $u$ sends $\frac{3\ch_i(u)-10}{3\deg_3(u)}$ charge to $v$.
    \item If $u$ is an isolated vertex in $D_3(G)$, $u$ sends $\frac{1}{18}$ charge to each adjacent vertex in $G$.
    \item Let $u$ be a vertex of degree at least four, and let $f(u)$ be the total charge received by $u$ in Step 2. If $\deg_3(u) \neq i_3(u)$, then the vertex $u$ sends $\frac{f(u)}{\deg_3(u)-i_3(u)}$ charge to each adjacent vertex of degree $3$ that is not isolated in $D_3(G)$.
\end{enumerate}

We will show that after discharging, the sum of the charges is at least $v(G)\left(\frac{10}{3}\right)$. Note that by the discharging rules, we have immediately that every vertex of degree at least four has final charge at least $\frac{10}{3}$. In light of this, we will focus our attention on the vertices of degree $3$: let $C$ be a component in $D_3(G)$, and let $\ch_f(C) = \sum_{v \in V(C)} \ch_f(v)$.

We note the following.

\begin{obs}\label{onesixth}
If $u$ sends charge to $v$ in Step 1, then $u$ sends $v$ at least $\frac{1}{6}$ charge.
\end{obs}

\begin{claim}\label{isolated}
If $C$ is an isolated vertex, then $\ch_f(C) \geq \frac{10}{3}$.
\end{claim}
\begin{proof}
Let $v \in V(C)$. Note that $\ch_i(v) = \deg(v) = 3$. Since $v$ is isolated in $D_3(G)$, every neighbour of $v$ has degree at least four. Thus by Observation \ref{onesixth}, $v$ receives at least $\frac{1}{6}$ from each of its neighbours in Step 1. Moreover, $v$ returns exactly $\frac{1}{18}$ to each of its neighbours in Step 2. It follows that
\begin{align*}
    \ch_f(v) &\geq 3 + 3\left(\frac{1}{6}\right)-3\left(\frac{1}{18}\right) \\
    &= \frac{10}{3}, \textnormal{ as desired. }
\end{align*} 
\aftermath
\end{proof}

\begin{claim}
\label{pathlength1}
If $C$ is a path of length one, then $\ch_f(C) \geq v(C) \left(\frac{10}{3}\right)$.
\end{claim}
\begin{proof}
Let $C= v_1v_2$. Note that $\ch_i(v_1) = \ch_i(v_2) = 3$, and that by Observation \ref{onesixth}, each of $v_1$ and $v_2$ receive at least $\frac{1}{6}$ from each of their neighbours of degree at least four, and since $V(C) =2$, neither $v_1$ nor $v_2$ sends charge in Step 2. It follows that
\begin{align*}
    \ch_f(C) &\geq \ch_i(v_1) + 2\left(\frac{1}{6}\right) + \ch_i(v_2) + 2\left(\frac{1}{6}\right) \\
    &= 2\left(\frac{10}{3}\right),\textnormal{ as desired. }
\end{align*}
\aftermath
\end{proof}

For the remaining cases, we will make use of the following fact. 
\begin{claim}\label{leaves}
If $v$ is a leaf in a tree $C \subseteq D_3(G)$ with $v(C) \geq 3$, then $v$ receives at least $\frac{4}{9}$ charge from its neighbourhood during Step 1.
\end{claim}
\begin{proof}
As $v$ is a leaf in a tree with at least three vertices, there exists a path $vuw$ in $C$. Let $x$ and $y$ be two neighbours of $v$ which are not $u$. By Lemma \ref{identificationlemma}, either $xy \in E(G)$, or $x$, $v$, and $y$ lie in an $M$-gadget with end $v$ that avoids both $u$ and $w$.
If $xy \in E(G)$, then note that $\deg_3(x) \leq \deg(x) -1$, and likewise $\deg_3(y) \leq \deg(y)-1$. In this case, each $z \in \{x, y\}$ sends $v$ at least  $\frac{\deg(z)}{\deg(z)-1}-\frac{10}{3(\deg(z)-1)}$ charge. Since $\deg(z) \geq 4$, it follows that $v$ receives at least $\frac{2}{9}$ from each of $x$ and $y$, and so at least $\frac{4}{9}$ in total.

If $xy \not \in E(G)$, then $v, x,$ and $y$ lie in an $M$-gadget with end $v$ that avoids both $u$ and $w$. Thus there exist two vertex-disjoint triangles $T$ and $T'$ such that $x$ is adjacent to a vertex $a \in V(T)$ and $a' \in V(T')$, and $y$ is adjacent to a vertex $b \neq a$ in  $V(T)$ and $b' \neq a'$ in $V(T')$. Note by Corollary \ref{triangledegrees} and Lemma \ref{acyclic}, each of $T$ and $T'$ contain at most one vertex of degree $3$. Thus at least two of $\{a, a', b, b'\}$ have degree at least four. Without loss of generality, we may assume that either $a$ and $a'$ have degree at least four, or that $a$ and $b'$ have degree at least four. In the first case, $x$ sends at least $\frac{1}{3}$ to $v$, and $y$ sends at least $\frac{1}{6}$ to $v$. Thus $v$ receives at least $\frac{1}{2}$ from $x$ and $y$. In the second case, each of $x$ and $y$ sends at least $\frac{2}{9}$ to $v$, and so $v$ receives at least $\frac{4}{9}$. Thus $v$ receives at least $\frac{4}{9}$ charge from its neighbourhood.
\end{proof}

\begin{claim}\label{p2}
If $C= v_1v_2v_3$ is a path of length $2$, then $\ch_f(C) \geq v(C) \left(\frac{10}{3}\right)$.
\end{claim}
\begin{proof}
By Claim \ref{leaves}, each of $v_1$ and $v_3$ receives at least $\frac{4}{9}$ units of charge from its neighbourhood during Step 1. Moreover, by Observation \ref{onesixth}, $v_2$ receives at least $\frac{1}{6}$ units of charge during Step 1. Thus 
\begin{align*}
    \ch_f(C) &\geq \ch_i(v_1) + \frac{4}{9} + \ch_i(v_2) + \frac{1}{6} + \ch_i(v_3) + \frac{4}{9} \\
    &=\frac{181}{18} \\
    &> 3\left(\frac{10}{3}\right),\textnormal{\ as desired. }
\end{align*}
\aftermath
\end{proof}

\begin{claim}\label{star}
If $C$ is a star with four vertices, then $\ch_f(C) \geq v(C) \left( \frac{10}{3} \right)$.
 \end{claim}
 \begin{proof}
By Claim \ref{leaves}, each leaf in $C$ receives at least $\frac{4}{9}$ from its neighbourhood during Step 1 of the discharging process.  Moreover, each $u \in V(C)$ has $\ch_i(u) = 3$. Thus it follows that 
\begin{align*}
  \ch_f(C) &\geq \left(\frac{4}{9} + 3\right) + \left(\frac{4}{9} + 3\right) + \left(\frac{4}{9} + 3\right) + 3 \\
  &=12 + \frac{4}{3} \\
  &= 4\left(\frac{10}{3}\right),\textnormal{\ as desired. }
\end{align*}
\aftermath
\end{proof}

For the remaining cases, we will need the following lemma.

\begin{lemma}\label{betterleaves}
Let $v$ be a leaf in a tree $C \subseteq D_3(G)$ with $v(C) \geq 3$, and let $u, w$ be the neighbours of $v$ that are not contained in $C$. Suppose $u, w$ are contained in an $M$-gadget with end $v$.  At the end of the discharging process, $v$ will have received at least $\frac{1}{2}$ charge from its neighbours.
\end{lemma}

\begin{proof}
By the structure of $M$-gadgets, there exist two vertex-disjoint triangles $T$ and $T'$ such that $u$ is adjacent to vertex $a$ in $T$ and $a'$ in $T'$, and such that $w$ is adjacent to $b \neq a$ in $T$ and $b' \neq a'$ in $T'$. First, we note that if either $\deg(u) \geq 5$ or $\deg(w) \geq 5$, we are done. To see this, suppose without loss of generality that $\deg(u) \geq 5$. Note that by Corollary \ref{triangledegrees}, at most one vertex in $T$ and at most one vertex in $T'$ has degree $3$. If both $b$ and $b'$ have degree at least four, then in Step 1 $u$ sends $v$ at least $\frac{1}{3}$ and $w$ sends $v$ at least $\frac{1}{3}$. If $a$ and $a'$ have degree at least four, then in Step 1 $u$ sends $v$ at least $\frac{5}{9}$. Finally, if $a$ and $b'$ have degree at least four, then in Step 1 $u$ sends $v$ at least $\frac{5}{12}$, and $w$ sends $v$ at least $\frac{2}{9}$. In all cases, $v$ receives at least $\frac{1}{2}$.

Thus we may assume that $\deg(u) = \deg(w) = 4$. We now break into cases depending on the degrees of $a, b, a',$ and $b'$.\\

\noindent
\textbf{Case 1. No vertex in $\{a,b,a',b'\}$ has degree $3$.} In this case, $a,b,a',b'$ all have degree at least four. Thus $\deg_3(u) \leq 2$ and $\deg_3(w) \leq 2$, and so $v$ receives at least $\frac{1}{3}$ from each of $u$ and $w$ in Step 1. Since $v \in V(C)$ and $v(C) \geq 3$, we have furthermore that $v$ sends no charge in Step 2. Thus $v$ receives at least  $\frac{2}{3}$ charge, as desired. \\

\noindent
\textbf{Case 2. Precisely one of $a, b, a',$ and $b'$ has degree $3$.}
Suppose $\deg(a) = 3$. Then $w$ is adjacent to at least two vertices of degree not equal to three, and so $w$ sends at least $\frac{1}{3}$ to $v$ in Step 1. Moreover, $u$ is adjacent to $a'$ with $\deg(a') \neq 3$, and so $u$ sends $v$ at least $\frac{2}{9}$ in Step 1. By Corollary \ref{triangledegrees}, since $a$ is contained in a triangle, it follows that $a$ is isolated in $D_3(G)$. Thus $a$ sends $\frac{1}{18}$ to $u$ in Step 2. Since $a$ is isolated and $\deg(a') \neq 3$, we have that $u$ sends at least $\frac{1}{18(\deg(u)-2)}$ to $v$ in Step 3. As our choice for $a$ was arbitrary but $\deg(u) = \deg(w) = 4$, it follows that during discharging $v$ receives at least
\[ \frac{1}{3} + \frac{2}{9} + \frac{1}{18(4)-36} = \frac{7}{12} \textnormal{\hskip 4mm charge, as desired. }
\]

\noindent
\textbf{Case 3. Either $\deg(a) = \deg(a') = 3$, or $\deg(b) = \deg(b')=3$.}
By symmetry, we may assume without loss of generality that $\deg(a) = \deg(a') = 3$. Then $u$ sends $v$ at least $\frac{1}{6}$ in Step 1. By Corollary \ref{triangledegrees} and Lemma \ref{acyclic}, neither $b$ nor $b'$ has degree $3$, and so $w$ sends $v$ at least $\frac{1}{3}$ in Step 1. By Corollary \ref{triangledegrees}, since $a$ and $a'$ are each contained in a triangle, it follows that both $a$ and $a'$ are isolated in $D_3(G)$. Thus each of $a$ and $a'$ sends $\frac{1}{18}$ to $u$ in Step 2, and so $u$ sends at least $\frac{1}{9(\deg(u)-2)}$ to $v$ in Step 3.

Since $\deg(u) = \deg(w) = 4$ by assumption, it follows that during discharging $v$ receives at least
\[
    \frac{1}{6} + \frac{1}{3} + \frac{1}{9(4-2)} = \frac{5}{9} \textnormal{\hskip 4mm charge, as desired.}
\]

\textbf{Case 4. $\deg(a) = \deg(b') = 3$ or $\deg(b) = \deg(a') = 3$.} By symmetry, we may assume that $\deg(a) = \deg(b') = 3.$ By Corollary \ref{triangledegrees} and Lemma \ref{acyclic}, each of $a$ and $b'$ are isolated in $D_3(G)$, and so each of $u$ and $w$ sends $v$ at least $\frac{2}{9}$ in Step 1. Moreover, $a$ sends $u$ $\frac{1}{18}$ charge in Step 2; similarly, $b'$ sends $w$ $\frac{1}{18}$ charge in Step 2. Thus $v$ receives at least $\frac{1}{18(\deg(u)-2)}$ from $u$ in Step 3, and at least $\frac{1}{18(\deg(w)-2)}$ from $w$ in Step 3. It follows that during discharging $v$ receives at least
\[
\frac{2}{9} + \frac{2}{9} + \frac{2}{18(4-2)} = \frac{1}{2} \textnormal{\hskip 4mm charge, as desired. }
\]
\end{proof}

\begin{claim}\label{p3}
If $C$ is a path of length three, then $\ch_f(C) \geq v(C) \left(\frac{10}{3}\right)$.
\end{claim}
\begin{proof}
Let $C$ be the path $v_1v_2v_3v_4$. Let $u$ be the neighbour of $v_2$ not contained in $C$. By Lemma \ref{identificationlemma} applied to the path $v_{4}v_{3}v_{2}$ where $v_{2}$ is playing the role of $x$ in Lemma \ref{identificationlemma}, either $uv_1 \in E(G)$, or $v_1, v_2$, and $u$ are contained in an $M$-gadget with end $v_2$. By Corollary \ref{triangledegrees}, $uv_1$ is not an edge in $E(G)$, as otherwise $uv_1v_2$ is a triangle containing exactly two vertices of degree $3$. Thus by the structure of $M$-gadgets, there exist two disjoint triangles $T=abca$ and $T'=a'b'c'a'$ such that, up to relabelling, $u$ is adjacent to $a$ in $T$ and $a'$ in $T'$, and $v_1$ is adjacent to  $b$ in $T$ and $b'$ in $T'$.

Next, note that by Lemma \ref{identificationlemma} applied to the path $v_3v_2v_1$ with $v_{1}$ playing the role of $x$ in Lemma \ref{identificationlemma}, either $bb' \in E(G)$, or $v_1$ is the end of an $M$-gadget with $b$ and $b'$. First suppose $bb' \in E(G)$. In this case, note that by Corollary \ref{triangledegrees}, at most one of $a$ and $c$ has degree $3$. Thus $b$ does not send charge to at least two of its neighbours in Step 1. Symmetrically, at most one of $a'$ and $c'$ has degree $3$, and so $b'$ does not send charge to at least two of its neighbours in Step 1. Thus $v_1$ receives at least $\frac{1}{3}$ from each of $b$ and $b'$ in Step 1. By Claim \ref{leaves}, $v_4$ receives at least $\frac{4}{9}$ charge in Step 1. Finally, each of $v_2$ and $v_3$ receive at least $\frac{1}{6}$ by Observation \ref{onesixth}. It follows that 

\begin{align*}
    \ch_f(C) & \geq 4(3) + 2\left(\frac{1}{3} \right) + \frac{4}{9} + 2\left(\frac{1}{6}\right) \\
    &> 4\left(\frac{10}{3}\right),\textnormal{\ as desired. }
\end{align*}

Thus we may assume that $bb'$ is not an edge in $G$. But then by Lemma \ref{identificationlemma} applied to the path $v_3v_2v_1$, we have that $v_1$, $b$, and $b'$ are contained in an $M$-gadget with end $v_1$. By Claim \ref{betterleaves}, $v_1$ thus receives at least $\frac{1}{2}$ charge during discharging. By a perfectly symmetrical argument, $v_4$ receives at least $\frac{1}{2}$ charge during discharging. As above, each of $v_2$ and $v_3$ receive at least $\frac{1}{6}$ by Observation \ref{onesixth}. It follows that

\begin{align*}
    \ch_f(C) & \geq 4(3) + 2\left(\frac{1}{2}\right) + 2\left(\frac{1}{6}\right) \\
    &= 4\left(\frac{10}{3}\right),\textnormal{\ as desired. }
\end{align*}
\aftermath
\end{proof}

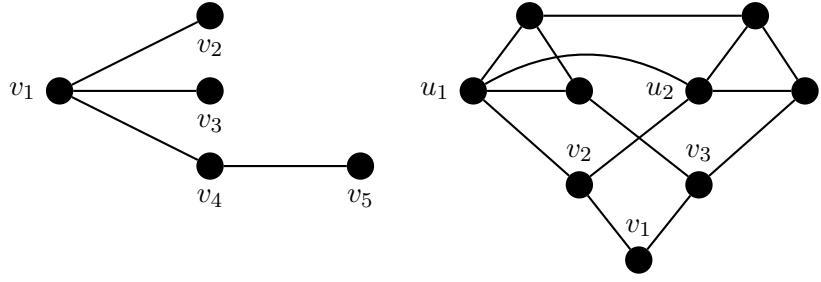
\begin{figure}
    \centering
    \begin{tikzpicture}
    \node[blackvertexv2] at (0,0) (v1) [label = left:$v_{1}$] {};
    \node[blackvertexv2] at (2,1) (v2) [label = below:$v_{2}$] {};
    \node[blackvertexv2] at (2,0) (v3) [label = below:$v_{3}$] {};
    \node[blackvertexv2] at (2,-1) (v4) [label = below:$v_{4}$] {};
    \node[blackvertexv2] at (4,-1) (v5) [label = below:$v_{5}$] {};
    \draw[thick,black] (v1)--(v4)--(v5);
    \draw[thick,black] (v1)--(v2);
    \draw[thick,black] (v1)--(v3);
\begin{scope}[xshift = 5.5cm]

\node[blackvertexv2] at (0,0) (x1) [label = left:$u_{1}$] {};
\node[blackvertexv2] at (.75,1) (x2) {};
\node[blackvertexv2] at (1.41,0) (x3){};
\draw[thick,black] (x1)--(x2)--(x3)--(x1);

\node[blackvertexv2] at (3,0) (x4) [label = left:$u_{2}$] {};
\node[blackvertexv2] at (3.75,1) (x5) {};
\node[blackvertexv2] at (4.41,0) (x6){};
\draw[thick,black] (x4)--(x5)--(x6)--(x4);

\node[blackvertexv2] at (2.2,-2.25) (v1) [label = above:$v_{1}$] {};
\node[blackvertexv2] at (1.41,-1.25) (v2) [label = above:$v_{2}$] {};
\node[blackvertexv2] at (3,-1.25) (v3) [label = above:$v_{3}$] {};
\draw[thick,black] (v1)--(v2);
\draw[thick,black] (v1)--(v3);
\draw[thick,black] (v2)--(x1);
\draw[thick,black] (v2)--(x4);
\draw[thick,black] (v3)--(x3);
\draw[thick,black] (v3)--(x6);
\draw[thick,black] (x2)--(x5);
\draw[thick,black, bend right = 30] (x4) to (x1);

\end{scope}
    \end{tikzpicture}
    \caption{On the left, the graph $C$ in Claim \ref{Y}, and on the right, we have the non-induced $M$-gadget if $u_{1}u_{2} \in E(G)$ which has potential at most $1$.}
    \label{claim8fig}
\end{figure}
\begin{claim}\label{Y}
If $V(C) = \{v_1, v_2, v_3, v_4, v_5\}$ and $E(C) = \{v_1v_2, v_1v_3, v_1v_4, v_4v_5\}$, then $\ch_{f}(C) \geq v(C)\left(\frac{10}{3}\right)$.
\end{claim}
\begin{proof}
Let $u_1, u_2$ be the neighbours of $v_2$ that are not in $C$. Let $w_1, w_2$ be the neighbours of $v_3$ not in $C$. Note by Lemma \ref{identificationlemma} applied to the path $v_{4}v_{1}v_{2}$ with $v_{2}$ playing the role of $x$ from Lemma \ref{identificationlemma}, either $u_1u_2 \in E(G)$ or $u_1, u_2$, and $v_2$ are in an $M$-gadget with end $v_2$. Symmetrically, either $w_1w_2 \in E(G)$ or $w_1, w_2,$ and $v_3$ are in an $M$-gadget with end $v_3$. We will aim to show $u_1u_2 \not \in E(G)$ (and by a symmetrical argument, $w_1w_2 \not \in E(G)$) as otherwise we are done. To see this, suppose not. Then $u_1u_2 \in E(G).$ By Lemma \ref{identificationlemma} applied to the path $v_{5}v_{4}v_{1}$ with $v_{1}$ playing the role of $x$ in Lemma \ref{identificationlemma}, since $v_2v_3 \not \in E(C)$ it follows that $v_1, v_2,$ and $v_3$ are in an $M$-gadget with end $v_1$. It follows that $u_{1}$ and $u_{2}$ are in distinct triangles $T_{1}$ and $T_{2}$, and that $u_{1}u_{2}$ does not belong to this $M$-gadget. But then subgraph induced by the vertices belonging to the $M$-gadget has potential at most $1$, contradicting Lemma \ref{potentiallemma}.
So we may assume $u_1u_2 \not \in E(G)$, and by symmetry $w_1w_2 \not \in E(G)$. By Lemma \ref{identificationlemma}, it follows that each of $v_2$ and $v_3$ is the end of an $M$-gadget with its neighbours outside $C$. By Lemma \ref{betterleaves}, $v_2$ and $v_3$ each receives at least $\frac{1}{2}$ during discharging. By Claim \ref{leaves}, $v_5$ receives at least $\frac{4}{9}$ in Step 1, and by Observation \ref{onesixth}, $v_4$ receives at least $\frac{1}{6}$. It follows that 
\begin{align*}
    \ch_f(C) &\geq 5(3) + 2\left(\frac{1}{2}\right) +\frac{4}{9} + \frac{1}{6} \\
    &> 5\left(\frac{10}{3}\right),\textnormal{\ as desired. }
\end{align*}
\aftermath
\end{proof}

\begin{figure}
    \centering
    \begin{tikzpicture}
    \node[blackvertexv2] at (0,0) (v1) [label = above:$v_{1}$] {};
    \node[blackvertexv2] at (2,0) (v2) [label = above:$v_{2}$] {};
    \node[blackvertexv2] at (4,1) (v5) [label = above:$v_{5}$] {};
    \node[blackvertexv2] at (4,-1) (v6) [label = above:$v_{6}$] {};
    \node[blackvertexv2] at (-2,1) (v3) [label = above:$v_{3}$] {};
    \node[blackvertexv2] at (-2,-1) (v4) [label = above:$v_{4}$] {};
    \draw[thick,black] (v3)--(v1)--(v2)--(v5);
    \draw[thick,black] (v4)--(v1);
    \draw[thick,black] (v6)--(v2);
    \end{tikzpicture}
    \caption{The graph $C$ in Claim \ref{>-<}.}
    \label{fig:my_label}
\end{figure}
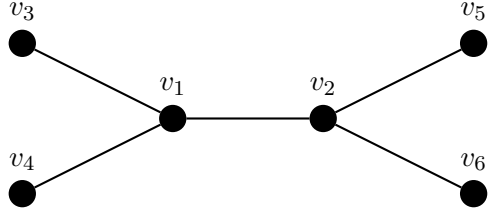
\begin{claim}\label{>-<}
If $V(C) = \{v_1, v_2, v_3, v_4, v_5, v_6\}$ and $E(C) = \{v_{1}v_{2},v_{1}v_{3},v_{1}v_{4},v_{2}v_{5},v_{2}v_{6}\}$, then $\ch_f(C) \geq v(C) \left(\frac{10}{3}\right)$.
\end{claim}
\begin{proof}
Let $u_1, u_2$ be the neighbours of $v_3$ that are not in $C$. Let $w_1, w_2$ be the neighbours of $v_{4}$ not in $C$. Note that by Lemma \ref{identificationlemma} applied to the path $v_{2}v_{1}v_{3}$ with $v_{3}$ playing the role of $x$, either $u_1u_2 \in E(G)$ or $u_1, u_2$, and $v_3$ are in an $M$-gadget with end $v_3$. By symmetry, either $w_1w_2 \in E(G)$ or $w_1, w_2,$ and $v_4$ are in an $M$-gadget with end $v_4$.  We will aim to show $u_1u_2 \not \in E(G)$ (and by a symmetrical argument, $w_1w_2 \not \in E(G)$) as otherwise we are done. To see this, suppose not. Then $u_1u_2 \in E(G).$ By Lemma \ref{identificationlemma} applied to the path $v_5 v_2v_1$ with $v_{1}$ playing the role of $x$ in Lemma \ref{identificationlemma}, since $v_3v_4 \not \in E(C)$ it follows that $v_1, v_3,$ and $v_4$ are in an $M$-gadget with end $v_1$. Similar to as in Claim \ref{Y}, this $M$-gadget is not induced, (as the edge $u_{1}u_{2}$ is not in the $M$-gadget), and hence the subgraph induced by the vertices of the $M$-gadget has potential at most $1$, contradicting Lemma \ref{potentiallemma}.

So we may assume $u_1u_2 \not \in E(G)$, and by symmetry $w_1w_2 \not \in E(G)$. By symmetry, $v_5$ is not contained in a triangle with its neighbours outside $C$, and nor is $v_6$. By Lemma \ref{identificationlemma}, it follows that each of $v_3$, $v_4$, $v_5$, and $v_6$ is the end of an $M$-gadget with its neighbours outside $C$. By Lemma \ref{betterleaves}, $v_3$, $v_4$, $v_5$, and $v_6$ each receive at least $\frac{1}{2}$ during discharging. It follows that 
\begin{align*}
    \ch_f(C) &\geq 6(3) + 4\left(\frac{1}{2}\right)  \\
    &= 6\left(\frac{10}{3}\right),\textnormal{\ as desired. }
\end{align*}
\aftermath
\end{proof}

 We are now equipped to prove Theorem \ref{maintheorem}.

\begin{proof}[Proof of Theorem \ref{maintheorem}]Suppose not. Let $G$ be a vertex-minimum counterexample. By Lemma \ref{ind-p4}, no component of $D_{3}(G)$ contains an induced path of length 4. By Lemma \ref{acyclic}, $D_3(G)$ is acyclic; and by Corollary \ref{nomorethansix}, every component of $D_3(G)$ has at most six vertices. Since every vertex in $D_{3}(G)$ has degree $3$, it follows that every component of $D_{3}(G)$ falls into one of the claims from Claims \ref{isolated} through \ref{p3}. Therefore it follows that that $\textnormal{KY}(G) \leq 0$. Moreover, by Lemma \ref{bigcomp}, $D_3(G)$ contains a component with at least three vertices. We break into cases depending on the structure of the components in $D_3(G)$. \\

\noindent
\textbf{Case 1: $D_3(G)$ contains a component $C$ with $v(C) \geq 3$ such that $C$ is any of the graphs described in Claims \ref{p3} through \ref{>-<}.} \\
In this case, $C$ contains a path $P$ of length two ending with a non-leaf vertex, $v$. Thus, by applying Lemma \ref{identificationlemma} to $P$ with $v$ playing the role of $x$, we get that either $x$ is contained in a triangle with its neighbours not on $P$, or that $x$ is the end of an $M$-gadget, $H$. By Corollary \ref{triangledegrees} and Lemma \ref{acyclic}, $x$ is not contained in a triangle with another vertex in $C$, and so it follows that $x$ is the end of an $M$-gadget, $H$. But $T^3(H) = 2$, and so $T^3(G) \geq 2.$ It follows that $p(G) \leq \textnormal{KY}(G) -2 \leq -2$, and so $G$ is not a counterexample.
\\
\vskip 4mm
\noindent
\textbf{Case 2: $D_3(G)$ contains no components described in Case 1, but contains a star $H$ with four vertices.} \\
Let $V(H) = \{v_1, v_2, v_3,v_4\}$ and $E(H) = \{v_4v_1, v_4v_2, v_4v_3\}$. By applying Lemma \ref{identificationlemma} to each of the paths $v_1v_4v_3,\ v_1v_4v_2$, and $v_2v_4v_3$, (where in each case the last vertex of the path is playing the role of $x$ in Lemma \ref{identificationlemma}) we see that $v_1,v_2,$ and $v_3$ are each the ends of $M$-gadgets, or that they are contained in triangles with their neighbours outside $H$. As in the above case, if $G$ contains an $M$-gadget, then $p(G) \leq -2$, and $G$ is not a counterexample. Thus we may assume neither $v_1, v_2$ nor $v_3$ is the end of an $M$-gadget. Let $T_1$, $T_2$, and $T_3$ be the triangles containing $v_1, v_2$ and $v_3$, respectively. By Corollary \ref{triangledegrees} and Lemma \ref{acyclic}, these triangles are distinct. If $T^3(T_1 \cup T_2 \cup T_3) \geq 2$, then $p(G) \leq -2$ and $G$ is not a counterexample. Thus we may assume the triangles share some vertices. Note that no two triangles in $\{T_1, T_2, T_3\}$ share two vertices, since $G$ contains no $K_4-e$ subgraph by Lemma \ref{k4-e}. There are thus two cases to consider: either there exists a single vertex contained in all three triangles and the triangles are otherwise disjoint, or $V(T_1) \cap V(T_2) \cap V(T_3) = \emptyset$ and every pair of triangles shares a distinct vertex. If every pair of triangles shares a distinct vertex, then since $H \cup T_1 \cup T_2 \cup T_3$ is $4$-critical (we leave this verification for the reader), $G = H \cup T_1 \cup T_2 \cup T_3$. But then $p(G) = 5(7)-3(12)-1 = -2$, and so $G$ is not a counterexample. Thus we may assume that $T_1 \cap T_2 \cap T_3 = \{u\}$, for some vertex $u \in G$. Note then that $\deg(u) \geq 6$. Moreover, $u$ is adjacent to at least three vertices that are adjacent to $H$ but not in $H$; thus $u$ neighbours at least three vertices of degree greater than three. It follows that $u$ sends at least $\frac{8}{9}$ charge to each of $v_1, v_2$, and $v_3$ in Step 1 of the discharging process. Thus $\ch_f(H) \geq 3\left(\frac{8}{9}\right) + 4(3) = \frac{44}{3} = 4\left(\frac{10}{3}\right) + \frac{4}{3}$. Note that every other component $C$ in $D_3(G)$ has final charge at least $v(C) \left(\frac{10}{3}\right)$ and every vertex of degree at least four has final charge at least $\frac{10}{3}$. Thus the sum of the charges is at least $v(G)\left(\frac{10}{3}\right) + \frac{4}{3}$. Thus $\textnormal{KY}(G) \leq -2$, and so $p(G) \leq -2$ (and in fact $p(G) \leq -3$ as $G$ contains at least one vertex-disjoint triangle). This contradicts the fact that $G$ is a counterexample.
\\
\begin{figure}
    \centering
    \begin{tikzpicture}
    \node[blackvertexv2] at (0,0) (v1) [label = below:$v_{1}$] {};
    \node[blackvertexv2] at (-1,1) (v2)[label = left:$v_{2}$] {};
    \node[blackvertexv2] at (0,1) (v3) [label = left:$v_{3}$] {};
    \node[blackvertexv2] at (1,1) (v4) [label = left:$v_{4}$] {};
    \node[blackvertexv2] at (-1,2) (u1) {};
    \node[blackvertexv2] at (0,2) (u2) {};
    \node[blackvertexv2] at (1,2) (u3) {};
    \draw[thick,black] (v1)--(v2);
    \draw[thick,black] (v1)--(v3);
    \draw[thick,black] (v1)--(v4);
    \draw[thick,black] (v2)--(u2);
    \draw[thick,black] (v2)--(u3);
    \draw[thick,black] (v3)--(u1);
    \draw[thick,black] (v3)--(u3);
    \draw[thick,black] (v4)--(u1);
    \draw[thick,black] (v4)--(u2);
    \draw[thick,black] (u1)--(u2)--(u3);
    \draw[thick,black,bend left = 30] (u1) to (u3);
    \begin{scope}[xshift=5cm]
    \node[blackvertexv2] at (0,0) (v1) [label = below:$v_{1}$] {};
    \node[blackvertexv2] at (-1,1) (v2)[label = left:$v_{2}$] {};
    \node[blackvertexv2] at (0,1) (v3) [label = left:$v_{3}$] {};
    \node[blackvertexv2] at (1,1) (v4) [label = left:$v_{4}$] {};
    \node[blackvertexv2] at (-1,2) (u1) {};
    \node[blackvertexv2] at (0,2) (u2) {};
    \node[blackvertexv2] at (1,2) (u3) {};
    \draw[thick,black] (v1)--(v2);
    \draw[thick,black] (v1)--(v3);
    \draw[thick,black] (v1)--(v4);
    \draw[thick,black] (v2)--(u1);
    \draw[thick,black] (v3)--(u2);
    \draw[thick,black] (v4)--(u3);
    \node[blackvertexv2] at (0,3) (u) [label = left:$u$] {};
    \draw[thick,black] (u)--(u1);
    \draw[thick,black] (u)--(u2);
    \draw[thick,black] (u)--(u3);
    \draw[thick,black] (u) to (v2);
    \draw[thick,black,bend right= 20] (u) to (v3);
    \draw[thick,black,] (u) to (v4);
    \end{scope}
    \end{tikzpicture}
    \caption{On the left, we have the $4$-critical graph $H \cup T_{1} \cup T_{2} \cup T_{3}$ in Case 2, and on the right we have the other possibility, where we have a vertex $u$ which has degree at least $6$.} 
    \label{Case2}
\end{figure}
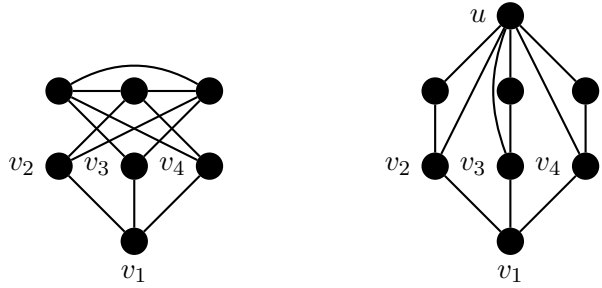

\vskip 4mm
\noindent
\textbf{Case 3: $D_3(G)$ contains no components described in Cases 1 or 2, but contains a path $H$ of length 2.} \\
Let $H = v_1v_2v_3$. Note that by Claim \ref{p2}, the final charge of $H$ is strictly greater than $v(H) \left( \frac{10}{3}\right)$. Moreover, every other component $C$ of $D_3(G)$ has final charge at least $v(C) \left(\frac{10}{3}\right)$ and every vertex of degree at least four has final charge at least $\frac{10}{3}$. It follows that the sum of the charges is strictly greater than $v(G) \left(\frac{10}{3}\right)$, and so $\textnormal{KY}(G) < 0$. Since $\textnormal{KY}(G)$ is integral, it follows that $\textnormal{KY}(G) \leq -1$. By Lemma \ref{identificationlemma} applied to the path $v_1v_2v_3$, we find that either $G$ contains an $M$-gadget or a triangle; and so since $M$-gadgets contain triangles, it follows that $T^3(G) \geq 1$. Thus $p(G) = \textnormal{KY}(G)-T^3(G) \leq -2$, and so $G$ is not a counterexample.
\end{proof}

\begin{ack}
The authors would like to thank the referees; in particular, one referee's careful reading  drastically improved the quality of the paper and fixed numerous errors. The authors would also like to thank Alexandr Kostochka and Jingwei Xu for pointing out errors in an earlier version of this paper, Daniel Cranston for pointing out some key observations about $k$-Ore graphs as well as James Davies for sharing his construction with us (see Figure \ref{jamesgraph}). Lastly, the first author would like to thank Sophie Spirkl and Gary MacGillivray for comments on a version of this paper which appeared in the author's PhD thesis.
\end{ack}
\bibliographystyle{plain}
\bibliography{trianglefreebib}

\end{document}